%% file: main.tex
\documentclass[a4paper,UKenglish,cleveref,thm-restate,hidelipics]{lipics-v2021}

\usepackage{tikz}
\usepackage{tikzconfig}
\usepackage{tikzmacros}

\usepackage{bm}
\usepackage{todonotes}
\usepackage{subcaption}
\usepackage{lmodern}

\usepackage{algorithm}
\usepackage[noend]{algpseudocode}
\usepackage[quotation,notion, electronic]{knowledge}
\input{kl-configure}

\input{popmdps.kl}

\newcommand{\set}[1]{\{#1\}}

\nolinenumbers

\pdfoutput=1 
\hideLIPIcs

\usepackage{xspace}
\usepackage{macros}

\title{Optimally Controlling a Random Population}

\author{Hugo Gimbert}{CNRS, LaBRI, Université de Bordeaux, France}{hugo.gimbert@labri.fr}{}{}{}

\author{Corto Mascle}{Max Planck Institute for Software Systems, Kaiserslautern, Germany}{cmascle@mpi-sws.org}{}{}{}

\author{Patrick Totzke}{University of Liverpool, UK}{totzke@liverpool.ac.uk}{}{}{}

\authorrunning{H.~Gimbert, C.~Mascle, P.~Totzke} 

\Copyright{ CC-BY;  http://creativecommons.org/licenses/by/3.0/} 



\category{} 

\relatedversion{} 

\theoremstyle{plain}

\begin{document}

\maketitle

\begin{abstract}
	\input{abstract}
\end{abstract}

%



\input{content}

\end{document}

%% file: kl-configure.tex
\definecolor{Blue Sapphire}{HTML}{003050} 
\definecolor{Gamboge}{HTML}{ee9b00}
\definecolor{Ruby Red}{HTML}{9b2226}

\IfKnowledgePaperModeTF{
}{
	\knowledgestyle{intro notion}{color={Ruby Red}, emphasize}
	\knowledgestyle{notion}{color={Blue Sapphire}}
	\hypersetup{
		colorlinks=true,
		breaklinks=true,
		linkcolor={Blue Sapphire}, 
		citecolor={Blue Sapphire}, 
		filecolor={Blue Sapphire}, 
		urlcolor={Blue Sapphire},
	}
}
\IfKnowledgeCompositionModeTF{
	\knowledgeconfigure{anchor point color={Ruby Red}, anchor point shape=corner}
	\knowledgestyle{intro unknown}{color={Gamboge}, emphasize}
	\knowledgestyle{intro unknown cont}{color={Gamboge}, emphasize}
	\knowledgestyle{kl unknown}{color={Gamboge}}
	\knowledgestyle{kl unknown cont}{color={Gamboge}}
}{
}

%% file: abstract.tex
The population control problem is a parameterised problem
where a controller sends messages to a whole population of identical finite-state agents,
aiming to eventually move them all into a target state.
The decision problem asks whether this can be achieved
for \emph{arbitrarily large} finite populations.
We focus on the \emph{randomised} version of this problem,
where every agent 
is a copy of the same finite Markov Decision Process
and non-determinism in the {global} action chosen by the controller is resolved independently and uniformly at random.
Colcombet, Fijalkow and Ohlmann \cite{ColcombetFO21} showed that this problem is decidable, but without any complexity upper bound.

We show that the random population control problem is in fact \exptime-complete.


%
%



%% file: content.tex
\section{Introduction}

\input{sec.intro}
\section{Preliminaries}
\label{sec:prelims}
\input{sec.prelims}

\section{Symbolic Configurations}%
\label{sec:symbolic}
 \input{sec.symbolic.tex}

 
 \section{Solving the {Random Population Control Problem}}
\input{sec.solution-main}


\section{Theorem 6: Small Winning Arenas}
\label{sec:proof-main-thm}
\input{sec.winning-regions}

\section{Theorem 7: The Dynamic Flow Problem}
\label{sec:seq-flows}
\input{sec.flow-semigroup}

\section{Lower bound}
\label{sec:lower-bound}

\input{lower-bound-main}

\section{Conclusion}
\input{sec.conclusion}

\bibliography{biblio.cleaned.bib}

\newpage
\appendix

\section{Missing Proofs in Section~\ref{sec:prelims}}
\label{app:winning-arenas}
\input{app-win-arenas}

\section{Correctness of Algorithm 1}
\label{app:alg}
\input{proof-alg1-correct}
\section{Missing Proofs for Section~\ref{sec:proof-main-thm}}
\subsection{Proof of Lemma~\ref{lem:onlyone}}
\label{app:funneling}
\input{proof-funneling-herd}

\subsection{Proof of Lemma~\ref{lem:isolation}}
\label{app:isolation}
\input{proof-isolation}

\subsection{Proof of Theorem~\ref{thm:cortosconj}}
\label{app:induction}
\input{proof-induction}

\section{Missing Proofs for Section~\ref{sec:seq-flows}}
\label{app:reduction-to-SFP}
\input{reduction-to-SFP}

\section{The EXPTIME Lower Bound}\label{sec:app-lower-bound}
\input{lower-bound}

\section{Hard Cases}
\input{sec.examples.tex}
\subsection{A Butterfly}%
\label{sec:butterfly}
\input{sec.example-butterfly}
\subsection{A Chain of Bottlenecks}
\label{app:example-chain-of-bottlenecks}
\input{sec.example-chain}

\subsection{A Double-exponential Cut-off}
\label{app:example-2exp-cutoff}
\input{sec.example-2exp-cutoff}

%% file: sec.intro.tex
We place $N$ tokens on the initial state of a non-deterministic finite automaton and update their positions in rounds, in each of which
a \Laetitia selects a letter $a$ and then every token moves along an $a$-labelled edge out of its current state.
The goal for the \Laetitia is to eventually synchronise all tokens, gathering them into accepting states at the same time.
Naturally, increasing $N$ can only make this task more difficult.
The decision question we study 
asks if, for a given automaton,
can the \Laetitia succeed for all values of $N$?

This population control problem{} has been studied both in the adversarial and random settings, which differ in how agents resolve choices and what guarantees the \Laetitia is after.
In the adversarial setting, an antagonistic environment resolves all agents' choices, trying to frustrate \Laetitia and avoid synchronisation.
Bertrand et al.~\cite{BertrandDGG17,lmcs:5647} showed that the adversarial population control problem is decidable and \exptime-complete.

In the random setting, all agents' choices are made uniformly at random and the \Laetitia aims to synchronise the agents almost surely, with probability one.

\begin{example}
	\label{ex:force-one}
	\begin{figure}[ht]
		\centering
		\input{Figures/force-initial.tikz}
	\caption{An automaton for which \Laetitia wins against a random environment but not an adversarial one. 
	The initial and target states are $s_0$ and $s_4$, respectively.
Not shown, but implicitly present is a sink state, to which all states move on actions not shown
($b$ from $s_3$ and $a$ from $s_2$).}
		\label{fig-2}
		\label{fig:force-one}
	\end{figure}

	\Laetitia loses the adversarial population control game from \cref{fig-2}. For example, the opponent can win by keeping all tokens in the initial state in every round.

	However, \Laetitia does have a winning strategy for the corresponding random population control game:
	Regardless of the size $N$ of the population, she plays the action $a$ until exactly one token is in $s_1$,
	which happens eventually with probability $1$. She then plays $b$
	followed by either $a$ (if the isolated token moved to $s_3$) or $b$ (if the token moved to $s_2$) to send that token to the final state.
She can repeat this procedure until all tokens are in the target.
	Note that if instead she plays $b$ while more than one token is in $s_1$ then some of them may spread simultaneously to $s_2$ and $s_3$ and she is stuck: whatever action she plays, at least one token will move to the (invisible) sink state and cannot be recovered.
\end{example}

The random population control problem is \exptime-hard \cite{DBLP:journals/corr/abs-1909-06420} and decidable \cite{ColcombetFO21}.
Colcombet et al.'s decision procedure is based on two ingredients.
First, winning regions are downward-closed with respect to the natural product order on $\NN^S$
and can therefore be finitely represented and manipulated as a union of ideals. 
The second ingredient is solving the \emph{sequential flow problem}, which asks whether an unbounded number of tokens can be moved from one set of states to another, while respecting capacity constraints on the number of tokens occupying each transition. They show that one can compute the full corresponding pre-set of a given downward-closed set of target configurations,
by a reduction to the boundedness problem of {nested distance-desert automata}~\cite{DBLP:journals/ita/Kirsten05}.
Ultimately, this allows us to compute a representation of the winning region by iterative refinement.
The main limitation of this approach is that
the bounds on the representation of intermediate sets,
and termination of the resulting algorithm 
rely on well-quasi-orders and therefore only provide a non-elementary upper bound. 

\subparagraph*{Our Contribution}
%
%
We show that the random population control problem{} is \exptime-complete. 
As in \cite{ColcombetFO21}, our upper bound relies on symbolic representations.
However, instead of computing the whole winning region,
we show that one can witness positive instances
already with a potentially smaller, more compactly representable subset
(Theorem~\ref{thm:cortosconj}).
Computing this in exponential time requires solving
what we call the \problemm{},
which asks whether we can transfer arbitrarily many tokens,
with positive probability,
from an initial configuration to the set of final configurations
while staying in a given set of safe configurations.
We also include the original, previously unpublished lower bound \cite{DBLP:journals/corr/abs-1909-06420} (Theorem~\ref{thm:exp-hard}).

\medskip

\subparagraph*{Related work}
Models for biological, chemical, or computational systems 
with large crowds of simple finite-state components
include 
Petri nets~\cite{Petri62}, 
population protocols~\cite{AngluinADFP06},
or chemical reaction networks~\cite{SoloveichikCWB08}.
Parameterised verification ~\cite{SistlaG87,EmersonN95}
refers to a line of research that 
aims to verify distributed systems of arbitrary size.
Specifically, our work follows the line of population control problems, where the goal is to synchronise a large population of agents~\cite{UhlendorfMDCFBHB15,lmcs:5647, ColcombetFO21}. While our model considers a discrete set of agents, other works over-approximate the population as a probability distribution~\cite{AkshayGV18, Doyen23}.
Another model close to our setting is explorable automata~\cite{HazardIK24}, which generalise both (adversarial) population control problems
and {history-deterministic} automata on infinite words.

%% file: Figures/force-initial.tikz
\begin{tikzpicture}[
		node distance=0.25cm and 1.5cm,
	]

	\node[state, initial] (Z)  {$s_0$};
	\node[state, right= of Z] (A)  {$s_1$};
	\node[state, below right= of A] (B)  {$s_3$};
	\node[state, above right= of A] (C)  {$s_2$};
	\node[state,accepting, above right= of B] (D) {$s_4$};

	\draw[->, bend right=20]
	(A) edge node[above] {$a$} (Z)
	(Z) edge node[below] {$a$} (A)

	(A) edge node[swap] {$b$} (B)
	(B) edge node[swap] {$b$} (D)
	;

	\draw[->, bend left=20]
	(A) edge node[] {$b$} (C)
	(C) edge node[] {$a$} (D)
	;

	\draw[->, thick, loop right] (D) edge node {$a,b$} 
	(D);
	\draw[->, thick, loop above] (A) edge node {$a$} (A);
	\draw[->, thick, loop above] (Z) edge node {$a,b$} 
	(Z);

\end{tikzpicture}

%% file: sec.prelims.tex




We assume familiarity with automata theory \cite{rozenberg1997handbook} and Markov Decision Processes \cite{Puterman:book}, and proceed to recall some necessary notations.
%


\newcommand{\arena}{\Gamma}

\newcommand{\config}{\Gamma}
%

\subparagraph*{Markov decision processes.}
\AP A ""Markov decision process"" (MDP) $\MM = (\states, \act, \prob)$
consists of a set $\states$ of states, $\act$ a set of actions, and a transition function $\prob : \states\times \act \to \Dist(\states)$,
where 
$\intro*\Dist(S)$ denotes the set of probability distributions over $S$.

We consider almost-sure reachability objectives, given by a set $F\subseteq \states$ of target states that \Laetitia aims to visit.
\AP A ""strategy"" is a function $\sigma : \states \to \Dist(\act)$.
Fixing such a strategy $\sigma$ and an initial state $s\in\states$
results in a Markov Chain with probability measure $\mathbb{P}_{\sigma,s}$ (see~\cite{Puterman:book} for details).
We call $\sigma$ ""winning@@strat"" from state $s$ if $F$ is reached $\mathbb{P}_{\sigma,s}$-almost surely.

\AP The ""winning region"" is the subset $W \subseteq \states$ of states from which a "winning strategy" exists.


\subparagraph*{Random walks in winning regions.}

In an MDP with finitely many states, if there is an almost surely winning strategy for \Laetitia, then there is a canonical one: play at random any action that guarantees staying in the winning region.
This is formalized using \emph{arenas} and \emph{safe random walks}, as follows.

%
{

\AP 
We call an MDP 
""simple"" if for all $s \in \states$, the set of states reachable from $s$ is finite.
A\,""commit"" is an element of $\states\x\act$. An ""arena"" is a set $\arena \subseteq \states \times \act$ of commits such that for all $(s,a) \in \arena$
and $t \in \states$, if $\prob(s,a)(t) > 0$ then there exists $b \in \act$
such that $(t,b) \in \arena$.
For brevity we will write $s \in \arena$ instead of $\exists b \in \act, (s,b) \in \arena$.
\AP 
A ""path"" in "arena" $\arena$ is a sequence $s_0 \xrightarrow{a_1} s_1 \xrightarrow{a_2} \ldots \xrightarrow{a_k}s_k \in S(\act S)^*$ such that $(s_{j-1},a_j) \in \arena$ and $\prob(s_{j-1},a_j)(s_j)> 0$ for all $j>0$.
A ""safe random walk"" in the "arena" $\arena$ is a strategy $\sigma$ defined such that $\sigma(s)$ is the uniform distribution on $\{a\in\act \mid (s,a) \in \arena\}$.
An "arena" is ""winning@@arena"" (with respect to $F$) if from every $s \in \arena$ there is a path in $\arena$ to $F$. This is closely linked with "winning region"s, as shown in the next lemma. A proof is in Appendix~\ref{app:winning-arenas}.
}

\begin{restatable}[Winning arenas]{lemma}{WinningArenas}
\label{lem:winningarena}
	Given a "simple" MDP and a set of target states $F$,
	\begin{enumerate}
		\item A union of "winning arenas" is a "winning arena", and the "winning region" is the projection onto $\states$ of the largest "winning arena".
		
		\item In a "winning arena", a "safe random walk" is a "winning strategy" from every state.
	\end{enumerate}
\end{restatable}



\subparagraph*{Random Populations.}
We consider \emph{populations} of agents (or \emph{tokens}), described by a common finite MDP $\?M=(\states, \act, \prob)$.
%
Our arguments require tracking the trajectory of selected (sets of) tokens along paths 
and the following definitions will be convenient for this purpose.

Write $T_\infty$ for the countably infinite set of tokens.
For any finite subset $T\subseteq T_\infty$,
let
$\?M^{T)}=(\states^{T},\act,\prob^{T})$
denote the MDP
whose states (called \emph{configurations}) are functions mapping each token of $T$ to a state of $\states$,
and $\prob$ is naturally lifted to $\states^{T}$:
$\prob^{T}(\gamma,a)(\gamma')=\prod_{t \in T}\prob(\gamma(t),a)(\gamma'(t))$
for all $\gamma, \gamma' \in \states^{T}$ and $a\in\act$.
For a subset $R\subseteq\states$,
$R^{T}$ contains those configurations where all tokens are mapped to states in $R$. In particular,
$\set{s}^T$ contains only the configuration mapping all tokens to $s\in S$,
and $F^{T}$ is the set of final configurations, mapping all tokens to some state of $F$.
%
\AP We study the following decision problem,
which asks if for every finite population $T$,
\Laetitia can almost surely ensure that all tokens simultaneously\footnote{Requiring all tokens to visit a final state \emph{simultaneously} or possibly at different times is irrelevant for our purposes:
	The two variants are logspace inter-reducible (see \cref{lem:simultaneously} in Appendix~\ref{app:winning-arenas}).
}
reach a final state.
\begin{center}
	\fbox{
		\begin{minipage}{38em}
			{\bf The \intro*\problem}\\
			\textbf{Given:}
			an MDP $\?M=(\states, \act, \prob)$, initial state $i\in\states$ and target set $F\subseteq\states$.\\
			\textbf{Question:}
			Is there a "winning strategy" to reach $F^{T}$ from $\set{i}^{T}$, for all finite $T \subseteq T_\infty$?
		\end{minipage}
	}
\end{center}

We define 
$\?M^{*}=\bigcup_{T \subseteq_f T_\infty} \?M^{T}$ as the disjoint union of all finite $\?M^{T}$.
%
Notice that, while the state space $\states^{*} =\bigcup_{T \subseteq_f T_\infty}\states^{T}$ is infinite, each $\states^{T}$ is finite and therefore $\?M^{*}$ is "simple".

According to \cref{lem:winningarena}, a given MDP $\?M$ is a positive instance if, and only if,
there is a winning arena $A \subseteq \states^{*}\x\act$ with
$\set{i}^{T}\subseteq A$ for all $T\subseteq_f T_{\infty}$.
Our main decidability result relies on identifying such a witness $A$.
By \cref{lem:winningarena}, the existence of such a witness $A$ does not depend on
the precise probabilities in $\prob$.
We will omit them in all examples
and just assume without loss of generality that all probability distributions are uniform.

%
\color{black}

%% file: sec.symbolic.tex
%

%
%

We first observe that the maximal "winning arena" $W$ of $\?M^{(*)}$ must be closed under renaming of tokens.
That is, for any bijective $\tau:T\to T'$ and configurations
$\gamma \in \states^{T}$ and $\gamma' \in \states^{T'}$,
if $\gamma = \gamma'\circ \tau$ then either both
$\gamma$ and $\gamma'$,
or neither is in $W$.
Furthermore, $W$ must be closed under removal of tokens:
if $W$ contains $\gamma\in\states^{T}$ and $T'\subseteq T$ then $W$ must also contain
the restriction of $\gamma$ to $T'$. 
Both these can be seen by simple transfers of winning strategies.

\newcommand{\Config}{\states^{T}}
The winning region can therefore conveniently be represented as a downward-closed set of $\abs{\states}$-dimensional vectors of integers, each of which just reflects how many tokens occupy each state.
Given a configuration $\gamma \in \Config$, we write $\intro*\counttokens{\gamma} $ for the vector of $\NN^S$ such that $\counttokens{\gamma}(s)$ is the number of tokens in $s$ in $\gamma$, for all $s \in S$.

Take
\(
\Nb = \{0 < 1<2<3<\ldots< \omega\}
\),
the non-negative integers with natural ordering extended by a maximal element $\omega$.
This ordering is lifted pointwise to
$\Nb^\states$ 
, whereby
$\vectv \leq \vectv'$ if and only if $\vectv(s) \leq \vectv'(s)$ for all $s \in \states$,
and $(\vectv,a) \leq (\vectv',a')$ if and only if $\vectv \leq \vectv'$ and $a=a'$.

\begin{definition}[Symbolic representations]
	\label{def:symbolic}
\AP A ""symbolic configuration"" is a vector $\vectv \in \Nb^S$. 
Given a "symbolic configuration", the ""ideal""  $\intro*\ideal{\vectv}$ is the set of all configurations $\gamma$ such that
$
\counttokens{\gamma} \leq \vectv.
$
For a set of "symbolic configurations" $\vectV$, we write $\ideal{\vectV}$ for the set $\bigcup_{\vectv \in \vectV} \ideal{\vectv}$.

Similarly, a ""symbolic commit"" $(\vectv, a)$ is a symbolic configuration $\vectv \in \Nb^S$ together with an action $a \in \Sigma$. 
The ""commit ideal"" $\ideal{(\vectv, a)}$ is the set $\set{(\gamma,a) \mid \gamma \in \ideal{\vectv}}$.
\end{definition}

We recall that the product ordering on commit ideals is a well-quasi-order. Every downward-closed set $A\subseteq\Nb^S\x\act$, including our maximal "winning arena", can therefore be written as a union $A=\bigcup_{0<i<m}\ideal{(\vec{v}_i,a_i)}$ of finitely many incomparable ideals \cite{schmitz:tel-01663266}.
We call such a finite union of commit ideals a ""population arena"", or simply an arena when clear from the context.
Throughout this work we will make use of the notion of \emph{$K$-definability}, that bounds the maximal finite constants in such ideal representations of population arenas.

\begin{definition}[$K$-definability]
	\label{def:definability}
	Let $K\in\NN$.
	A set $\arena$ of commits
	is ""$K$-definable"" if it is a finite union of "commit ideals"  of the form $(\ideal{\vectv},a)$ with $\vectv \in \set{0, \ldots, K, \omega}^S$.
The maximal $K$-definable subset of $\arena$
is denoted by
\(
\intro*\BoundK{\arena}{K}.
\)
\end{definition}


%
%
\begin{example}
	In the MDP from \cref{ex:force-one}, the maximal "winning arena" is 
	\[
		W = \ideal{((\omega, \omega, \omega, 0, \omega), a)}  
		~\cup~ \ideal{((\omega, 1, 0, \omega, \omega), b)}	      
	\]
	with states enumerated as $(s_0, \dots , s_4)$.
	A "safe random walk" in $W$ instructs to
	play, with positive probability, the action $a$ in any configuration where no tokens are on $s_3$.
	Further, to play $b$ with positive probability if at most one token is on $s_1$ and none on $s_2$.
	This set is clearly "$1$-definable" and therefore, $W = \BoundK{W}{1}$.  
	Its maximal "$0$-definable" subset is
	\[
		\BoundK{W}{0} = \ideal{((\omega, \omega, \omega, 0, \omega), a)}
		~\cup~ \ideal{((\omega, 0, 0, \omega, \omega), b)}
	\]
	While $\BoundK{W}{0}$ is still an "arena" (it is a forward invariant), it is not "winning@arena" because
	there is no path in it from initial to target configurations. This is because the action $b$ that takes tokens to $s_2$ or $s_3$ is never played when there are tokens on $s_1$.
\end{example}


%% file: sec.solution-main.tex
We present our main result: the \problem\ is solvable in exponential time.
According to \cref{lem:winningarena}, there is a maximal "winning arena" containing all configurations from which there is a "winning strategy". 
As a consequence, there is a "winning strategy" from all initial configurations if and only if there is a "winning arena" which contains them all.

Thus, all positive instances are witnessed by a set of commits that is
1) an "arena" (forward-closed),
2) "winning@@arena" (every included configuration can reach a target configuration without leaving this set),
and finally, 3) includes all initial configurations.

We will symbolically represent candidate sets as per \cref{def:symbolic}, as finite unions of "commit ideals" $Y=\bigcup_{0<i<m}\ideal{(\vec{v}_i, a_i)}$.
This makes it straightforward to check conditions 1) and 3) syntactically.
The two main remaining ingredients for our approach are \cref{thm:cortosconj}, that positive instances admit polynomially-definable winning arenas; and \cref{thm:pspace-path-problem}, that one can check condition 2) in exponential time (and even polynomial space).
We formally state these theorems below and derive an algorithm to solve the \problem{} in exponential time.
The following sections contain proofs.

We will symbolically represent candidate winning sets as per \cref{def:symbolic}, as finite unions of "commit ideals" $Y=\bigcup_{0<i<m}\ideal{(\vec{v}_i, a_i)}$.

In what follows, fix an MDP $\?M=(\states, \act, \prob)$, initial state $i\in\states$ and target set $F\subseteq\states$.
Let $W$ denote the "winning arena" in $\?M^{(*)}$, and recall that for every $K\in \NN$, 
\(
\BoundK{W}{K}
\)
is its maximal "$K$-definable" subset.
Let $\textbf{i},\vec{f}\in\Nb^\states$ 
be the "symbolic configuration"s mapping the state $i$ (resp. all states in $F$) to $\omega$ and other states to $0$. 
Then in particular, $\ideal{\vec{f}}$ is the set of configurations whose tokens are all in $F$.  
Observe that, since all those configurations are in the "winning region" by definition, we have $\ideal{\textbf{f}} \subseteq \BoundK{W}{0}$
and
the answer to the \problem\ is positive if and only if $\BoundK{W}{0}$ contains all configurations with all tokens on the initial state, i.e., $\ideal{\textbf{i}} \subseteq \BoundK{W}{0}$.


\begin{restatable}{theorem}{MainThm}
	\label{thm:cortosconj} 
	Let $W$ be the maximal "winning arena".
	There is a "winning arena"
	$Y \subseteq W$
	with
	\begin{itemize}
		\item
		$Y$ contains $\BoundK{W}{0}$; and
		\item
		$Y$ is "$|S|$-definable".
	\end{itemize}
\end{restatable}


\begin{remark*}
	Theorem~\ref{thm:cortosconj} implies that,
	if we can control arbitrarily many tokens then we can do so 
	with a "safe random walk" in some "$\card{\states}$-definable" "arena".
	This has no immediate consequences for the shape of the 
	"winning region" $W$.
	In fact, $W$
	may not be "$|S|$-definable" and its "ideal" representation may require doubly exponentially large constants
	(cf.~Section~\ref{app:example-2exp-cutoff}).
	Therefore, in general, $Y$ is strictly contained in $W$.
	Moreover, despite there existing memoryless and deterministic strategies
	from every configuration in the "winning region" \cite{Ornstein:AMS1969},
	the "$\card{\states}$-definable" "winning strategies" guaranteed by Theorem~\ref{thm:cortosconj} may still require randomization (cf.~Section~\ref{sec:butterfly}).
\end{remark*}

To verify that a candidate arena 
is "winning@arena", we need to be able to check that all configurations in it can reach a final one while staying in the arena.
We show (in \cref{sec:seq-flows}), how to
solve the following problem\footnote{
	This is a slightly different presentation of the ``sequential flow problem'' in \cite{ColcombetFO21}. They presented a solution in exponential space, which was recently improved to polynomial space
	\cite{GimbertMT25}.
}.
\AP
\begin{center}
	\fbox{
		\begin{minipage}{\textwidth-1em}
		{\bf The \intro*\problemm}\\
			\textbf{Given:}
			an MDP
			$\?M=(\states, \act, \prob)$,
			a finite set of "symbolic commits" $\vectV$,
			a "symbolic configuration" $\vectv_0 \in \vectV$ and a set of "symbolic configurations" $\vectF$.
			
			\textbf{Question:}
			Does every
			configuration in $\ideal{\vectv_0}$,
			admit a path to $\ideal{\vectF}$
			inside
			$\ideal{\vectV}$?
		\end{minipage}
	}
\end{center}

\begin{theorem}
	\label{thm:pspace-path-problem}
	The \problemm{} can be solved in polynomial space in $\log(K)$ and $|S|$, where $K$ is the largest integer appearing in the input.
\end{theorem}

Together, \cref{thm:cortosconj,thm:pspace-path-problem} suggest a simple algorithm to compute a suitable "winning@arena" arena $Y$
by refining a candidate set $\vectV$ of "symbolic commits" until the corresponding union of "ideals" is a "winning arena".
Recall that given a set of commits $\arena$ we use the notation $\gamma \in \arena$ for $\exists b \in \Sigma, (\gamma,b) \in \arena$. 
Similarly, given a set of configurations $C$ and a set of commits $\arena$, we write $C \subseteq \arena$ instead of $\forall \gamma \in C, \exists a \in \Sigma, (\gamma,a) \in \arena$.


\input{alg1}

The idea is to compute a "$\abs{\states}$-definable" "winning arena" $Y$ as guaranteed by \cref{thm:cortosconj} by refinement. 
Start with the largest such candidate set (line 1),
and reduce 
it if it is not an "arena" (lines 3 and 4)
or if it is not a "winning@arena" one (lines 5 and 6).
Computing ideal representations for reduced candidate sets, and the check in line 3, are syntactic and therefore simple. The check in line 5 is an instance of the \problemm.
Once the procedure stabilizes, $\ideal{\vec{V}}$ is a winning arena that satisfies the conditions of \cref{thm:cortosconj}.
In particular, it contains all of $\BoundK{\arena}{0}$.
The input therefore describes a positive instance if and only if $\vec{i} \in \vec{V}$, the output of the algorithm in line 8.

Observe that $\vectV$ is deliberately initialised as the set of \emph{all} "symbolic configurations" that may appear in the decomposition of $Y$, and therefore only reduces during the computation.

%

\begin{theorem}\label{thm:main}
	The \problem\ is \exptime-complete.
\end{theorem}
\begin{proof}
\exptime-hardness is shown as \cref{thm:exp-hard}.
The matching upper bound is provided by \cref{alg:EXPSPACE}:
The full proofs that this is effective and correct are in Appendix~\ref{app:alg}.
It remains to argue that it takes at most exponential time.	
	\begin{itemize}
		\item $\vectV$ has $(|S|+2)^{|S|}|\Sigma|$ elements at the start, and every
			iteration removes at least one.
		
		\item The first condition (line 3) can be checked simply by computing the set of successors of each symbolic commit in $\vectV$.
			This can be straightforwardly done in exponential time in the number of states (\cref{lem:condition-1}).
		
		\item Checking the second condition (line 5) is solving an instance of the \problemm\ with largest constant $\leq |S|$, for each $(\vectv, a)$ in $\vectV$. This problem is solvable in polynomial space (and thus exponential time) in $|S|$ (\cref{sec:seq-flows}).
	\end{itemize}
	Therefore, Algorithm~\ref{alg:EXPSPACE} takes exponential time in the size of the input.
\end{proof}

%% file: alg1.tex

\newcommand{\OLY}{\vectV}

\begin{algorithm}[ht]
	\caption{Algorithm for the \problem}\label{alg:EXPSPACE}
	\begin{algorithmic}[1]

		\State{$\OLY \gets {\set{0,\ldots, |S|, \omega}}^S \times \act$}  \Comment{start with maximal $\abs{\states}$-definable candidate}

		\Repeat
		\If{$\exists (\vectv,a) \in \OLY,  \gamma_1 \in \ideal{\vectv}$, $\gamma_2 \notin \ideal{\OLY}$ s.t.~$\Delta(\gamma_1,a)(\gamma_2)>0$}
		\State $\OLY \gets \OLY \setminus \set{(\vectv, a)}$
	\Comment{reduce if not an arena}
		\EndIf
		\If{$\exists (\vectv,a)\in \OLY$, $\gamma \in \ideal{\vectv}$ s.t.~there is no path from $\gamma$ to $\ideal{\textbf{f}}$ in $\ideal{\OLY}$}
		\State $\OLY \gets \OLY \setminus \set{(\vectv, a)}$
		\Comment{reduce if not a winning arena}
		\EndIf
		\Until{$\OLY$ does not change}
		\State \Return {$\textbf{i} \in \OLY$}

	\end{algorithmic}
\end{algorithm}

%
%
%
%
%
%
%
%

%% file: sec.winning-regions.tex
%
%
%
Consider now a positive instance of the \problem{} with
maximal winning arena $W$.


Towards proving \cref{thm:cortosconj}, we 
now describe a "winning strategy" for \Laetitia that
remains in, and thus defines a suitable "winning@@arena" sub-arena $Y \subseteq W$  which satisfies the claim of the theorem. As discussed previously, the inclusion $Y \subseteq W$ might be strict.

We consider the tokens of any configuration
to be split between those that are part of the 
large \emph{cohort} and a few
\emph{individuals} that are tracked separately.
%
Increasing the number of tokens in states occupied by the cohort cannot result in configurations outside the winning region (see \cref{def:omega-base}).
By contrast, increasing the number of individuals on a state may turn a winning configuration into a losing one.

The strategy operates in two modes.
The first mode attempts to follow 
a ``lucky'' path in $W$ that leads to a final configuration:
one where at any step only a few tokens
leave the cohort
and those that do, subsequently never meet until they rejoin the cohort.
For instance, in \cref{ex:force-one} we would try sending tokens one by one via $s_0\to s_1\to s_2 \to s_4$, while keeping all others in their current states.
Formally, this means that the path stays within $\BoundK{W}{1}$, and therefore introduces at most $|S|$ individual tokens.
The existence of such paths is guaranteed by \cref{lem:onlyone}.
Successfully following such a path to the end has a positive probability that is bounded away from zero (for a given number of tokens).
%
%



%

In case the ``lucky''
path is exited prematurely, our strategy enters a second, recovery mode that, almost surely, brings the individual tokens back into the large cohort
(reaches $\BoundK{W}{0}$) in order to attempt a new ``lucky'' path to a final configuration again. See \cref{fig:sketch} for an illustration.
This must be possible because we have not left the maximal winning arena $W$, and therefore can still almost surely reach a final configuration in $\ideal{\vec{F}}\subseteq\BoundK{W}{0}$.
We again attempt to follow another ``lucky'' path towards this new target, using \cref{lem:onlyone}, thereby creating up to $|S|$ additional individuals along the way.
Again, following this path succeeds with probability bounded away from zero.

The issue with this approach is that we may continue to be unlucky
and track ever more individual tokens. 
However, as we show in \cref{lem:isolation},
one can play as described above while ensuring that 
\emph{individuals introduced at different steps never meet} before they are brought back to the cohort. 
As each step produces at most $|S|$ new individual tokens (necessarily on distinct states), we can see at most $|S|$ individual tokens in each state overall.


\begin{figure}[ht]
	\centering
	\input{Figures/unlucky-paths.tikz}
	\caption{Controller tries to follow the black path from $\ideal{\vec{i}}$ to $\ideal{\vec{F}}$, along which individual tokens (in blue) may be produced temporarily yet never meet in the same state.
		If this path is exited (red arrow) those individuals can meet,
		in which case we leave $\BoundK{W}{1}$ as here.
		We then attempt to recover by following a new black path towards 
		$\BoundK{W}{0}$ which can spawn new individual tokens (as here, in blue-green) and so on.
		\cref{lem:isolation} guarantees that tokens of different colors never meet 
		and therefore that there are no more than $|S|$ layers.
		We define $Y$ as the union of all those layers. 
	}
	\label{fig:unlucky-paths}
	\label{fig:sketch}
\end{figure}

%

We will need some notation to formalize our argument.
In particular, we extend all previous definitions to (symbolic) configurations
that explicitly track finitely many individual tokens.
We thus extend the symbolic representations in \cref{def:symbolic}
accordingly.


\begin{definition}	
	\label{def:symbolic-tracking}
	Let $T_f$ be a finite set of tokens, called individuals, whereas other tokens in $T_\infty \setminus T_f$
	are called the cohort tokens.
	Given a "symbolic configuration" $\vectv \in \Nb^S$ and a configuration $\gamma_f : T_f \to S$, the ""ideal tracking $T_f$"" generated by $\vectv$ and $\gamma_f$ is written $\intro*\trackideal{\gamma_f}{\vectv}$ and defined as the set of configurations $\gamma : T \to S$ such that:
	\begin{itemize}
		\item individuals are placed on states according to $\gamma_f$, i.e. 
		$\forall t \in T_f, \gamma (t) = \gamma_f (t)$.
		
		\item the number of cohort tokens on a state is bounded from above by $\vectv$, i.e. $\forall s \in S, |\{ t \in T\setminus T_f, \gamma(t)=s\}| \leq \vectv(s)$.
	\end{itemize}
	
	Similarly, given an action $a \in \Sigma$, the ""commit ideal tracking $T_f$"" associated with $\vectv$, $\gamma_f$ and $a$ is denoted $\intro*\trackcommit{\gamma_f}{\vectv}{a}$, and is the set of commits $(\gamma, a)$ with $\gamma \in \trackideal{\gamma_f}{\vectv}$.
	A ""population arena tracking $T_f$"" is a finite union of "commit ideals tracking $T_f$".
	
	
	Let $\arena$ be an "arena"
	and $K \in \NN$.
	Then
	\(
	\intro*\BoundKT{\arena}{K}{T_f}
	\)
	denotes the union of all "ideals tracking $T_f$" of the form $\trackideal{\gamma_f}{\vectv}$ with $\vectv \in \set{0,\ldots, K,\omega}^S$ that are included in $\arena$.
\end{definition}








\subsection{The Existence of Lucky Paths\label{sec:luckyplays}}


We start from a fairly simple idea: Suppose that we are able to transfer arbitrarily many tokens from one state to another along paths that remain in the "winning region".
It may be necessary to transfer those tokens in small groups, as they have to go through a bottleneck. In this case, we might as well transfer them one by one.


%
%

\begin{restatable}{lemma}{FunnelingTheHerd}\label{lem:onlyone}
	Let $T_f$ be a finite set of tokens,  $\vectF$ an "ideal tracking $T_f$" and $\arena$ a "population arena tracking $T_f$".
	If $\arena$ is a "winning arena" with respect to $\vectF$, then for every configuration $\gamma_0 \in \BoundKT{\arena}{0}{T_f}$ 
	there is a path in $\BoundKT{\arena}{1}{T_f}$ from $\gamma_0$ to $\vectF$.
\end{restatable}

\begin{proof}[Sketch of proof]
	The full proof is presented in Appendix~\ref{app:funneling}. 
	As $\arena$ is a "population arena tracking $T_f$", by definition it is the union of finitely many "commit ideals tracking $T_f$".
	Let $B$ be the highest finite number used in the "symbolic commits" defining those "ideals".
	
	By definition, $\BoundKT{\arena}{0}{T_f}$ is also a finite union of "commit ideals tracking $T_f$". Let  $ \trackcommit{\gamma_f}{\vectv}{a}$ with $\vectv \in \set{0, \omega}^S$ be one of them.
	Define $\trackconf{\gamma_f}{\vectv}[N]$ as a configuration obtained by taking $\gamma_f$ and adding $N$ tokens on each state such that $\vectv(s) = \omega$. 
	Since all configurations of $ \trackideal{\gamma_f}{\vectv}$ can be obtained from one of the form $\trackconf{\gamma_f}{\vectv}[N]$ 
	by deleting and renaming tokens, we only need to show that for all $N$ we have a path from $\trackconf{\gamma_f}{\vectv}[N]$ to $\vectF$ in $\BoundKT{\arena}{1}{T_f}$. Let $d$ be the number of states $s$ such that $\vectv(s) = \omega$.
	
	We will use the fact that we have paths from $\trackconf{\gamma_f}{\vectv}[N]$ for arbitrarily large $N$, while there is a uniform bound $B$ on the constraints defining $\arena$.
	Intuitively, if we are able to transfer arbitrarily many tokens through a ``bottleneck'' of size $B$, then we can also transfer them one by one: essentially, we take the path used for $BN$ tokens and turn it into a path for $N$ tokens where they go through bottlenecks one by one.
	The transformation is done by interpreting the movement of tokens as a flow through a graph, where "commit ideals" are interpreted as capacities.
	More formally, let $N \in \NN$. We have $\trackconf{\gamma_f}{\vectv}[B \cdot N] \in \trackideal{\gamma_f}{\vectv} \subseteq \arena$, hence there is a path $\gamma_0 \xrightarrow{a_1} \cdots \gamma_n$ from $\gamma_0 = \trackconf{\gamma_f}{\vectv}[B\cdot N]$ to $\vectF$ in $\arena$.
	We interpret it as a directed graph $G$ whose set of vertices is $S \times \{0,\ldots, n\}$, plus a source and a target vertex.
	In $G$, there is an edge from $(s, j)$ to $(s',j+1)$ whenever there exists a token $t \notin T_f$ such that $\gamma_{j}(t) = s$ and $\gamma_{j+1}(t) = s'$.
	
	We assign capacities to all vertices of $G$ according to the "commit ideals" of $\arena$ the associated commits belong to. Those capacities are in $\set{0, \ldots, B, \omega}$. 
	The trajectories of tokens outside $T_f$ in the path define a flow in $G$ of value $dBN$. 
	We replace every positive finite capacity in $G$ by $1$ and use the max-flow min-cut theorem to show that the new graph has a flow $\geq dN$. we show that this flow defines a path from $\trackconf{\gamma_f}{\vectv}[N]$ to $\vectF$ in $\BoundKT{\arena}{1}{T_f}$.
\end{proof}


\subsection{The Isolation Lemma\label{ref:isolation}}

In a configuration, we say that a set of states is an \emph{$\omega$-base} if
an arbitrary amount of extra tokens could be placed 
on these states without exiting the arena $\arena$.
This is formally defined as follows.
\begin{definition}[$\omega$-base and finite base]
	\label{def:omega-base}
	Fix an "arena" $\arena$
	and $\gamma:T\to \states$ a configuration in $\arena$ over a set of tokens $T$.
	A set of states $S_\omega$ is an ""$\omega$-base"" of $\gamma$ in $\arena$ if
	\[
	\ideal{\gamma[{S_\omega * \omega}]} \subseteq \arena\enspace.
	\]
	with $\gamma[{S_\omega * \omega}]$ the "symbolic configuration" obtained from $\counttokens{\gamma}$ by mapping states of $S_\omega$ to $\omega$.
	
	By extension, a set of tokens $T_\omega \subseteq T$ is an \reintro{$\omega$-base} of $\gamma$ in $\arena$ if the set
	of states occupied by those tokens in $\gamma$ is.
	Dually, a set of tokens $T_f \subseteq T$ is a ""finite base"" of $\gamma$ in $\arena$
	iff $T \setminus T_f$ is an "$\omega$-base" of $\gamma$ in $\arena$.
\end{definition}

Note that a configuration $\gamma$ may have several "$\omega$-bases" and "finite bases": 
for instance, say we have two states $s_1, s_2$ and the "arena" $\Gamma$ is $\ideal{(\omega, 1)} \cup \ideal{(1,\omega)}$, then the configuration $[t_1 \mapsto s_1, t_2 \mapsto s_2]$ has $\set{t_1}$ and $\set{t_2}$ as "$\omega$-bases", but not $\set{t_1, t_2}$.

\AP We say that two tokens $t_1, t_2$ \intro{meet} in a
configuration $\gamma$ if they share the same state i.e. if $\gamma(t_1) = \gamma(t_2)$.
The following lemma says that if starting from a configuration with a finite base $T_f$ we can reduce the size of the finite base (bring back an individual token into the cohort) with a strategy that treats tokens symmetrically, then we can do so while making sure that the tokens of $T_f$ never meet the tokens outside $T_f$ before this goal is achieved.

Intuitively, while the tokens of $T_f$ are all in bounded places, it means that they only meet boundedly many other tokens. 
Since this strategy works with an arbitrarily large cohort, we can show that in fact we can strengthen it to make sure that no token meets the ones in $T_f$. 

{
\begin{restatable}[Isolation Lemma]{lemma}{IsolationLemma}
	\label{lem:isolation}
	Fix a "population arena" $\arena$ and $k \in \NN$, and assume that from all $\gamma \in \arena$ there is a strategy
	to almost surely reach a configuration with a finite base of size $<k$,
	without leaving $\arena$.
	
	Then for all $\gamma' \in \arena$ with a finite base $T_f$ of size $k$, there is a strategy $\sigma$ which, when starting from $\gamma'$:
	\begin{itemize}
		\item surely remains in $\arena$;
		\item almost surely reaches a configuration with a finite base strictly included in $T_f$, and
		\item guarantees that in every configuration before that, every state contains either only tokens of $T_f$ or no token of $T_f$.
	\end{itemize}
\end{restatable}

\begin{proof}[Sketch of proof]
	A crucial element in the assumption is that $\arena$ is a "population arena", i.e., it does not distinguish tokens.
	Since we have a "winning strategy" to reach a configuration with finite base of size $<k$ from everywhere in $\Gamma$, a "safe random walk" in $\Gamma$ is a winning strategy, from all configurations of $\Gamma$.
	Again, this random walk treats tokens symmetrically.
	
	Now consider an initial configuration $\gamma'$ in $\Gamma$, with $N$ tokens. On states of the "$\omega$-base", we can add many more tokens and still have a winning strategy. We add many imaginary tokens on all states of the "$\omega$-base".
	Then we make the following observation: as long as tokens in $T_f$ remain outside the cohort,
	no token of the initial "finite base" $T_f$ meets the cohort; they can only meet a few tokens at a time. 
	If a token $t$ of $T_f$ meets sufficiently many tokens along a path, then some of them must have re-entered the cohort since they met.
	Since the random walk treats all tokens symmetrically, $t$ also had a chance of reaching the cohort.
	This lets us bound the expected number of tokens $t$ meets before reaching the cohort, independently of the number of tokens in total.
	As we add more imaginary tokens, the probability that one of the tokens $t$ meets one of the  real ones converges to $0$.
	
	We use these observations to show that we can obtain a probability as small as we want that any of the tokens of $T_f$ meets tokens outside $T_f$ before reaching the cohort.
	Since the set of configurations reachable from $\gamma'$ is finite, this implies that this probability can be brought down to $0$.
	As a consequence, we can make sure that a token of $T_f$ reaches the cohort while no token of $T_f$ meets one outside $T_f$ beforehand.
	The proof is detailed in Appendix~\ref{app:isolation}.
\end{proof}
}

\subsection{Proof of \cref{thm:cortosconj}}

The proof of Theorem~\ref{thm:cortosconj} uses a partition $T_1,\ldots ,T_d$
of the individuals $T_f$
into non-empty groups with at most $|S|$ tokens per group.
Moreover, the induction hypothesis assumes that
\begin{itemize}
\item individuals of different groups never meet in the "arena",
unless in final configurations.
\item $T_f$ is a finite base, i.e., the set of states occupied by the cohort is an $\omega$-base.
\end{itemize}
Since every group occupies at least one state then $d\leq |S|$.
The induction step is done by induction on $|S|-d$.
The base case $|S|=d$ is trivial, because in that case all tokens are individuals: the cohort is empty,
and $|S|$-definability follows from the hypothesis that tokens of different groups never meet unless on $F$.
The value of $F$ is not fixed; the induction step uses the inductive hypothesis for $F=W_0$.
The induction step is done as follows,
full details are in Appendix~\ref{app:induction}.
\begin{itemize}
	\item We try to follow a path to the target set $\ideal{\vec{F}}$. We can choose this path so that we create at most $|S|$ individual tokens, by \cref{lem:onlyone}. These $|S|$ isolated tokens constitute a new group of individuals, denoted $T_{d+1}$.
	\item If we deviate from that path, we use \cref{lem:isolation} to  define a "sub-arena" that lets us recover all individual tokens (i.e. bring them into the cohort) while making sure that
		none of them meets any token from outside their group.
	We then apply the induction hypothesis (for $d+1$ and $F$ the set of configurations where tokens of $T_{d+1}$ have been recovered) to recover those individual tokens within an "$|S|$-definable" sub-arena.

	\item Once all individuals are recovered, we apply the first step again, until we successfully follow the path to the end.
\end{itemize}
We define the desired "sub-arena" $Y$ as the union of those paths and "sub-arenas". Since, in any included configuration, the number of individuals residing on any one state is at most $\abs{\states}$, this arena must be $\abs{\states}$-definable. Clearly, \Laetitia wins from any included configuration by following the strategy outlined above. Therefore, $Y$ is a "winning@arena" "arena".\qed

%% file: Figures/unlucky-paths.tikz
\tikzset{every state/.style={
            thick,
            fill=black!4,
            rounded corners=1mm,
            minimum size=1.5mm,inner sep=0.5mm,
}}

\tikzset{anode/.style={
        font=\small,
}}
\tikzset{estate/.style={state,rectangle,font=\scriptsize}}

 \definecolor{col1}{HTML}{FA6687} 
 \definecolor{col2}{HTML}{1E88E5} 
 \definecolor{col3}{HTML}{FFC107} 
 \definecolor{col4}{HTML}{A3E4D9} 

\newcommand{\HL}[2]{{\color{#1}{{#2}}}}
\begin{tikzpicture}[>=latex',shorten >=1pt,node distance=1.9cm,on grid,auto,
		lost/.style={densely dotted, ->, line width=1.5pt},
	]
        \coordinate (diff) at (1mm,1mm);
        \coordinate (tr0) at ($(35mm,35mm)-3*(diff)$);
        \coordinate (tr1) at ($(65mm,35mm)-2*(diff)$);
        \coordinate (tr2) at ($(90mm,35mm)-1*(diff)$);
        \coordinate (tr3) at ($(105mm,35mm)$);
        \coordinate (trS) at ($(\textwidth,35mm)+1*(diff)$);
	\draw[rounded corners, thick] (0, 0) rectangle (trS);
	\draw[color=col3,rounded corners, thick] ($(diff)$) rectangle (tr3);
	\draw[color=col4,rounded corners, thick] ($2*(diff)$) rectangle (tr2);
	\draw[color=col2,rounded corners, thick] ($3*(diff)$) rectangle (tr1);

	 \node (dots) at (11cm,1.5cm) {$\cdots$};

	\begin{scope}
	    \clip[rounded corners] ($4*(diff)$) rectangle (tr0);
		\fill[fill=black, opacity=0.1] (0,6) circle (3.6cm);
		\fill[fill=black, opacity=0.2] (1,-2) circle (3cm);
	\end{scope}
	\draw[rounded corners, thick] ($4*(diff)$) rectangle (tr0);

	 \node at (tr0) [anchor=north east]   {$\BoundK{\arena}{0}$};
	 \node at (trS) [anchor=north east]   {$Y = Y_{|S|}\subseteq W $};

	 \node (iF) at (1,7mm)  {$\ideal{\vec{f}}$};
	 \node (iI) at (tr0) [xshift=-2cm,anchor=north east]   {$\ideal{\vec{i}}$};

	 \node[state] (p0) at (15mm,30mm) {};
	 \node[estate] (p1) at (55mm,15mm) {$\omega,\HL{col2}{2},0,\omega,0$};
	 \node[estate] (p2) at (23mm,23mm) {$0,\omega,0,\omega,0$};
	 \node[estate] (p3) at (21mm,12mm) {$0,0,\omega,\omega,0$};
	 \node[state] (p4) at (15mm,5mm) {};

	 \node[estate] (p3') at (42mm,25mm) {$\omega,\HL{col2}{1},\HL{col2}{1},\omega, 0$};
	 \node[estate] (p4') at (42mm,10mm) {$0,\HL{col2}{1},\omega,\HL{col2}{1}, \omega$};

	 \node[estate] (p3'') at (72mm,24mm) {$\omega,\HL{col4}{1},\HL{col2}{2},0,0$};
	 \node[estate] (p4'') at (71mm,5mm) {$\HL{col4}{1},\HL{col2}{2},\omega,\HL{col4}{1},\omega$};

	 \node[estate] (p5) at (80mm,15mm) {$\HL{col4}{2},\omega,\HL{col2}{2},0,\omega$};
	 \node[estate] (p7) at (97mm,20mm) {$\HL{col2}{2},\HL{col4}{2},\HL{col3}{1},\omega,\HL{col3}{1}$};
	 
	 \draw[->]
	     (p0) edge (p2)
	     (p2) edge (p3')
	     (p3) edge (p4)
	     ;
	 
	 \draw[->]
	 (p3') edge[col1] node[yshift=-3pt]  {\footnotesize!} (p1) 
	 (p1) edge (p3'')
	     (p3') edge (p4')
	     (p4') edge (p3)
	     ;
	  \draw[->]
	      (p3'') edge (p4'') 
	      (p4'') edge (p4')
	      (p4'') edge[col1] node[yshift=-3pt] {\footnotesize !} (p5)
	      (p5) edge (p7)
	      ;
\end{tikzpicture}

%% file: sec.flow-semigroup.tex

To prove \cref{thm:pspace-path-problem}, we rely on a result by Gimbert, Mascle and Totzke~\cite{GimbertMT25}
about maximising the flow through directed graphs, where edge capacities are dynamically chosen by the maximiser.
We recall the notation and statement of the relevant theorem.

\knowledgenewrobustcmd{\globalflow}[1]{\cmdkl{g(#1)}}
\knowledgenewrobustcmd{\multiflow}{\cmdkl{\xi}}

\AP A ""capacity"" is a function $\alpha : S^2 \to \NN \cup \set{\omega}$ mapping each pair of states $(s_1, s_2)$ to a maximal number of tokens that can be transferred from $s_1$ to $s_2$.
Let $A \subseteq  (\NN \cup \set{\omega})^{S^2}$ be a finite alphabet of "capacities", and $E \subseteq S^2$ a set of pairs of states.
Given a word $w = \alpha_1 \dots \alpha_k \in A^*$ and a set of tokens $T$, a ""token flow"" over $w$ is a sequence of configurations $\tau = \gamma_0 \to \cdots \to \gamma_k \in (S^T)^*$ such that, for every $i \in [1, k]$
and $s, s' \in S$,
$|\set{t \in T \mid \gamma_{i-1}(t) = s \land \gamma_{i}(t) = s'}| \leq \alpha_i(s, s')$.
Define the ""global flow"" of $\tau$ as the function $\intro*\globalflow{\tau} : S^2 \to \NN$ counting the number of tokens moving between each pair of states: 
formally, for all $s,s' \in S$,
\[
	\globalflow{\tau}(s,s') = \abs{\set{t \in T : \gamma_0(t) = s \land \gamma_k(t) = s'}}.
\] 
 Now let $L \subseteq A^*$ be a language over $A$. 
The ""maximal sequential multi-flow on $E$ satisfying $L$"", denoted $\intro*\multiflow(E,L)$,  is the maximum number of tokens that can be simultaneously transferred between every pair of states in $E$ by a path over some word of $L$.
\[
\multiflow(E,L)=\sup_{w \in L}\quad\sup_{\pi \text{ over } w}\quad\min_{(s,s') \in E} \globalflow{\pi}(s,s').
\]

\begin{example}
Consider the capacities from \cref{fig:tiles},
and set of edges $E=\set{(s_1,s_4), (s_3,s_4)}$. For the language $L=\set{abba}$ containing only the capacity word $abba$, we get
$\multiflow(E,L) = 1$, witnessed by the token flow $\tau$ with global flow $\globalflow{\tau}(s_1,s_4)=\globalflow{\tau}(s_3,s_4)=1$.
Note that the maximal flow from $s_2$ to $s_3$ for the intermediate word $bb$ cannot exceed $2$.
However letting $L=\{ab^na\}$, we have 
$\multiflow(E,L) = \infty$, because for every $n$ there is a token flow $\tau_n$ over $ab^{2n}a$ so that $\globalflow{\tau_n}(s_1,s_4)=\globalflow{\tau_n}(s_3,s_4)=n$.
\end{example}

As Gimbert et al.~show, one can determine if the such flows are unbounded, and otherwise precisely compute their finite value, in polynomial space.
\input{fig3}
\begin{theorem}[Theorem~37 in \cite{GimbertMT25}]\label{thm:citeSFP}
	Let $A$ be a set of capacities over a set of states $S$, with coefficients in $\set{0, \dots, K, \omega}$, and $L \subseteq A^*$ a regular language recognised by an NFA with $m$ states.
	Then, either $\multiflow(E,L) = +\infty$, or $\multiflow(E,L) \leq K(2|S|)^{(170 \log_2(m)+835)|S|^{12}}$.
\end{theorem}
This is almost what we need for deciding the \problemm, except for two things.
First, our constraints are on vertices while theirs are on edges. 
Second, in our initial "ideal" we may have both states marked $\omega$ and states bounded by finite numbers.
We will translate the constraints given by "ideals" over configurations into "capacities", thereby reinterpreting "paths" as "token flows".
This way one can use \cref{thm:citeSFP} to check that we can transfer arbitrarily many tokens from the states marked $\omega$, while using the regular constraint to track the finite number of extra tokens and make sure that all token end up in $F$.

%

Let $\?M = (S, \Sigma, \Delta)$, $\vectV$, $\vectv_0$, $\vectF$ be an instance of the \problemm.
We split $\vectv_0$ in two parts $\vectv_0 = \vectv_f + \vectv_\omega$, so that $\vectv_f \in \NN^S$ and $\vectv_\omega \in \set{0, \omega}^S$. 
Vector $\vectv_\omega$ describes the set of states with arbitrarily many tokens at the start, while $\vectv_f$ describes the remaining tokens in other states.
Intuitively, the \problemm{} comes down to checking that we can simultaneously transfer tokens of $\vectv_f$ and arbitrarily many tokens from every state marked  $\omega$ in $\vectv_0$ to $\ideal{\vectF}$.
We will construct a finite alphabet of capacities and a finite automaton $\?A$ which guesses the movement of the remaining tokens. Its language is the set of sequences of "capacities" which can be used while making sure that both the cohort of tokens from unbounded states and those remaining tokens are brought safely to $\vectF$. 
The largest finite number appearing in the "capacities" is the same as in $\vectV$, and $\?A$ has exponential size in the input.

See Appendix~\ref{app:reduction-to-SFP} for the automaton construction and a proof of the following lemma.

\begin{restatable}{lemma}{ReductionToSFP}
	\label{lem:reduction-to-SFP}
	$\?M$, $\vectV$, $\vectv_0$, $\vectF$ is a positive instance of the \problemm{} if and only if 
	there exists $E \subseteq S^2$ such that 
	\begin{enumerate}
		\item\label{cond1} $\multiflow(E,L(\?A)) = +\infty$
		
		\item\label{cond2} $\vectv_0^{-1}(\omega) \subseteq \set{s \in S  \mid \exists s' \in S, (s,s') \in E}$.
	\end{enumerate}
\end{restatable}

\begin{proof}[Proof of Theorem~\ref{thm:pspace-path-problem}]
	By \cref{lem:reduction-to-SFP}, in order to check if $\?M$, $\vectV$, $\vectv_0$, $\vectF$ is a positive instance of the \problemm, we can
	enumerate sets $E \subseteq S^2$
and for each such $E$, check the two conditions from \cref{lem:reduction-to-SFP}.		
	%
	Condition~\ref{cond2} can be checked easily in polynomial time (and space).
	For condition~\ref{cond1}, we can rely on \cref{thm:citeSFP}:
	Note that the automaton $\?A$ as constructed above is singly exponential in the our input.
	Consequently, the same holds for $m$ and the quantity $M = K(2|S|)^{(170 \log_2(m)+835)|S|^{12}}$.
	This means that in order to check whether $\multiflow(E, L(\?A))= + \infty$, it suffices to check that there is a word $w \in L(\?A)$ and a path $\pi$ over $w$ carrying $M+1$ tokens from $s$ to $s'$ for all $(s,s') \in E$.
	These can be guessed on the fly: it suffices to memorise the current state in the automaton and the number of tokens in each state, in binary.
	This only requires polynomial space in $\log(K)$ and $|S|$.
 \end{proof}

%



%% file: fig3.tex
\begin{figure}[t]
  \centering
  \newlength{\subfigheight}
  \setlength{\subfigheight}{2.5cm} 
  \begin{subfigure}[b][\subfigheight][b]{0.16\textwidth}
    \centering
    \vspace*{\fill}
    \input{Figures/example1.a.tikz}
    \vspace*{\fill}
    \caption*{\centering capacity $a$}
  \end{subfigure}
  \hfill
  \begin{subfigure}[b][\subfigheight][b]{0.18\textwidth}
    \centering
    \vspace*{\fill}
    \input{Figures/example1.b.tikz}
    \vspace*{\fill}
    \caption*{\centering capacity $b$}
  \end{subfigure}
  \hfill
  \begin{subfigure}[b][\subfigheight][b]{0.27\textwidth}
    \centering
    \vspace*{\fill}
    {
      \begin{tikzpicture}[scale=0.47]
	\flowlabel[xshift=-0.3cm]{4}{1/$s_1$,2/$s_2$,3/$s_3$,4/$s_4$}
        \flowwithletter[extracolour1]{$a$}{4}{1-2/$\omega$,3-2/$\omega$,3-4/$\omega$}
        \flowwithletter[extracolour2,xshift=1.5cm]{$b$}{4}{2-2/$\omega$,2-3/$1$,3-3/$\omega$}
        \flowwithletter[extracolour2,xshift=3cm]{$b$}{4}{2-2/$\omega$,2-3/$1$,3-3/$\omega$}
        \flowwithletter[extracolour1,xshift=4.5cm]{$a$}{4}{1-2/$\omega$,3-2/$\omega$,3-4/$\omega$}
      \end{tikzpicture}
    }
    \vspace*{\fill}
    \caption*{\centering capacity word $abba$}
  \end{subfigure}
  \hfill
  \begin{subfigure}[b][\subfigheight][b]{0.3\textwidth}
    \centering
    \vspace*{\fill}
    {
\begin{tikzpicture}[scale=0.47]
  \flowlabel[xshift=-0.4cm]{4}{1/$s_1$,2/$s_2$,3/$s_3$,4/$s_4$}
  \node[flow point content,colored tokens={col3}] (L1) at (0, 0cm) {};
  \node[flow point content] (L2) at (0, -1cm) {};
  \node[flow point content,colored tokens={col4}] (L3) at (0, -2cm) {};
  \node[flow point content] (L4) at (0, -3cm) {};
  \node[flow point content] (R1a) at (1.5, 0cm) {};
  \node[flow point content,colored tokens={col4,col3}] (R2a) at (1.5, -1cm) {};
  \node[flow point content] (R3a) at (1.5, -2cm) {};
  \node[flow point content] (R4a) at (1.5, -3cm) {};

  \path[extracolour1] (L1) edge[-] node[flow edge label] {$1$} (R2a);
  \path[extracolour1] (L3) edge[-] node[flow edge label] {$1$} (R2a);
  \node[yshift=0.5cm, flow point label,extracolour1] at (0.75,0) {$a$};

  \begin{scope}[xshift=1.5cm]
    \node[flow point content] (R1b) at (1.5, 0cm) {};
    \node[flow point content,colored tokens={col3}] (R2b) at (1.5, -1cm) {};
    \node[flow point content,colored tokens={col4}] (R3b) at (1.5, -2cm) {};
    \node[flow point content] (R4b) at (1.5, -3cm) {};
    \path[extracolour2] (R2a) edge[-] node[flow edge label] {$1$} (R2b);
    \path[extracolour2] (R2a) edge[-] node[flow edge label] {$1$} (R3b);
  \node[yshift=0.5cm, flow point label,extracolour2] at (0.75,0) {$b$};
  \end{scope}

  \begin{scope}[xshift=3cm]
    \node[flow point content] (R1c) at (1.5, 0cm) {};
    \node[flow point content] (R2c) at (1.5, -1cm) {};
    \node[flow point content,colored tokens={col3,col4}] (R3c) at (1.5, -2cm) {};
    \node[flow point content] (R4c) at (1.5, -3cm) {};
    \path[extracolour2] (R2b) edge[-] node[flow edge label] {$1$} (R3c);
    \path[extracolour2] (R3b) edge[-] node[flow edge label] {$1$} (R3c);
  \node[yshift=0.5cm,flow point label,extracolour2] at (0.75,0) {$b$};
  \end{scope}

  \begin{scope}[xshift=4.5cm]
    \node[flow point content] (R1d) at (1.5, 0cm) {};
    \node[flow point content] (R2d) at (1.5, -1cm) {};
    \node[flow point content] (R3d) at (1.5, -2cm) {};
    \node[flow point content,colored tokens={col3,col4}] (R4d) at (1.5, -3cm) {};
    \path[extracolour1] (R3c) edge[-] node[flow edge label] {$2$} (R4d);
  \node[yshift=0.5cm, flow point label,extracolour1] at (0.75,0) {$a$};
  \end{scope}
\end{tikzpicture}
    }
    \vspace*{\fill}
    \caption*{\centering token flow $\tau$ over $abba$}
  \end{subfigure}
  \caption{Two capacity constraints $a,b:\states\to\NN\cup\{\omega\}$,
a capacity word, and a token flow over it.}
\label{fig:tiles}
\end{figure}

%% file: Figures/example1.a.tikz
\begin{tikzpicture}[node distance=1.2cm,font=\scriptsize]

\node[state] (s) at (0,0) {$s_1$};
\node[state,right of=s] (u) {$s_2$};
\node[state,below of=s] (v) {$s_3$};
\node[state,right of=v] (t) {$s_4$};

\path[use as bounding box] (s) rectangle (t);
\draw[extracolour1] (s) edge node[above] {$\omega$} (u);
\draw[extracolour1] (v) edge node {$\omega$} (u);
\draw[extracolour1,->] (v) edge node[above] {$\omega$} (t);
\end{tikzpicture}

%% file: Figures/example1.b.tikz
\begin{tikzpicture}[node distance= 1.2cm,font=\scriptsize,scale=0.8]

\node[state] (s) at (0,0) {$s_1$};
\node[state,right of=s] (u) {$s_2$};
\node[state,below of=s] (v) {$s_3$};
\node[state,right of=v] (t) {$s_4$};

\path[use as bounding box] (s) rectangle (t);
\draw[extracolour2,->] (v) edge[loop left] node[extracolour2,left] {$\omega$} (v);
\draw[extracolour2,->] (u) edge[loop right] node[extracolour2,right] {$\omega$} (u);
\draw[extracolour2,->] (u) edge[] node[extracolour2,above] {$1$} (v);
\end{tikzpicture}

%% file: lower-bound-main.tex
An exponential-time lower bound for the \problem\ 
can be established by reduction from countdown games, as shown by 
Mascle, Shirmohammadi, and Totzke in \cite{DBLP:journals/corr/abs-1909-06420}.

\begin{restatable}{definition}{DefCDG}\label{def:CountDownGames}
\AP A ""Countdown Game""
is given by a directed graph $\?G=(V,E)$,
where edges carry positive integer weights, 
$E\subseteq (V\x\+N_{>0}\x V)$. 
For an initial pair $(v,c_0)\in V\x\+N$ of a vertex and a number, 
two opposing players (Player~1 and~2)
alternatingly determine a sequence of such pairs as follows.
In each round, from $(v,c)$,
Player~1 picks a number $d\le c$ such that $E$ contains at least one edge $(v,d,v')$;
then Player~2 picks one such edge and the game continues from $(v',c-d)$.
Player~1 wins the game iff the play reaches a pair in $V\x\{0\}$.
\end{restatable}

Determining the winner of a Countdown Game, where all constants are given in binary, 
is \exptime-complete \cite{JLS2007}.
We state the lower bound and a sketch of the construction. The full proof is detailed in Appendix~\ref{sec:app-lower-bound}.

\begin{restatable}{theorem}{ThmExpHard}\label{thm:exp-hard}
	The \problem\ is \exptime-hard.
\end{restatable}
\begin{proof}[Proof sketch]
%
The number of turns in a Countdown Game cannot exceed the initial value of the counter, as the initial counter value decreases at each turn. Thus, if Player~2 has a winning strategy, choosing actions at random yields a positive probability of applying that strategy, hence a positive probability of winning. Therefore, Player~1 wins the initial game if and only if she wins with probability one against a randomised adversary.

The main idea for the further construction is to require \Laetitia, who impersonates Player~1, to move tokens one-by-one 
from a waiting state, first into the control graph of the Countdown game, and ultimately into the target.
To avoid a loss in the intermediate phase, she must win an instance of the game against a randomising opponent.
This is enforced using a combination of gadgets, including two binary counters (encoded using two states per bit) that 
can effectively test for zero, be set to specific numbers,
and that are set up so that they can decrement at the same rate.
These are used to hold the global integral value $n$, and an auxiliary counter holding the value $d$ chosen by Player~1.
\Laetitia is compelled to reduce them both and can only continue once the auxiliary counter is exhausted.
She can only afford to safely end the simulation of the game if the first counter holds value $0$.
As a result, Player~1 has a winning strategy for the two-player Countdown Game if, and only if, \Laetitia
can synchronise the $n$-fold product of the constructed MDP
for all $n$.
\end{proof}

%% file: sec.conclusion.tex
We have shown that the \problem\ is \exptime-complete.
We establish
that it is possible to win the population MDP
while staying in a part of the winning region
that has low descriptive complexity, in the sense of 
Theorem~\ref{thm:cortosconj}.
This is the key to defining an algorithm (Algorithm~\ref{alg:EXPSPACE}) to solve the \problem\ using an exponential number of calls to an oracle solving the \problemm.



These results shed new light on parameterised control
and pave the way for further positive results with more ambitious objectives.
For example, the tools developed in this paper may be used to address a generalised version of the problem, 
where infinite executions must satisfy $\omega$-regular conditions.

A natural question that remains open is the cut-off: if there is a number $n$ such that we cannot synchronize almost surely $n$ tokens, how large can the smallest such $n$ be? 
In the adversarial case studied in~\cite{BertrandDGG17}, it was shown that the cut-off was doubly exponential in the worst case. 
We conjecture that this is also the case in our setting. We show in Appendix~\ref{app:example-2exp-cutoff} that the cut-off may be doubly exponential, but we are missing a matching upper bound.

%% file: app-win-arenas.tex

\WinningArenas*

\begin{proof}
	\begin{enumerate}
		\item Let $(\Gamma_i)_{i \in I}$ be a collection of "winning arenas". Their union is still an "arena", since for all $i$ every commit of $\Gamma_i$ guarantees to stay in $\Gamma_i$.
		Further, it is "winning@@strat" since for all $i$, for every state $s \in \Gamma_i$ there is a path to $F$ within $\Gamma_i$.
		As a consequence, the \emph{largest winning arena} $W$ is well-defined, as the union of all "winning arenas".
		It remains to argue that $W$ is the "winning region".
		Let $s$ be a state of the "winning region", and $\sigma$ a "winning strategy" from $s$.
		Let $\Gamma$ be the set of "commits" $(s',a)$ such that there is a positive probability to reach $s'$ before $F$ from $s$ while applying $\sigma$ and $\sigma(s',a)>0$.
		Clearly $\Gamma$ contains $s$, and it is easy to check that $\Gamma$ is a "winning arena", hence included in $W$.
		Thus the "winning region" is included in the projection of $W$ on $S$.
		The other inclusion is proven by the following item, which shows that there is a "winning@@strat" "strategy" from every state in $W$.

		\item Let $s$ be a state in the "winning region".
		Since we are in a "simple" MDP, the set of states reachable from $s$ is finite.
		Let $B$ be its size, and $p$ the smallest positive probability appearing in it.
		When starting from $s$, at all times the probability to reach $F$ within the next $B$ steps with a "safe random walk" is thus at least $(\frac{p}{|\Sigma|})^B$.
		Thus, the probability to not reach $F$ within $N$ steps from $s$ decreases exponentially with $N$, and converges to $0$ as $N$ grows.
		As a result, a "safe random walk" in $W$ is a "winning strategy" from $s$.\qedhere
	\end{enumerate}
\end{proof}

\noindent
The following lemma justifies our focus on the \emph{simultaneous} reachability objectives, that \Laetitia aims to move all tokens into target states at the same time rather than just ensuring that all tokens eventually visit a target state. These two objectives are clearly different. This can be seen for example by a two-state automaton between which tokens must alternate on any action and where only one state is final.
However, the two respective variants of the random population control problem are easily inter-reducible in logarithmic space.

\begin{lemma}
	\label{lem:simultaneously}
	The \problem{}, where \Laetitia aims to move all tokens into target states \emph{at the same time},
	is (logspace) inter-reducible to the variant where she aims to ensure that every token eventually reaches a target state, not necessarily at the same time as others.
\end{lemma}
\begin{proof}
	Let $\?M=(\states, \act, \prob)$, initial state $i\in\states$ and target set $F\subseteq\states$ be an instance of the \problem{}.
	To construct an equivalent instance of the non-simultaneous problem, 
	\begin{itemize}
		\item introduce a state $f$ that acts as a sink (has self-loops on all actions $a\in\act$) and is the only target state, and
		\item for every original target state $s\in F$, add a new action $\$$ and transitions $s\step{\$} f$.
	\end{itemize}
	If the given instance was winning, then for 
	every set $T\subseteq T_\infty$, there is a strategy to simultaneously move them all into $f$. Following such a strategy and playing action $\$$ once $F^T$ is reached is a winning strategy for the non-simultaneous objective in the constructed instance.
	Conversely, if the constructed instance is positive, then for all $T\subseteq T_\infty$ there must be a strategy that almost surely reaches $F^T$, as the new target $f$ is only reachable from states of $F$ via the action $\$$, which is not safe to play in configurations outside $F^T$.
	\medskip

	Towards the other reduction, if given an instance of the non-simultaneous problem,
	\begin{itemize}
		\item introduce a state $f$ that acts as a sink (has self-loops on all actions $a\in\act$) and is the only target state, and 
		\item for every original target state $s\in F$, add a new action $a_s$ and transitions $s\step{a_s} f$ as well as $s'\step{a_s}s'$ for all other states $s'\neq s$.
	\end{itemize}
	We claim that the constructed automaton is a positive instance of the simultaneous problem iff the given instance of the non-simultaneous problem is positive.
	Indeed, for any set $T\subseteq T_\infty$, a winning strategy in the non-simultaneous instance ensures  almost surely that every token $t\in T$ eventually visits a target state. We can follow this strategy and whenever
	some tokens occupy a state $s\in F$, interrupt and play the action $a_s$. This will move tokens from $s$ to the accepting sink and have no effect on all others.
	Since doing this only removes tokens from the configuration, we are still in the "winning region", one can continue with a "winning strategy".
	By assumption, eventually, all tokens will reach  $F$, and thus gather in $f$.
	The described interrupting strategy must therefore be winning for the simultaneous objective in the constructed automaton.
	Conversely, any winning strategy
	for the simultaneous objective towards $f$ dictates a winning strategy in the original instance (ignore the actions $a_s$) because a token can only enter $f$ from original accepting states.
\end{proof}

%% file: proof-alg1-correct.tex
Consider \cref{alg:EXPSPACE}.
The check in line 3, whether a candidate $\vec{V}$ is an "arena", is easy and syntactic, see the lemma below.

\begin{lemma}\label{lem:condition-1}
	Given a finite set $\OLY \subseteq \set{0,\ldots, |S|, \omega} \times \Sigma$ of "symbolic commits" with constants below $\card{S}$
	and $(\vectv,a)\in \OLY$,
	one can check, in exponential time in $|S|$,
	if there exist a commit $(\gamma_1,a) \in \ideal{(\vectv,a)}$ and a  configuration $\gamma_2 \notin \ideal{\OLY}$ with $\Delta(\gamma_1,a)(\gamma_2)>0$.
\end{lemma}
\begin{proof}
	Notice that if such configurations $\gamma_1$ and $\gamma_2$ exist, then already small ones.
	Indeed, let $\gamma_1^-$ be a configuration obtained as follows:
	in every state $s \in S$, if $\gamma_1$ has more than $|S|+1$ tokens in $s$, remove every token from $s$ except $|S|+1$ of them.
	Remove the same tokens from $\gamma_2$ to obtain a configuration $\gamma_2^-$.
	
	As constants in $\OLY$ do not exceed $|S|$, and as $\gamma_2 \notin \ideal{\OLY}$, $\gamma_2^-$ is not in $\ideal{\OLY}$ either.
	As $\gamma_1 \in \ideal{\OLY}$, we have $\gamma_1^- \in \ideal{\OLY}$.
	
	Observe that $\gamma_1^-$ and $\gamma_2^-$ both contain at most $(|S|+1)|S|$ tokens. 
	We can enumerate triples $\gamma_1,a,\gamma_2$ (up to renaming tokens) with at most $(|S|+1)|S|$ tokens in $\gamma_1$ and $\gamma_2$, and check that they satisfy the conditions, in exponential time in $|S|$.
\end{proof}

The condition in line 5 is much more challenging. It requires solving (the negation of) the \problemm.
\cref{sec:seq-flows} provides
an  algorithm which solves the \problemm{} in polynomial space in the number of states $|S|$ (and in the logarithm of the largest constant appearing in the input, which happens to be at most $|S|$ here) (Theorem~\ref{thm:pspace-path-problem}).

We now argue that \cref{alg:EXPSPACE} is correct.

\begin{lemma}
	\label{lem:alg-correct}
	Algorithm~\ref{alg:EXPSPACE} returns \emph{True} if and only if the answer to the \problem\ is positive.
\end{lemma}

\begin{proof}
Let $\vectV_0 \supsetneq \vectV_1 \supsetneq \ldots \supsetneq \vectV_n$ be the successive values of $\vectV$ throughout the execution,
which terminates by monotonicity.

	Suppose the algorithm returns \emph{True} in line 8. Let 	
	$Y = \ideal{\vectV_n}$.
	As we exited the loop in line 7, no "symbolic commit" was removed in lines 4 and 6.
	Hence the condition on line 3 was false, meaning that $Y$ is an "arena", and the condition on line 5 was false as well, meaning that $Y$ is a "winning arena" for reaching $\ideal{\textbf{f}}$. Furthermore, it contains $\ideal{\textbf{i}}$, since we returned \emph{True}.
	As a consequence, by Lemma~\ref{lem:winningarena},
	a "safe random walk in $Y$" almost surely reaches $\ideal{\textbf{f}}$ from all configurations in $\ideal{\textbf{i}}$.
	
	For the other direction, suppose there is a "winning strategy". Let $W$ be the maximal "winning arena". 
	By Theorem~\ref{thm:cortosconj}, there is an "arena" $Y \subseteq W$ satisfying the three conditions of Theorem~\ref{thm:cortosconj}. 
	Let $\vectU = \set{(\vectv,a) \in {\set{0,\ldots, |S|, \omega}}^S \times \act \mid \ideal{(\vectv,a)} \subseteq Y}$.
	As $Y$ is "$|S|$-definable", we have $Y = \ideal{\vectU}$.
	
	We show that the algorithm maintains the following invariant,
	defined for  $i \in \set{0\ldots n}$.
	\begin{align}
	\label{invariant1}
&	\vectU \subseteq \vectV_i.
	\end{align}
	This is clear for $i=0$.
	Assume the invariant holds for some $i< n$:
	since $\vectU \subseteq \vectV_i$, we have $Y \subseteq \ideal{\vectV_i}$.
	Since $\vectV_i \supsetneq \vectV_{i+1}$, some "symbolic commit" $(\vectv, a)$ is removed from $\vectV_i$ to obtain $\vectV_{i+1}$.
	As $Y$ is a "winning arena", for all $(\gamma,a) \in Y$, every successor of $\gamma$ by $a$ is in $Y$, and thus in $\ideal{\vectV_i}$.
	Also, for all $\gamma \in Y$ there is a path in $Y$ (and thus in $\ideal{\vectV_i}$) from $\gamma$ to $\ideal{\textbf{f}}$.
	
	By contrast, since $(\vectv,a)$ was removed from $\vectV_i$, it contains some $(\gamma,a)$ such that either $(\gamma,a)$ has a successor outside of $\vectV_i$, or $\ideal{\textbf{f}}$ is not reachable from $\gamma$ in $\vectV_i$.
	As a consequence, 
	 $\ideal{(\vectv,a)} \nsubseteq Y$,
	hence $(\vectv,a) \notin \vectU$, and thus $\vectU \subseteq \vectV_i\setminus \set{(\vectv,a)} = \vectV_{i+1}$. 

	Finally, the invariant holds for $i=n$. Thus, $ Y \subseteq \ideal{\vectV_n}$.
	Furthermore we have $ \BoundK{\arena}{0} \subseteq Y$, by definition of $Y$. 
	As there is a "winning strategy", we also have $\ideal{\textbf{i}} \subseteq \BoundK{\arena}{0}$. Hence $\ideal{\textbf{i}} \subseteq Y$, and thus $\textbf{i} \in \vectU$. As a result, the algorithm returns \emph{True}.
\end{proof}

%% file: proof-funneling-herd.tex

\newcommand{\RRp}{\overline{\mathbb{R}^+}}
\newcommand{\valueflow}{\mathbf{val}}
\newcommand{\capacity}{\mathbf{weight}}

Let us start by defining the necessary terminology on flows and cuts.
Let $\RRp$ be the set of non-negative real numbers, with an additional maximum element $\omega$. We extend the addition by setting $\omega + r = r+ \omega$ for all $r \in \RRp$.

Given a finite directed graph $G = (V, E)$, a capacity function $c : V \to \RRp$, and two vertices $src$ and $tgt$ (a source and a target), we define a \emph{flow} from $src$ to $tgt$ as follows.
It is a function $f : E \to \RRp$ such that for all $v \in V \setminus \set{src,tgt}$, we have 
\[\sum_{(v_-,v) \in E} f(v_-,v) = \sum_{(v,v_+) \in E} f(v,v_+) \leq c(v).\]
The \emph{value} of the flow is defined as $\valueflow_{G,c}(f) = \sum_{(src,v) \in E} f(src,v)$.

A \emph{cut} is a set $M$ of vertices such that every path from $src$ to $tgt$ contains a vertex of $M$. Its \emph{weight} is defined as $\capacity_{G,c}(M) = \sum_{v \in M} c(v)$.

\begin{theorem}[Max flow-min cut~\cite{ford1956maximal}]
	For every graph $G = (V,E)$, capacity function  $c : V \to \RRp$ and vertices $src, tgt \in V$, 
	\[
		\max_{f \text{flow}} \valueflow_{G,c}(f) = \min_{M \text{cut}} \capacity_{G,c}(M)
	\]
\end{theorem}

\begin{remark}
	We state this theorem with the capacities on the vertices, as it is convenient for the next proof.
	It is more commonly stated with capacities on the edges $c : E \to \RRp$.
	The constraints on the flow is then $\sum_{(v_-,v) \in E} f(v_-,v) = \sum_{(v_-,v) \in E} f(v_-,v)$ for all $v \in V$ and $f(e)\leq c(e)$ for all $e \in E$.
	A cut is then defined as a set of edges, and the theorem is stated analogously.
\end{remark}

Another classic result is the \emph{integer flow theorem}. It says that if all capacities in the graph are integers, then there is an integer maximal flow $f : E \to \Nb$.
This is  a by-product of the Ford-Fulkerson algorithm.

\begin{theorem}[Integer flow theorem~\cite{ford1956maximal}]
	For every graph $G = (V,E)$, capacity function  $c : V \to \Nb$ and vertices $src, tgt \in V$, 
	there exists a flow $f : E \to \RRp$ of maximal value such that $f(u,v) \in \Nb$ for all $(u,v) \in E$.
\end{theorem}

With those two results in mind, we can establish the following lemma.

\FunnelingTheHerd*

\begin{proof}
	As $\arena$ is a "population arena tracking $T_f$", by definition it is a finite union of "commit ideals tracking $T_f$", hence we can decompose it as \[\arena = \bigcup_{j=1}^k \trackcommit{\gamma_j}{\vectv_j}{a_j}.\]
	Let $B<\omega$ be the largest constant used to define those "commit ideals". That is, $B = \max (\set{\vectv_j(s) \mid 1 \leq j \leq k, s \in S, \vectv_j < \omega})$.
	
	By definition, $\BoundKT{\arena}{0}{T_f}$ is also a finite union of "commit ideals tracking $T_f$". Let $\trackcommit{\gamma_f}{\vectv}{a}$ with $\vectv \in \set{0, \omega}^S$ be one of them.
	Let $S_\omega$ be the set of states $s$ such that $\vectv(s) = \omega$ and let $d=\card{S_\omega}$. 
	For all $N \in \NN$, define $\trackconf{\gamma_f}{\vectv}[N]$ as a configuration obtained by taking $\gamma_f$ and adding $N$ tokens on each state of $S_\omega$. 
	The set of configurations $\gamma_0$ which satisfy the conclusion of the theorem is clearly stable by token renaming and token removal (because the set of paths in $\BoundKT{\arena}{1}{T_f}$ is stable by token deletion).
	Thus, it is enough to prove it for configurations $\gamma_0$ of the specific form $\trackconf{\gamma_f}{\vectv}[N], N \in \NN$.
	
	Fix $N \in \NN$, we show that a path from $\trackconf{\gamma_f}{\vectv}[N]$ to $\vectF$ in $\BoundKT{\arena}{1}{T_f}$ exists. 
	
	The configuration $\trackconf{\gamma_f}{\vectv}[B \cdot N]$ belongs to the "ideal tracking $T_f$" $\trackideal{\gamma_f}{\vectv}$, which is a subset of $\arena$, hence there is a path $y_0 \xrightarrow{a_1} y_1 \xrightarrow{a_2} \cdots y_n$ within $\arena$ with $y_0 = \trackconf{\gamma_f}{\vectv}[B\cdot N]$ and $y_n \in \vectF$.
	Let $T$ be the set of tokens used in $y_0$.
	
	\begin{figure}[ht]
		\begin{center}
		\input{Figures/fig-lucky-plays.tikz}
		\end{center}
		\caption{An illustration of the proof of Lemma~\ref{lem:onlyone}. 
			We do not describe $\Gamma$ in full, only the two relevant ideals (which by themselves, do not form an  "arena"). We have two actions, $a$ and $b$. We can play $a$ when there are no tokens in $q_2$, $b$ when there are at most $3$, and they let us transfer tokens along the indicated transitions. 
			Those two actions let us transfer arbitrarily many tokens from $q_1$ and $q_3$ to $q_4$. With the notations of the proof, we have $d=2$, $B =3$ and $T_f = \emptyset$. Let $N = 2$.
			The graph $G$ and flow given below represents the path transferring $BN =6$ tokens from $q_1$ and $q_3$ to $q_4$. The corresponding flow has value $dBN=12$.
			On the top right are restricted versions of our ideals, where finite positive numbers have been replaced by $1$: in order to play $b$, we must have at most one token on $q_2$.
			By applying those constraints in $G$, we get a graph where the maximal flow has been divided by at most $3$: by the max flow-min cut theorem and because all capacities have been divided by at most $B=3$. 
			We are thus guaranteed that there is a flow of value $\geq 4$ in this graph (such as the one on the bottom graph), which corresponds to the transfer of $N=2$ tokens from $q_1$ and $q_3$.
		}
	\end{figure}
	
	Consider the following directed graph $G = (V,E)$ and capacity function $c : V \to \Nb$.
	The set of  vertices is \(V = S \times \{0,\ldots, n\} \cup \set{src,tgt} \cup \set{r_s \mid s \in S}\), i.e., one vertex for each state and configuration in the path, plus a source, a target, and one intermediate vertex $r_s$ between $src$ and each $(s,0)$.


	We define the set of edges $E$ and the capacity function $c$ on vertices of $G$ as follows: 
	for all $j \in \{0,\ldots, n-1\}$, we pick a "commit ideal tracking $T_f$" $\trackcommit{\gamma_j}{\vectv_j}{a_{j+1}}$
	in the ideal decomposition of $\arena$ such that $(\gamma_j,a_{j+1})$ is in it.
	We also define $\gamma_n, \vectv_n$ such that $\vectF = \trackideal{\gamma_n}{\vectv_n}$ (they are well-defined since $\vectF$ is an "ideal tracking $T_f$").
	For all $j<n$ and states $s,s'$, there is an edge from $(s,j)$ to $(s', j+1)$ if and only if $\Delta(s,a_{j+1})(s')>0$.
	We also have edges from $src$ to $r_s$, $r_s$ to $(s,0)$ and from $(s,n)$ to $tgt$ for all $s\in S$.
	
	For every state $s$ and index $j$, the vertex $(s,j)$ is assigned capacity $c(s,j) = \vectv_j(s)$.
	Furthermore, for all $s \in S$, $r_s$ has capacity $c(r_s) = B \cdot N$ if $\vectv(s) = \omega$ and $0$ otherwise. 
	Observe that by definition of $B$, we only assigned capacities in $\set{0,\dots, B,\omega}$ to vertices $(s,j)$.

	The trajectories of tokens outside of $T_f$ in the path $y_0 \xrightarrow{a_1} y_1 \xrightarrow{a_2} \cdots \xrightarrow{a_n} y_n$ define a flow $\phi:E\to\+N$ from $src$ to $tgt$ in $G$ of value $dBN$: 
	for each edge $e = ((s,j-1),(s',j))$ we define $\phi(e)$ as the number of tokens of $T \setminus T_f$ going from $s$ to $s'$ at the $j$th step.
	We also define $\phi((src, r_s)) = \phi(r_s, (s,0)) = BN$ for all $s \in S_\omega$ and $0$ otherwise.
	For all $s \in S$, $\phi(((s,n), tgt))$ is the number of tokens of $T \setminus T_f$ in $s$ at the end of the path.
	This flow satisfies the capacity constraints of $G$ by definition of  $(\vectv_j)_{0 \leq j \leq n}$.
	
	Let us now define a new capacity function $c^1$ on $G$, as follows. 
	\begin{itemize}
		\item For all $s \in S$ we set $c^1(r_s) = N$ if $s \in S_\omega$ and $0$ otherwise. 
		
		\item For all other vertices $v \in V \setminus \set{r_s \mid s \in S}$, we set $c^1(v) = 1$ if $c(v) \in \set{1,\dots, B}$ and $c^1(v) = c(v)$ if $c(v) \in \set{0, \omega}$. 
	\end{itemize} 
	
	We claim that $G$ has an integer flow of value $dN$ satisfying the capacity constraint $c^1$. Indeed, suppose the contrary, then by the integer flow theorem all flows have value $<dN$. Therefore, by the max-flow min-cut theorem there is a cut $M$ in $G$ such that $\capacity_{G,c^1}(M)<dN$.
	As $c(v) \leq B\cdot c^1(v)$ for all $v \in V$, that same cut has capacity $\capacity_{G,c}(M)< dBN$ in $G$. By the max-flow min-cut theorem, this contradicts the existence of a flow of value $dBN$ in $(G,c)$.
	We obtain that $(G,c^1)$ has a flow of value $dN$, which is optimal as $\{r_s \mid s\in S\}$ is a cut of value $dN$.
	
	As all capacities defined by $c^1$ are integers, the integer flow theorem guarantees the existence of an optimal integer flow
	$\phi^1:E\to\+N$ of value $dN$, which in turn defines a path of length $n$
	from $\trackconf{\gamma_f}{\vectv}[N]$ to $\vectF$: at step $j$, the tokens of $T_f$ move in the same way as in the step $y_{j-1} \xrightarrow{a_j} y_j$,
	and the number of other tokens sent from state $s$ to $s'$ is $\phi^1((s,j-1),(s',j))$. 
	Note that by definition of $E$ tokens can only move from state $s$ to $s'$ if $\Delta(s,a_j)(s') >0$.

	Observe that the $j$th commit along this path belongs to the "ideal" $\trackcommit{\gamma_{j-1}}{\vectv_{j-1}}{a_j}$
	because the flow $f$ must respect the capacities of vertices $G^1$, which were derived from $\vectv_{j-1}$.
	The path therefore stays in $\arena$.
	In fact, by definition of $c^1$, the $j$th commit belongs to the smaller ideal $\trackcommit{\gamma_{j-1}}{\vectv_{j-1}}{a_j}$, where $\vectv^1_{j-1}$ is $\vectv_{j-1}$ where all finite positive coefficients  have been replaced by $1$.
	Consequently, it belongs to $\BoundKT{\arena}{1}{T_f}$.
	Since all $(r_s)_{s \in S_\omega}$ have capacity $N$, the resulting path starts with $N$ tokens outside $T_f$ in each state of $S_\omega$.
	The path ends in $\vectF$ as the capacities of $(s,n)_{s \in S}$ constrain the final configuration to be in $\trackideal{\gamma_n}{\vectv_n} \subseteq \vectF$.
	This concludes the proof.
\end{proof}

%% file: Figures/fig-lucky-plays.tikz
%
%
%
%
%
%
%
\definecolor{col1}{HTML}{DA6687} 
\definecolor{col2}{HTML}{0E78D5} 
\definecolor{col3}{HTML}{EEC107} 
\definecolor{col4}{HTML}{83D4B9} 
%
%
%
%


	\begin{tikzpicture}[node distance=1.8cm,auto,>= triangle 45, scale=0.5]

		
		\foreach \y in {1,...,4}{
							\pgfmathparse{15-\y};
			\node (q\y) at (0,\pgfmathresult) {$q_{\y}$};
			\foreach \x in {1,...,4}{
				\pgfmathparse{10+\y};
				\node[minimum size=0mm] (n\x\y)  at (2*\x,\pgfmathresult) {};
				\draw[very thick, fill=white] (n\x\y) circle (2mm);
			}	
		
			\foreach \x in {6,...,9}{
				\pgfmathparse{10+\y};
				\node[minimum size=0mm] (n\x\y)  at (2*\x,\pgfmathresult) {};
				\draw[very thick, fill=white] (n\x\y) circle (2mm);
			}	
		}

		\node[color=col2] at (3,15) {$\mathbf{a}$};
		\node[color=col1] at (7,15.1) {$\mathbf{b}$};
		\node[color=col4] at (13,15.1) {$\mathbf{a}$};
		\node[color=col3] at (17,15.1) {$\mathbf{b}$};
		
		
		\node[color=col2,xshift=-10pt] at (n14) {$\omega$};
		\node[color=col2,xshift=-10pt] at (n13) {$0$};
		\node[color=col2,xshift=-10pt] at (n12) {$\omega$};
		\node[color=col2,xshift=-10pt] at (n11) {$\omega$};

		\draw[color=col2,  thick, ->,>=latex ]  (n14) -> (n24);
		\draw[color=col2,  thick, ->,>=latex ]  (n12) -> (n22);
		\draw[color=col2,  thick, ->,>=latex ]  (n11) -> (n21);
		
		\draw[color=col2,  thick, ->,>=latex ]  (n14) -> (n23);
		\draw[color=col2,  thick, ->,>=latex ]  (n12) -> (n23);
		
\node[color=col1,xshift=-10pt] at (n34) {$\omega$};
\node[color=col1,xshift=-10pt] at (n33) {$3$};
\node[color=col1,xshift=-10pt] at (n32) {$\omega$};
\node[color=col1,xshift=-10pt] at (n31) {$\omega$};
		
		\draw[color=col1,  thick, ->,>=latex ]  (n34) -> (n44);
		\draw[color=col1,  thick, ->,>=latex ]  (n32) -> (n42);
		\draw[color=col1,  thick, ->,>=latex ]  (n33) -> (n41);
		\draw[color=col1,  thick, ->,>=latex ]  (n31) -> (n41);

		\node[color=col4,xshift=-10pt] at (n64) {$\omega$};
		\node[color=col4,xshift=-10pt] at (n63) {$0$};
		\node[color=col4,xshift=-10pt] at (n62) {$\omega$};
		\node[color=col4,xshift=-10pt] at (n61) {$\omega$};
		
		\draw[color=col4,  thick, ->,>=latex ]  (n64) -> (n74);
		\draw[color=col4,  thick, ->,>=latex ]  (n62) -> (n72);
		\draw[color=col4,  thick, ->,>=latex ]  (n64) -> (n73);
		\draw[color=col4,  thick, ->,>=latex ]  (n62) -> (n73);
		\draw[color=col4,  thick, ->,>=latex ]  (n61) -> (n71);
		
		\node[color=col3,xshift=-10pt] at (n84) {$\omega$};
		\node[color=col3,xshift=-10pt] at (n83) {$1$};
		\node[color=col3,xshift=-10pt] at (n82) {$\omega$};
		\node[color=col3,xshift=-10pt] at (n81) {$\omega$};
		
		\draw[color=col3,  thick, ->,>=latex ]  (n84) -> (n94);
		\draw[color=col3,  thick, ->,>=latex ]  (n82) -> (n92);
		\draw[color=col3,  thick, ->,>=latex ]  (n83) -> (n91);
		\draw[color=col3,  thick, ->,>=latex ]  (n81) -> (n91);

		\foreach \y in {1,...,4}{
			\foreach \x in {1,...,9}{
				\pgfmathparse{\y+5}
				\node[minimum size=0mm] (n\x-\y)  at (2*\x,\pgfmathresult) {};	
				\draw[very thick, fill=white] (n\x-\y) circle (2mm);			
			}	
		}

	\node[minimum size=0mm] (src)  at (-1.5,7.5) {};	
	\draw[very thick, fill=white] (src) circle (2mm);	
	
	\foreach \y in {1,...,4}{
		\pgfmathparse{\y+5}
		\node[minimum size=0mm] (n0\y)  at (0.5,\pgfmathresult) {};	
		\draw[very thick, fill=white] (n0\y) circle (2mm);		
		\draw[thick, ->,>=latex ]  (n0\y) -> (n1-\y);	
		\draw[thick, ->,>=latex]  (src) -> (n0\y);
	}

	\node[minimum size=0mm] (t)  at (19.5,7.5) {};	
	\draw[very thick, fill=white] (t) circle (2mm);	
	
	\node at (18.95,6.5) {\scriptsize $12$};
	\node at (18.5,7.4) {\scriptsize $0$};
	\node at (18.5,8.1) {\scriptsize $0$};
	\node at (18.8,8.5) {\scriptsize $0$};

	\draw[thick, ->,>=latex ]  (n02) -> (n1-2);
	\draw[thick, ->,>=latex ]  (n9-1) -> (t);
	\draw[thick, ->,>=latex ]  (n9-2) -> (t);
	\draw[thick, ->,>=latex ]  (n9-3) -> (t);
	\draw[thick, ->,>=latex ]  (n9-4) -> (t);
	
	\node at (1.1,9.4) {\scriptsize $6$};
	\node at (1.1,7.3) {\scriptsize $6$};
	\node at (1.1,8.3) {\scriptsize $0$};
	\node at (1.1,6.3) {\scriptsize $0$};
	
	\node at (-0.4,8.6) {\scriptsize $6$};
	\node at (-0.4,7) {\scriptsize $0$};
	\node at (-0.4,8) {\scriptsize $0$};
	\node at (-0.4,7.5) {\scriptsize $6$};

	\draw[color=col2,  thick, ->,>=latex ]  (n1-4) -> (n2-4);
	\draw[color=col2,  thick, ->,>=latex ]  (n1-2) -> (n2-2);
	\draw[color=col2,  thick, ->,>=latex ]  (n1-1) -> (n2-1);
	\draw[color=col2,  thick, ->,>=latex ]  (n1-4) -> (n2-3);
	\draw[color=col2,  thick, ->,>=latex ]  (n1-2) -> (n2-3);
	
	\node at (3,9.3) {\scriptsize $4$};
	\node at (5,9.3) {\scriptsize $4$};
	\node at (7,9.3) {\scriptsize $3$};
	\node at (9,9.3) {\scriptsize $3$};
	\node at (11,9.3) {\scriptsize $1$};
	\node at (13,9.3) {\scriptsize $1$};
	\node at (15,9.3) {\scriptsize $0$};
	\node at (17,9.3) {\scriptsize $0$};
	
	\node at (3,5.7) {\scriptsize $0$};
	\node at (5,5.7) {\scriptsize $0$};
	\node at (7,5.7) {\scriptsize $3$};
	\node at (9,5.7) {\scriptsize $3$};
	\node at (11,5.7) {\scriptsize $6$};
	\node at (13,5.7) {\scriptsize $6$};
	\node at (15,5.7) {\scriptsize $9$};
	\node at (17,5.7) {\scriptsize $9$};
	
	\node at (2.8,6.7) {\scriptsize $5$};
	\node at (4.8,6.7) {\scriptsize $5$};
	\node at (6.8,6.7) {\scriptsize $3$};
	\node at (8.8,6.7) {\scriptsize $3$};
	\node at (10.8,6.7) {\scriptsize $2$};
	\node at (12.8,6.7) {\scriptsize $2$};
	\node at (14.8,6.7) {\scriptsize $0$};
	\node at (16.8,6.7) {\scriptsize $0$};
	
	\node at (2.8,8.3) {\scriptsize $2$};
	\node at (6.8,8.3) {\scriptsize $1$};
	\node at (10.8,8.3) {\scriptsize $2$};
	\node at (14.8,8.3) {\scriptsize $1$};
	
	\node at (2.8,7.7) {\scriptsize $1$};
	\node at (6.8,7.7) {\scriptsize $2$};
	\node at (10.8,7.7) {\scriptsize $1$};
	\node at (14.8,7.7) {\scriptsize $2$};
	
	\node at (4.8,7.7) {\scriptsize $3$};
	\node at (8.8,7.7) {\scriptsize $3$};
	\node at (12.8,7.7) {\scriptsize $3$};
	\node at (16.8,7.7) {\scriptsize $3$};
	
	\draw[color=col1,  thick, ->,>=latex ]  (n2-4) -> (n3-4);
	\draw[color=col1,  thick, ->,>=latex ]  (n2-2) -> (n3-2);
	\draw[color=col1,  thick, ->,>=latex ]  (n2-3) -> (n3-1);
	\draw[color=col1,  thick, ->,>=latex ]  (n2-1) -> (n3-1);

	\draw[color=col2,  thick, ->,>=latex ]  (n3-4) -> (n4-4);
	\draw[color=col2,  thick, ->,>=latex ]  (n3-2) -> (n4-2);
	\draw[color=col2,  thick, ->,>=latex ]  (n3-1) -> (n4-1);
	\draw[color=col2,  thick, ->,>=latex ]  (n3-4) -> (n4-3);
	\draw[color=col2,  thick, ->,>=latex ]  (n3-2) -> (n4-3);
	
	\draw[color=col1,  thick, ->,>=latex ]  (n4-4) -> (n5-4);
	\draw[color=col1,  thick, ->,>=latex ]  (n4-2) -> (n5-2);
	\draw[color=col1,  thick, ->,>=latex ]  (n4-3) -> (n5-1);
	\draw[color=col1,  thick, ->,>=latex ]  (n4-1) -> (n5-1);
	
	\draw[color=col2,  thick, ->,>=latex ]  (n5-4) -> (n6-4);
	\draw[color=col2,  thick, ->,>=latex ]  (n5-2) -> (n6-2);
	\draw[color=col2,  thick, ->,>=latex ]  (n5-1) -> (n6-1);
	\draw[color=col2,  thick, ->,>=latex ]  (n5-4) -> (n6-3);
	\draw[color=col2,  thick, ->,>=latex ]  (n5-2) -> (n6-3);

	\draw[color=col1,  thick, ->,>=latex ]  (n6-4) -> (n7-4);
	\draw[color=col1,  thick, ->,>=latex ]  (n6-2) -> (n7-2);
	\draw[color=col1,  thick, ->,>=latex ]  (n6-3) -> (n7-1);
	\draw[color=col1,  thick, ->,>=latex ]  (n6-1) -> (n7-1);
	
	\draw[color=col2,  thick, ->,>=latex ]  (n7-4) -> (n8-4);
	\draw[color=col2,  thick, ->,>=latex ]  (n7-2) -> (n8-2);
	\draw[color=col2,  thick, ->,>=latex ]  (n7-1) -> (n8-1);
	\draw[color=col2,  thick, ->,>=latex ]  (n7-4) -> (n8-3);
	\draw[color=col2,  thick, ->,>=latex ]  (n7-2) -> (n8-3);
	
	\draw[color=col1,  thick, ->,>=latex ]  (n8-4) -> (n9-4);
	\draw[color=col1,  thick, ->,>=latex ]  (n8-2) -> (n9-2);
	\draw[color=col1,  thick, ->,>=latex ]  (n8-3) -> (n9-1);
	\draw[color=col1,  thick, ->,>=latex ]  (n8-1) -> (n9-1);


\foreach \y in {1,...,4}{
	\pgfmathparse{5-\y}
	\foreach \x in {1,...,9}{
		\node[minimum size=0mm] (n\x-\y)  at (2*\x,\y) {};	
		\draw[very thick, fill=white] (n\x-\y) circle (2mm);			
	}	
}

\node[minimum size=0mm] (src)  at (-1.5,2.5) {};	
\draw[very thick, fill=white] (src) circle (2mm);	

\foreach \y in {1,...,4}{
	\node[minimum size=0mm] (n0\y)  at (0.5,\y) {};	
	\draw[very thick, fill=white] (n0\y) circle (2mm);		
	\draw[thick, ->,>=latex ]  (n0\y) -> (n1-\y);	
	\draw[thick, ->,>=latex]  (src) -> (n0\y);
}

\node[minimum size=0mm] (t)  at (19.5,2.5) {};	
\draw[very thick, fill=white] (t) circle (2mm);	

\node at (18.95,1.5) {\scriptsize $4$};
\node at (18.5,2.4) {\scriptsize $0$};
\node at (18.5,3.1) {\scriptsize $0$};
\node at (18.8,3.5) {\scriptsize $0$};

\draw[thick, ->,>=latex ]  (n02) -> (n1-2);
\draw[thick, ->,>=latex ]  (n9-1) -> (t);
\draw[thick, ->,>=latex ]  (n9-2) -> (t);
\draw[thick, ->,>=latex ]  (n9-3) -> (t);
\draw[thick, ->,>=latex ]  (n9-4) -> (t);

\node at (1.1,4.4) {\scriptsize $2$};
\node at (1.1,2.3) {\scriptsize $2$};
\node at (1.1,3.3) {\scriptsize $0$};
\node at (1.1,1.3) {\scriptsize $0$};

\node at (-0.4,3.6) {\scriptsize $2$};
\node at (-0.4,2) {\scriptsize $0$};
\node at (-0.4,3) {\scriptsize $0$};
\node at (-0.4,2.5) {\scriptsize $2$};

\draw[color=col4,  thick, ->,>=latex ]  (n1-4) -> (n2-4);
\draw[color=col4,  thick, ->,>=latex ]  (n1-2) -> (n2-2);
\draw[color=col4,  thick, ->,>=latex ]  (n1-1) -> (n2-1);
\draw[color=col4,  thick, ->,>=latex ]  (n1-4) -> (n2-3);
\draw[color=col4,  thick, ->,>=latex ]  (n1-2) -> (n2-3);

\node at (3,4.3) {\scriptsize $1$};
\node at (5,4.3) {\scriptsize $1$};
\node at (7,4.3) {\scriptsize $1$};
\node at (9,4.3) {\scriptsize $1$};
\node at (11,4.3) {\scriptsize $0$};
\node at (13,4.3) {\scriptsize $0$};
\node at (15,4.3) {\scriptsize $0$};
\node at (17,4.3) {\scriptsize $0$};

\node at (3,0.7) {\scriptsize $0$};
\node at (5,0.7) {\scriptsize $0$};
\node at (7,0.7) {\scriptsize $1$};
\node at (9,0.7) {\scriptsize $1$};
\node at (11,0.7) {\scriptsize $2$};
\node at (13,0.7) {\scriptsize $2$};
\node at (15,0.7) {\scriptsize $3$};
\node at (17,0.7) {\scriptsize $3$};

\node at (2.8,1.7) {\scriptsize $2$};
\node at (4.8,1.7) {\scriptsize $2$};
\node at (6.8,1.7) {\scriptsize $1$};
\node at (8.8,1.7) {\scriptsize $1$};
\node at (10.8,1.7) {\scriptsize $1$};
\node at (12.8,1.7) {\scriptsize $1$};
\node at (14.8,1.7) {\scriptsize $0$};
\node at (16.8,1.7) {\scriptsize $0$};

\node at (2.8,3.3) {\scriptsize $1$};
\node at (6.8,3.3) {\scriptsize $0$};
\node at (10.8,3.3) {\scriptsize $1$};
\node at (14.8,3.3) {\scriptsize $0$};

\node at (2.8,2.7) {\scriptsize $0$};
\node at (6.8,2.7) {\scriptsize $1$};
\node at (10.8,2.7) {\scriptsize $0$};
\node at (14.8,2.7) {\scriptsize $1$};

\node at (4.8,2.7) {\scriptsize $1$};
\node at (8.8,2.7) {\scriptsize $1$};
\node at (12.8,2.7) {\scriptsize $1$};
\node at (16.8,2.7) {\scriptsize $1$};

\draw[color=col3,  thick, ->,>=latex ]  (n2-4) -> (n3-4);
\draw[color=col3,  thick, ->,>=latex ]  (n2-2) -> (n3-2);
\draw[color=col3,  thick, ->,>=latex ]  (n2-3) -> (n3-1);
\draw[color=col3,  thick, ->,>=latex ]  (n2-1) -> (n3-1);

\draw[color=col4,  thick, ->,>=latex ]  (n3-4) -> (n4-4);
\draw[color=col4,  thick, ->,>=latex ]  (n3-2) -> (n4-2);
\draw[color=col4,  thick, ->,>=latex ]  (n3-1) -> (n4-1);
\draw[color=col4,  thick, ->,>=latex ]  (n3-4) -> (n4-3);
\draw[color=col4,  thick, ->,>=latex ]  (n3-2) -> (n4-3);

\draw[color=col3,  thick, ->,>=latex ]  (n4-4) -> (n5-4);
\draw[color=col3,  thick, ->,>=latex ]  (n4-2) -> (n5-2);
\draw[color=col3,  thick, ->,>=latex ]  (n4-3) -> (n5-1);
\draw[color=col3,  thick, ->,>=latex ]  (n4-1) -> (n5-1);

\draw[color=col4,  thick, ->,>=latex ]  (n5-4) -> (n6-4);
\draw[color=col4,  thick, ->,>=latex ]  (n5-2) -> (n6-2);
\draw[color=col4,  thick, ->,>=latex ]  (n5-1) -> (n6-1);
\draw[color=col4,  thick, ->,>=latex ]  (n5-4) -> (n6-3);
\draw[color=col4,  thick, ->,>=latex ]  (n5-2) -> (n6-3);

\draw[color=col3,  thick, ->,>=latex ]  (n6-4) -> (n7-4);
\draw[color=col3,  thick, ->,>=latex ]  (n6-2) -> (n7-2);
\draw[color=col3,  thick, ->,>=latex ]  (n6-3) -> (n7-1);
\draw[color=col3,  thick, ->,>=latex ]  (n6-1) -> (n7-1);

\draw[color=col4,  thick, ->,>=latex ]  (n7-4) -> (n8-4);
\draw[color=col4,  thick, ->,>=latex ]  (n7-2) -> (n8-2);
\draw[color=col4,  thick, ->,>=latex ]  (n7-1) -> (n8-1);
\draw[color=col4,  thick, ->,>=latex ]  (n7-4) -> (n8-3);
\draw[color=col4,  thick, ->,>=latex ]  (n7-2) -> (n8-3);

\draw[color=col3,  thick, ->,>=latex ]  (n8-4) -> (n9-4);
\draw[color=col3,  thick, ->,>=latex ]  (n8-2) -> (n9-2);
\draw[color=col3,  thick, ->,>=latex ]  (n8-3) -> (n9-1);
\draw[color=col3,  thick, ->,>=latex ]  (n8-1) -> (n9-1);

	\end{tikzpicture}	
%
%

%% file: proof-isolation.tex
	
In this section we prove the following lemma.

\IsolationLemma*
~

As $\arena$ is a "population arena", it is a finite union of "commit ideals". As a consequence, there exists a bound $B$ such that  $\arena =\BoundK{\arena}{B}$; $B$ is the largest finite number appearing in the description of $\arena$.
Let $T$ be the set of tokens of $\gamma_{init}$.

Denote $T_\omega = T \setminus T_f$,
and $S_\omega$ the states occupied in $\gamma_{init}$ by tokens in $T_\omega$ (i.e., $S_\omega=\gamma_{init}(T_\omega)$).
By hypothesis, $S_\omega$ is an "$\omega$-base" of $\gamma_{init}$ in $\arena$.
Let $M \in \NN$, and $\gamma_M$ be a configuration obtained from $\gamma_{init}$ by adding $M$ tokens for each token in $T_\omega$, on the state of that token. A state that contains $K$ tokens of $T_\omega$ in $\gamma_{init}$ contains those tokens plus $MK$ extra tokens in $\gamma_M$.
We denote $T_M$ the set of all these new tokens (i.e. the tokens of $\gamma_M$ which are not in $\gamma_{init}$).
In total the set of tokens of $\gamma_M$ is $T_f \cup T_\omega \cup T_M$. 
Since we only added tokens on states of $S_\omega$, we have $\gamma_M \in \arena$.

We say a token $t$ is \intro{saved} in a configuration if the singleton $\{t\}$ is an "$\omega$-base" of that configuration, i.e., if we can add as many tokens as we want on the state occupied by $t$ without leaving $\arena$.
By definition of $B$, if in a configuration of $\arena$ a token $t$ occupied a state on which there are at least $B$ other tokens,
 then $t$ is necessarily "saved" in this configuration.
Observe that, as a consequence, in a configuration of $\arena$ there can be at most $B|S|$ tokens which are not "saved".

Define the following sets of configurations:
\begin{itemize}
	\item $X$ be the set of configurations with a "finite base" of size $< k$
	
	\item $Z$ the set of configurations where one of the tokens in $T_f$
	is "saved".
\end{itemize}

\newcommand{\srw}{\sigma_{rw}}

We show first that a safe random walk in $\arena$ guarantees almost surely to reach $Z$ when starting in $\gamma_M$.
Observe that a token that is not "saved" has to be part of every finite base.
Therefore we have the inclusion $X \subseteq Z$.
Note that this inclusion may be strict. Indeed, it might be the case that there is a "finite base" not containing $t$ but containing some tokens of $T_\omega \cup T_M$ (albeit, initially, all tokens in $T_\omega \cup T_M$ belong to an "$\omega$-base" and are "saved", but this might not be the case all along the path, i.e. some of them might ``leave the cohort'').
By hypothesis of the lemma, from every configuration in $\arena$ there is a strategy to reach $X$ without leaving $\arena$.
Let $\srw$ be a "safe random walk" in $\arena$. 
A consequence is that $\srw$ guarantees to reach $X$ almost surely from everywhere in $\arena$.
In particular, $\srw$ guarantees to reach $Z$ almost surely from $\gamma_M$.

\begin{lemma}
	For all $\gamma_0 \in \arena$, for all token $t$ in a finite base of $\gamma_0$, the probability, when applying $\srw$, that $t$ "meets" $\geq B|S|$ tokens while not being "saved" is at most $\frac{(B|S|)!}{(B|S|)!+1}$.	
\end{lemma}

\begin{proof}
		\AP Recall that we say that two tokens $t_1,t_2\in T$ 
	\reintro{meet} in a configuration $\gamma \in S^T$
	if they are placed on the same state i.e. if $\gamma(t_1)=\gamma(t_2)$.
	
	Let $\pi = \gamma_0  \xrightarrow{a_1}  \dots \xrightarrow{a_k} \gamma_k$ be a "path" from $\gamma_0$, and $t \in T$. 
	\AP The ""heard-of sequence"" $H(\pi, t)$ is the sequence of sets of tokens $H_0, \dots, H_k$, defined as follows:
	\begin{itemize}
		\item $H_0 = \set{t' \in  T \mid \gamma_0(t') = \gamma_0(t)}$.
		\item $H_{i} = \set{t' \in T \mid \exists t'' \in H_{i-1}, \gamma_i(t') = \gamma_i(t'')}$
	\end{itemize} 
	It indicates, at each step, the set of tokens which have met a token which has met a token ... which has met $t$.
	We say that a token $t'$ ""hears of"" $t$ in a configuration $\gamma_i$ if $t' \in H_i \setminus H_{i-1}$. If $t' \in H_i$ we say $t_i$ has heard of $t$.
	
	Let $\gamma_0$ be a configuration, $T_0$ a "finite base" of it, and $t \in T_0$.  
	Our first goal is to show an upper bound on the probability that a token which has "heard of" $t$ is "saved" before $t$ itself.
	Let $\PP_{\srw, \gamma_0}$ be the probability distribution induced by $\srw$ on the paths from $\gamma_0$ in $\Gamma$.
	
	\newcommand{\gb}{\overline{\gamma}}
	
	We define a transformation on paths.
	Given $\pi = \gamma_0 \to \dots \to \gamma_k$ a path over a set of tokens $T$, 
	and $t, t'\in T$ such that $t$ and $t'$ "meet" in some $\gamma_j$,
	we define a new path $\pi^{t \leftrightarrow t'}$ as follows.
	In words, $\pi^{t \leftrightarrow t'}$ is the run obtained from $\pi$ by switching the trajectories of $t$ and $t'$ after their first meeting. Formally, 
let $j$ be the minimal index such that $\gamma_j(t) = \gamma_j(t')$.
	Let $\tau_{t,t'}: T \to T$ be the function mapping $t$ to $t'$, $t'$ to $t$ and all other tokens to themselves, i.e., the transposition swapping $t$ and $t'$.
	We set $\pi^{t \leftrightarrow t'} = \gb_0 \to \dots \to \gb_k$
where $\gb_i = \gamma_i$ for all $i\leq j$, and $\gb_i = \gamma_i \circ \tau_{t,t'}$ for all $i > j$.

	\begin{claim}
		\label{claim:heard-of-facts}
		For all finite path $\pi$ from $\gamma_0$, $t' \in T$ which "meets" $t$ at some point in $\pi$:
		\begin{enumerate}
			\item $(\pi^{t \leftrightarrow t'})^{t \leftrightarrow t'} = \pi$
			
			\item $\PP_{\srw,\gamma_0}(\pi^{t \leftrightarrow t'}) = \PP_{\srw,\gamma_0}(\pi)$
			
			\item $H(\pi^{t \leftrightarrow t'},t) = H(\pi, t)$
		\end{enumerate}
	\end{claim}
	
	\begin{claimproof}
		\begin{enumerate}
			\item It suffices to observe that if $t$ and $t'$ first "meet" in the $j$th configuration of $\pi$, then they also first "meet" in the $j$th configuration of $\pi^{t \leftrightarrow t'}$.
			
			\item Since $\srw$ is a "safe random walk" in $\arena$, and $\arena$ is invariant under renaming tokens, the probability to go from $\gamma_i \circ \tau_{t,t'}$ to $\gamma_{i+1} \circ \tau_{t,t'}$ in a step is the same as the probability to go from $\gamma_i$ to $\gamma_{i+1}$.
			Note that if $\gamma_i(t) = \gamma_i(t')$ then $\gamma_i = \gamma_i \circ \tau_{t,t'}$.
			A simple induction on the run length shows that $\PP_{\srw,\gamma_0}(\pi^{t \leftrightarrow t'}) = \PP_{\srw,\gamma_0}(\pi)$.
			
			\item This follows from a simple induction on the run length.\claimqedhere
		\end{enumerate}
	\end{claimproof}

	Let $p$ be the probability that a token which has "heard of" $t$ is "saved".
	Let $\Pi$ be the set of finite runs such that a token which has "heard of" $t$ is "saved" in the last configuration, and no token having "heard of" $t$ is saved beforehand.
	Let $\Pi_t$ be the subset of $\Pi$ such that $t$ is "saved" in the last configuration, and $\Pi' = \Pi \setminus \Pi_t$. We obtain
	
	\[ p = \sum_{\pi \in \Pi} \PP_{\srw,\gamma_0}(\pi) = \sum_{\pi \in \Pi_t} \PP_{\srw,\gamma_0}(\pi) + \sum_{\pi \in \Pi'} \PP_{\srw,\gamma_0}(\pi).\]
	

	\begin{claim}
		\label{claim:switches}
		Let $\pi \in \Pi$, and $\pi^-$ the path obtained by removing its last step.
		Let $TH$ be the set of tokens which "hear of" $t$ in $\pi^-$.
		Then $|TH| \leq B|S|$.
		
		Furthermore, there exist distinct tokens $t_1, \dots, t_k \in TH$ and runs $\pi_0, \dots, \pi_k$ such that $\pi = \pi_0$, $\pi_k \in \Pi_t$, and for all $i$ we have $(\pi_i)^{t \leftrightarrow t_i} = \pi_{i+1}$.
	\end{claim}

	\begin{claimproof}
		We first prove that $|TH| \leq B|S|$.
		By definition of $\Pi$, no token which has "heard of" $t$ is "saved" in $\pi^-$.
		In particular,  in the last configuration of $\pi_-$ no token which has "heard of" $t$ is "saved".  
		By definition of $B$, at most $B|S|$ tokens can be not "saved" in a configuration.
		Therefore, $|TH| \leq B|S|$.

		For the second part of the claim, 
		Define $\pi = \gamma_1 \to \cdots \to \gamma_m$.
		By definition of $\Pi$, there is a token $t'$ which "hears of" $t$ in $\pi$ and is "saved" in the last configuration.
		We can assume without loss of generality that $t'$ "hears of" $t$ in $\pi^-$ (i.e.,$t' \in TH$): a token which "hears of" $t$ only in the last configuration must be in the same state as a token which has "heard of" $t$ before.
		Therefore, there is a sequence of distinct tokens $t_0, \dots, t_k \in TH$ and indices $i_1 < \dots < i_k$ such that $k \leq B|S|$, $t = t_0$, $t_k = t'$ and for all $j$, $t_{j-1}$ and $t_{j}$ "meet" in $\gamma_{i_j}$.

		We define $\pi_0 = \pi$, and for all $j < k$, $\pi_{j+1} = (\pi_j)^{t \leftrightarrow t_j}$.
		An easy induction on $j$ shows that $t$ and $t_{j}$ "meet" in the $i_j$th configuration of $\pi_j$. 
		Therefore, all $\pi_j$ are well-defined.
		Moreover, in $\pi_k$ token $t$ must eventually follow the trajectory of $t_k$ in $\pi$, and is therefore saved at the end.
		The "heard-of sequence" remains the same by Claim~\ref{claim:heard-of-facts}.
		Since
		 all we did was switch trajectories of tokens that are not "saved" before the last configuration, no token which has heard of $t$ is "saved" before the last configuration in $\pi_k$.
		 As a result, $\pi_k \in \Pi'$.
	\end{claimproof}

	\begin{figure}[ht]
		\centering
		\input{Figures/proof-isolation.tikz}
		\caption{A visual of the proof of Claim~\ref{claim:switches}. 
			Thin lines describe the trajectory of single tokens: the blue one is the trajectory of a token $t$, and others the ones of tokens which hear of $t$ at some point along the path. The wide gray line describes the movement of the cohort. Tokens which are not in the cohort but do not hear of $t$ are ignored.
			Lines of other tokens are dashed before the token hears of $t$, and full afterwards.}
		\centering
	\end{figure}

	\begin{claim}
		$(B|S|)(B|S|)! \sum_{\pi \in \Pi_t} \PP_{\srw,\gamma_0}(\pi) \geq  \sum_{\pi' \in \Pi'} \PP_{\srw,\gamma_0}(\pi')$. 
	\end{claim}
	
	\begin{claimproof}
		Let $P_\Pi$ be the countable set $\set{r \in ]0,1] \mid \exists \pi \in \Pi, \PP_{\srw,\gamma_0}(\pi) = r}$.
		We have $\sum_{\pi' \in \Pi'} \PP_{\srw,\gamma_0}(\pi') = \sum_{r \in P_\Pi} r \cdot |\set{\pi' \in \Pi' \mid \PP_{\srw,\gamma_0}(\pi') = r}|$.
		
		Let $r \in P_\Pi$.
		Consider the set $X_r$ of tuples $(\pi, t_1, \ldots, t_k)$ where $\pi$ is a path of $\Pi_t$ with $\PP_{\srw,\gamma_0}(\pi)=r$ and $t_1, \ldots, t_k$ is a list of distinct tokens which "hear of" $t$ in $\pi$ before the last configuration. 
		Note that $X_r$ is a finite set, since the events $(\pi)_{\pi \in \Pi}$ are mutually exclusive and $r>0$.
		
		By Claim~\ref{claim:switches}, each path of $\Pi$ (in particular of $\Pi_t$) has at most $B|S|$ such tokens. As a consequence, for a given $\pi$ there are at most $(B|S|)(B|S|)!$ possibilities for $t_1, \ldots, t_k$ (choose $k$ and then choose the sequence). Hence $|X_r| \leq |\set{\pi \in \Pi_t \mid \PP_{\srw,\gamma_0}(\pi)=r}|(B|S|)(B|S|)!$.
		
		By Claim~\ref{claim:switches}, every path $\pi \in \Pi$ can be turned into a path of $\Pi_t$ through a sequence of switches of $t$ with distinct tokens which "hear of" $t$ in $\pi$ before the last configuration. 
		By Claim~\ref{claim:heard-of-facts}, this means that every path $\pi \in \Pi$ can be obtained from a path $\pi_t \in \Pi_t$ through a sequence of switches of $t$ with distinct tokens which "hear of" $t$ before the last configuration in $\pi_t$. Furthermore, $\PP(\pi_t) = \PP(\pi)$.
		
		This gives us a natural surjection from $X_r$ to $\set{\pi' \in \Pi' \mid \PP_{\srw,\gamma_0}(\pi') = r}$, meaning that $|\set{\pi' \in \Pi' \mid \PP_{\srw,\gamma_0}(\pi') = r}| \leq |X_r| \leq |\set{\pi \in \Pi_t \mid \PP_{\srw,\gamma_0}(\pi)=r}|(B|S|)(B|S|)!$. 
		
		Finally, from the equality at the beginning of the proof we get
		\begin{align*}
			\sum_{\pi' \in \Pi'} \PP_{\srw,\gamma_0}(\pi') 
			&\leq \sum_{r \in P_\Pi} r \cdot \abs{\set{\pi \in \Pi_t : \PP_{\srw,\gamma_0}(\pi)=r}} (B\abs{S})(B\abs{S})!\\
			&= (B|S|)(B|S|)! \sum_{\pi \in \Pi_t} \PP_{\srw,\gamma_0}(\pi).
		\end{align*}
			\claimqedhere
	\end{claimproof}
	The previous claim implies that
	\[ p \geq \frac{(B|S|)!+1}{(B|S|)!}\sum_{\pi \in \Pi'} \PP_{\srw,\gamma_0}(\pi).\]
	Now, since $p \leq 1$, we obtain $\sum_{\pi \in \Pi'} \PP_{\srw,\gamma_0}(\pi)\leq \frac{(B|S|)(B|S|)!}{(B|S|)(B|S|)!+1}$.
	
	As noted before, the number of tokens that are not saved in a configuration of $\arena$ is at most $B|S|$.
	Hence, if in a finite run $\pi$ we have that $t$ "meets" $B|S|$ distinct tokens and is not "saved" then one of them must be saved at some point after meeting $t$.
	Thus $\pi$ must have a prefix in $\Pi'$.
\end{proof}

The safe random walk $\srw$ is positional, i.e., the probability distribution chosen in a configuration is independent from what happened before in the path.
For all $\ell\geq 1$, the probability for $t$ to "meet"  $\geq B|S| \ell$ tokens is at most the probability to meet $B|S|$ tokens, then $B|S|$ others, and so on $\ell$ times.
An easy induction then shows that the probability for $t$ to "meet" more than $B|S|\ell$ tokens is bounded by  $\left(\frac{(B|S|)(B|S|)!}{(B|S|)(B|S|)!+1}\right)^\ell$.
We infer that the expected number of different tokens met by $t$ is at most \[\sum_{\ell \in \NN} B\cdot |S| \cdot (\ell+1) \cdot \left(\frac{(B|S|)(B|S|)!}{(B|S|)(B|S|)!+1}\right)^\ell= B\cdot |S| \cdot ((B|S|)(B|S|)!+1)^2.\]

	We will now bound the probability that a token of $T_\omega$ either meets a token of $T_f$ before we reach $Z$, or is in a state with $\leq B$ tokens at the end.
	We say that a token that does one of those two things ""misbehaves"".
	By definition of $B$, we have $|T_f|\leq B$.
	By linearity of the expected value, the number of tokens met by at least one token of $T_f$ is at most $ B^2\cdot |S|^2 \cdot ((B|S|)!+1)^2$.
	Furthermore, when $Z$ is reached the number of tokens which are in states with $\leq B$ tokens at the end is bounded by $B|S|$.
	Hence the expected number of "misbehaved" tokens is at most $B^2 |S|  ((B|S|)!+1)^2 + B|S|$.

	Let $h$ be a token of $T_M$, we can then bound its probability to "misbehave".			
	As $h$ starts in the same state as $M$ tokens of $T_M$, and $\srw$ treats all those tokens symmetrically, 
	they all have the same probability $p_{mis}$ of "misbehaving".
	By linearity of the expected value, the expected number of them which "misbehave" is $(M+1) \cdot p_{mis}$.
	On the other hand, this number is bounded by the expected total number of tokens "misbehaving".
	We obtain $M p_{mis} \leq B^2 |S|^2  ((B|S|)!+1)^2 + B|S|$ and thus $p_{mis} \leq \frac{B^2 |S|^2  ((B|S|)!+1)^2 + B|S|}{M}$. 
	As a result, the probability that a token $h \in T_\omega$ misbehaves converges to $0$ as $M$ grows.

	Recall that $T_\omega$ denotes the set of tokens in $S_\omega$ in $w_0$,
	and $T_M$ the set of additional tokens in $S_\omega$ in  $w_M$.
	The "safe random walk" $\srw$ from $w_M$ can be projected onto a strategy from $w_0$ by simulating the additional tokens in $T_M$.
	The resulting strategy is randomized, and uses a finite memory in order to store the 
	positions of all simulated tokens in $T_M$. Note that this strategy still guarantees to reach $X$, and thus $Z$, almost surely.
	By choosing $M$ large enough,
	the probability that a token in $T_\omega$ "misbehaves" can be made arbitrarily small.

	Since the set of configurations reachable from $w_0$ is finite,
	and the corresponding condition is a reachability condition under safety constraint, 
	the probability can be turned to $0$: 
	there is a memoryless strategy that achieves a probability $1$ of reaching $Z$~\cite{Ornstein:AMS1969} and a probability $0$ of a token in $T_\omega$ "misbehaving". 
	
	Reaching $Z$ with no token "misbehaving" implies that each token of $T_\omega$ is in a state with at least $B$ other tokens.
	As a consequence, it cannot be part of a "finite base" of the current configuration.
	Since we are in $Z$, any "finite base" of this configuration has to be a strict subset of $T_f$.
	
	We have shown that there exists a strategy to reach a configuration with a "finite base" strictly included in $T_f$ almost surely, while making sure that no token of $T_f$ meets a token outside $T_f$ beforehand.

%% file: Figures/proof-isolation.tikz
		\begin{tikzpicture}[node distance=1.8cm,auto,>= triangle 45, scale=0.5]

		
			\foreach \y in {1,...,5}{
				\node[label] (q\y) at (17.2,\y) {$q_{\y}$};
			\foreach \x in {1,...,6}{
				\pgfmathparse{\x+8}
				\node[minimum size=0mm] (n\x\y)  at (2*\pgfmathresult,\y) {};
				%
			}	
		}
		
			\fill[fill=gray!20!white] (n14.north) -- (n23.north) -- (n23.south) -- (n14.south); 
		
		\fill[fill=gray!20!white]  (n33.north) --  (n33.south) -- (n23.south)   -- (n23.north); 
		\fill[fill=gray!20!white]  (n35.north)  -- (n35.south) -- (n23.south) -- (n23.north) ; 
		
		\fill[fill=gray!20!white]  (n33.north) --  (n33.south) -- (n41.south)   -- (n41.north); 
		\fill[fill=gray!20!white]  (n35.north)  -- (n35.south) -- (n45.south) -- (n45.north) ;
		\fill[fill=gray!20!white]  (n33.north) --  (n33.south) -- (n42.south)   -- (n42.north); 
		
		\fill[fill=gray!20!white]  (n51.north) --  (n51.south) -- (n41.south)   -- (n41.north); 
		\fill[fill=gray!20!white]  (n55.north)  -- (n55.south) -- (n45.south) -- (n45.north) ; 
		\fill[fill=gray!20!white]  (n51.north) --  (n51.south) -- (n42.south)   -- (n42.north); 
		
		\fill[fill=gray!20!white]  (n51.north) --  (n51.south) -- (n61.south)   -- (n61.north); 
		\fill[fill=gray!20!white]  (n55.north)  -- (n55.south) -- (n64.south) -- (n64.north) ;

		\foreach \y in {1,...,5}{
			\foreach \x in {9,...,14}{
				\draw[very thick, fill=white] (2*\x,\y) circle (2mm);
			}	
		}


			\draw[color=col2, ultra thick] (n12) -- (n21) -- (n32) -- (n44) -- (n52) -- (n61);
		
		\draw[loosely dashed, color=col4, ultra thick] (n14) -- (n22) -- (n32);
		\draw[color=col4, ultra thick] (n32) -- (n43) -- (n53) -- (n65);
		
		\draw[color=col1, ultra thick] (n44) -- (n54) -- (n63);
		\draw[loosely dashed, color=col1, ultra thick] (n13) -- (n24) -- (n35) -- (n44);
		
		\draw[color=col3, ultra thick] (n52) -- (n62);
		\draw[loosely dashed, color=col3, ultra thick] (n15) -- (n24) -- (n34) -- (n42) -- (n52) ;

		\node (arr) at (14.5,10) {\rotatebox{330}{\Huge $\Rightarrow$}};
		\node (lab) at (14.5,11) { \rotatebox{330}{\color{col2} $\bullet$ \color{black} $\leftrightarrow$ \color{col4} $\bullet$}};
		
		
		\foreach \y in {1,...,5}{
			\pgfmathparse{\y+9}
			\node[label] (q\y) at (1.2,\pgfmathresult) {$q_{\y}$};
			\foreach \x in {1,...,6}{
				\node[minimum size=0mm] (n\x\y)  at (2*\x,\pgfmathresult) {};
%
			}	
		}

		\fill[fill=gray!20!white] (n14.north) -- (n23.north) -- (n23.south) -- (n14.south); 
		
		\fill[fill=gray!20!white]  (n33.north) --  (n33.south) -- (n23.south)   -- (n23.north); 
		\fill[fill=gray!20!white]  (n35.north)  -- (n35.south) -- (n23.south) -- (n23.north) ; 
		
		\fill[fill=gray!20!white]  (n33.north) --  (n33.south) -- (n41.south)   -- (n41.north); 
		\fill[fill=gray!20!white]  (n35.north)  -- (n35.south) -- (n45.south) -- (n45.north) ;
		\fill[fill=gray!20!white]  (n33.north) --  (n33.south) -- (n42.south)   -- (n42.north); 
		
		\fill[fill=gray!20!white]  (n51.north) --  (n51.south) -- (n41.south)   -- (n41.north); 
		\fill[fill=gray!20!white]  (n55.north)  -- (n55.south) -- (n45.south) -- (n45.north) ; 
		\fill[fill=gray!20!white]  (n51.north) --  (n51.south) -- (n42.south)   -- (n42.north); 
		
		\fill[fill=gray!20!white]  (n51.north) --  (n51.south) -- (n61.south)   -- (n61.north); 
		\fill[fill=gray!20!white]  (n55.north)  -- (n55.south) -- (n64.south) -- (n64.north) ; 
		
		\foreach \y in {10,...,14}{
			\foreach \x in {1,...,6}{
				\draw[very thick, fill=white] (2*\x,\y) circle (2mm);
			}	
		}
		
		
	\draw[color=col2, ultra thick] (n12) -- (n21) -- (n32) -- (n43) -- (n53) -- (n65);

	\draw[color=col4, ultra thick] (n32) -- (n44) -- (n54) -- (n63);
	\draw[loosely dashed, color=col4, ultra thick] (n14) -- (n22) -- (n32);
	
	\draw[color=col1, ultra thick] (n44) -- (n52) -- (n62);
	\draw[loosely dashed, color=col1, ultra thick] (n13) -- (n24) -- (n35) -- (n44);
	
	\draw[color=col3, ultra thick] (n52) -- (n61);
	\draw[loosely dashed, color=col3, ultra thick] (n15) -- (n24) -- (n34) -- (n42) -- (n52) ;

	\node (arr) at (14.5,7) {\rotatebox{210}{\Huge $\Rightarrow$}};
	\node (lab) at (14.5,8) { \rotatebox{210}{\color{col2} $\bullet$ \color{black} $\leftrightarrow$ \color{col1} $\bullet$}};

		\foreach \y in {1,...,5}{
			\pgfmathparse{\y+6}
			\edef\myvar{\pgfmathresult}
			\node[label] (q\y) at (17.2,\pgfmathresult) {$q_{\y}$};
			\foreach \x in {1,...,6}{
				\pgfmathparse{\x+8}
				\node[minimum size=0mm] (n\x\y)  at (2*\pgfmathresult,\myvar) {};
				%
			}	
		}

	\fill[fill=gray!20!white] (n14.north) -- (n23.north) -- (n23.south) -- (n14.south); 
	
	\fill[fill=gray!20!white]  (n33.north) --  (n33.south) -- (n23.south)   -- (n23.north); 
	\fill[fill=gray!20!white]  (n35.north)  -- (n35.south) -- (n23.south) -- (n23.north) ; 
	
	\fill[fill=gray!20!white]  (n33.north) --  (n33.south) -- (n41.south)   -- (n41.north); 
	\fill[fill=gray!20!white]  (n35.north)  -- (n35.south) -- (n45.south) -- (n45.north) ;
	\fill[fill=gray!20!white]  (n33.north) --  (n33.south) -- (n42.south)   -- (n42.north); 
	
	\fill[fill=gray!20!white]  (n51.north) --  (n51.south) -- (n41.south)   -- (n41.north); 
	\fill[fill=gray!20!white]  (n55.north)  -- (n55.south) -- (n45.south) -- (n45.north) ; 
	\fill[fill=gray!20!white]  (n51.north) --  (n51.south) -- (n42.south)   -- (n42.north); 
	
	\fill[fill=gray!20!white]  (n51.north) --  (n51.south) -- (n61.south)   -- (n61.north); 
	\fill[fill=gray!20!white]  (n55.north)  -- (n55.south) -- (n64.south) -- (n64.north) ;

		\foreach \y in {7,...,11}{
			\foreach \x in {9,...,14}{
				\draw[very thick, fill=white] (2*\x,\y) circle (2mm);
			}	
		}
		
		
			\draw[color=col2, ultra thick] (n12) -- (n21) -- (n32);
		\draw[color=col2, ultra thick] (n32) -- (n44) -- (n54) -- (n63);
		
		\draw[loosely dashed, color=col4, ultra thick] (n14) -- (n22) -- (n32);
		\draw[color=col4, ultra thick] (n32) -- (n43) -- (n53) -- (n65);
		
		\draw[color=col1, ultra thick] (n44) -- (n52) -- (n62);
		\draw[loosely dashed, color=col1, ultra thick] (n13) -- (n24) -- (n35) -- (n44);
		
		\draw[color=col3, ultra thick] (n52) -- (n61);
		\draw[loosely dashed, color=col3, ultra thick] (n15) -- (n24) -- (n34) -- (n42) -- (n52) ;

		\node (arr) at (14.5,4) {\rotatebox{330}{\Huge $\Rightarrow$}};
		\node (lab) at (14.5,5) { \rotatebox{330}{\color{col2} $\bullet$ \color{black} $\leftrightarrow$ \color{col3} $\bullet$}};
		
		
		\foreach \y in {1,...,5}{
			\pgfmathparse{\y+3}
			\node[label] (q\y) at (1.2,\pgfmathresult) {$q_{\y}$};
			\foreach \x in {1,...,6}{
				\node[minimum size=0mm] (n\x\y)  at (2*\x,\pgfmathresult) {};
			}	
		}
		
		
		\fill[fill=gray!20!white] (n14.north) -- (n23.north) -- (n23.south) -- (n14.south); 
		
		\fill[fill=gray!20!white]  (n33.north) --  (n33.south) -- (n23.south)   -- (n23.north); 
		\fill[fill=gray!20!white]  (n35.north)  -- (n35.south) -- (n23.south) -- (n23.north) ; 
		
		\fill[fill=gray!20!white]  (n33.north) --  (n33.south) -- (n41.south)   -- (n41.north); 
		\fill[fill=gray!20!white]  (n35.north)  -- (n35.south) -- (n45.south) -- (n45.north) ;
		\fill[fill=gray!20!white]  (n33.north) --  (n33.south) -- (n42.south)   -- (n42.north); 
		
		\fill[fill=gray!20!white]  (n51.north) --  (n51.south) -- (n41.south)   -- (n41.north); 
		\fill[fill=gray!20!white]  (n55.north)  -- (n55.south) -- (n45.south) -- (n45.north) ; 
		\fill[fill=gray!20!white]  (n51.north) --  (n51.south) -- (n42.south)   -- (n42.north); 
		
		\fill[fill=gray!20!white]  (n51.north) --  (n51.south) -- (n61.south)   -- (n61.north); 
		\fill[fill=gray!20!white]  (n55.north)  -- (n55.south) -- (n64.south) -- (n64.north) ;

		\foreach \y in {4,...,8}{
			\foreach \x in {1,...,6}{
				\draw[very thick, fill=white] (2*\x,\y) circle (2mm);
			}	
		}

	\draw[color=col2, ultra thick] (n12) -- (n21) -- (n32) -- (n44) -- (n52) -- (n62);
	
	\draw[loosely dashed, color=col4, ultra thick] (n14) -- (n22) -- (n32);
	\draw[color=col4, ultra thick] (n32) -- (n43) -- (n53) -- (n65);
	
	\draw[color=col1, ultra thick] (n44) -- (n54) -- (n63);
	\draw[loosely dashed, color=col1, ultra thick] (n13) -- (n24) -- (n35) -- (n44);
	
	\draw[color=col3, ultra thick] (n52) -- (n61);
	\draw[loosely dashed, color=col3, ultra thick] (n15) -- (n24) -- (n34) -- (n42) -- (n52) ;

		\end{tikzpicture}	

%% file: proof-induction.tex

	\knowledgenewrobustcmd{\closure}[2]{\cmdkl{\overline{#1}^{#2}}}

\MainThm*

In this section we combine the funnelling lemma (Lemma~\ref{lem:onlyone}) 
and the isolation lemma (Lemma~\ref{lem:isolation}) in order
to almost surely gather the individual tokens back into the cohort,
while keeping the total number of individuals bounded.
This section is the most technical part of the paper.

In order to articulate those lemmas together, we need the following definitions.
Throughout the proofs, we track specific subsets of tokens.
However, we sometimes need to re-anonymise some of them: for instance, to apply Lemma~\ref{lem:isolation}, we need $\arena$ to be a "population arena", which does not track any specific token.
This is why we introduce the following closure operation, which closes an "arena" under renaming and removing tokens (apart from a set $T_f$), making it a "population arena" (tracking $T_f$).

\begin{definition}
	Let $\arena$ be an "arena" and $T_f$ a finite set of tokens. 
	We define $\intro*\closure{\arena}{T_f}$ as its closure under renaming and deleting tokens outside $T_f$. 
	When $T_f = \emptyset$ we simply write $\closure{\arena}{}$.
	
	Formally, for all configuration $\gamma \in \arena$ we define $\phi_{T_f}(\gamma)$ as the pair $(\gamma_f, \vectv) \in S^T \times \NN^S$, where $\gamma_f$ is the projection of $\gamma$ on $T_f$ and $\vectv$ counts the number of other tokens in each state.
	Then we define $\closure{\arena}{T_f} = \bigcup_{(\gamma_f,\vectv) \in \phi_{T_f}(\arena)} \trackideal{\gamma_f}{\vectv}$.
		%
\end{definition}

\begin{lemma}
	\label{lem:closure-pop}
	Let $T_f$ be a finite set of tokens, let $F$ be a union of "ideals tracking $T_f$".
	Let $\arena$ be a "winning@@arena" "arena" with respect to $F$.
	Then $\closure{\arena}{T_f}$ is a  "population arena tracking $T_f$", and is "winning@@arena" with respect to $F$.
\end{lemma}

\begin{proof}
	The fact that $\closure{\arena}{T_f}$ is an "arena" is straightforward.
	Let $\gamma_f \in S^{T_f}$. 
	By Dickson's lemma~\cite{dickson1913finiteness}, for each fixed $\gamma_f$ there are finitely many maximal vectors $\vectv$ such that $(\gamma_f,\vectv) \in \phi_{T_f}(\arena)$.
	Since there are finitely many configurations in $S^{T_f}$, this implies that the set  $\bigcup_{(\gamma_f,\vectv) \in \phi_{T_f}(\arena)} \trackideal{\gamma_f}{\vectv}$ is in fact a finite union of "ideals tracking $T_f$".
	There exists a finite set $V \subseteq \Nb^S$ such that $\closure{\arena}{T_f} = \bigcup_{(\gamma_f,\vectv) \in V} \trackideal{\gamma_f}{\vectv}$.
	
	Consequently, $\closure{\arena}{T_f}$ is a "population arena tracking $T_f$". It is "winning@@arena" as a consequence of the fact that $F$ is a union of "ideals tracking $T_f$", and is thus closed under renaming and deleting tokens outside $T_f$.
\end{proof}

We will now tackle the induction on the number of groups of individual tokens that lets us prove Theorem~\ref{thm:cortosconj}.
\AP In all that follows we call a set of tokens ""secluded"" in a configuration $w$ if every state contains either only tokens of that set or no token of that set.

\begin{remark}
	\label{rem:dummy}
	We can assume without loss of generality that $\Sigma$ contains a letter $\dum$ that labels a loop on each state of $\A$, and no other transition. Adding that letter clearly does not affect the answer to the \problem.
\end{remark}

Finally, we naturally extend definitions of "$\omega$-base" and "finite base" to commits naturally: a set of states $S_\omega$ is an $\omega$-base of $(\gamma,a)$ in an "arena" $\arena$ if $\ideal{(\gamma[{S_\omega * \omega}],a)} \subseteq \arena$, with $\gamma[{S_\omega * \omega}]$ the "symbolic configuration" obtained from $\abs{\gamma}$ by replacing the coefficients in all states of $S_\omega$ with $\omega$. 
A set of tokens $T_\omega$ occupying a set of states $S_\omega$ is an $\omega$-base of $(\gamma,a)$ if $S_\omega$ is an "$\omega$-base" of $(\gamma,a)$. A set of tokens $T_f$ is a finite base of $(\gamma,a)$ if its complement is an $\omega$-base.

\begin{restatable}[Induction]{lemma}{InductionLemma}\label{lem:recstartfine}
	
	Let $T_1, \ldots, T_d$ be disjoint non-empty sets of tokens of size at most $|S|$. Let $T_f = \bigcup_{i=1}^d T_i$.
	Let $I$ be a set of initial configurations containing at least one token of each $T_i$. 
	Let $F$ be a finite union of "ideals tracking $T_f$" of the form $\trackideal{\gamma_f}{\vectv}$ with $\vectv \in \set{0,\omega}^S$.

	Let $\arena$ be a "population arena tracking $T_f$" such that:
	\begin{itemize}
		\item  $I \subseteq \arena$, and for all $\gamma \in I$,  $T_f$ is a "finite base" of $\gamma$ in $\arena$, and
		
		\item for all $\gamma \in \arena \setminus F$, for all $s$ and $T_i$, either $\gamma^{-1}(s) \subseteq T_i$ or $\gamma^{-1}(s) \cap T_i =\emptyset$,and
		
		\item $\arena$ is "winning@@arena" with respect to reaching $F$
		
	\end{itemize}
	
	Then there exists a "winning@@arena" "sub-arena" $Y$ of $\arena$ such that $I \subseteq Y$ and $Y = \BoundKT{Y}{|S|}{T_f}$. 
	
\end{restatable}

\begin{proof}
	We proceed by induction on $|S|-d$.
	
	Suppose $|S| = d$.
	By hypothesis on $\arena$, the tokens of $T_1, \ldots, T_d$ occupy disjoint sets of states in every configuration of $\arena \setminus F$.
	Furthermore, there is at least one token of each set and no other token can be on the states they occupy. 
	Hence for all $\gamma \in \Gamma \setminus F$, the only tokens in $\gamma$ are the ones of $T_1, \ldots, T_d$, and they each occupy one state.
	As a consequence, we have $\arena = \BoundKT{\arena}{|S|}{T_f}$.
	Hence setting $Y = \arena$ yields the result.

	Now suppose $|S| > d$. 
	By Lemma~\ref{lem:onlyone}, for all $\gamma_0 \in \BoundKT{\arena}{0}{T_f}$ there is a path in $\arena$ from $\gamma_0$ to $F$ along which every commit is in $\BoundKT{\arena}{1}{T_f}$.

	Let $V$ be the set of configurations $v \in \arena$ such that there is a non-empty set $T^v$ of at most $|S|$ tokens such that:
	\begin{itemize}
		\item $T_f \cup T^v$ is a finite base of $v$, and
		
		\item  all $T_i$ and $T^v$ are "secluded" in $v$.
	\end{itemize}
	For each $v \in V$, we choose such a set $T_v$ of minimal size.
	We also define $F^v$ the set of configurations of $\arena$ with a strict subset of $T_f \cup T^v$ as a "finite base".
	
	Note that as all $T_i$ are "secluded" in every configuration of $\arena \setminus F$, they must be contained in every "finite base" of every configuration of $\arena \setminus F$. 
	Indeed, adding tokens on states occupied by tokens of $T_i$ immediately gets us out of $\arena$.
	As a consequence, every configuration of $F^v$ is either in $F$ or has a finite base of the form $T_f \cup T'$ with $T' \subsetneq T^v$.
	
	\begin{claim}\label{claim-path}
		For every commit $(\gamma,a)$ in $\arena_{0,1,\omega, T_f}$, every successor $v$ of $(\gamma,a)$ is either in $F$, in $\BoundKT{\arena}{0}{T_f}$ or in $V$.
	\end{claim}

	\begin{claimproof}
		First of all, as $\arena_{0,1,\omega, T_f} \subseteq \arena$ and $\arena$ is an "arena", $v \in \arena$.
		
		Suppose $v \notin F$, then $T_1, \ldots, T_d$ are "secluded" in $v$.
		Let $IT$ be the set of tokens that are not in any $T_i$ and that were alone in their state in $\gamma$.  
		Let $T^v$ be the set of tokens of $IT$ whose state in $v$ only contains tokens of $IT$.
		Note that by definition of $IT$, we have $|T^v| \leq |IT| \leq |S|$.
		
		We now show that $T_f \cup T^v$ is a "finite base"  of $v$ in $\arena$.
		As $(\gamma,a) \in \Gamma_{0,1,\omega, T_f}$, we know that $T_f'= IT \cup \bigcup_{i=1}^d T_i$ is a "finite base" of $(\gamma,a)$.
		Let $S_\omega$ be the set of states containing tokens outside $T_f'$ in $w$.
		It is an "$\omega$-base" of $(\gamma,a)$.
		
		Let $S'_\omega$ be the set of states reachable from $S_\omega$ by playing $a$. 
		As $S_\omega$ is an "$\omega$-base" of $(\gamma,a)$, we can safely play $a$ from any configuration obtained from $w$ by adding tokens in states of $S_\omega$. 
		As we may then get arbitrarily many tokens in each state of $S'_\omega$, it is an "$\omega$-base" of $v$ in $\arena$.
		Every token outside $T_f'$ must end up in $S'_\omega$ in $v$. 
		Consequently, $T_f'$ is a finite base of $v$ in $\arena$.
		
		Furthermore, since we assumed $v \notin F$, tokens of $T_f$ cannot meet other tokens, hence states containing tokens of $IT$ either contain only those or are in $S'_\omega$.
		
		If $T_f$ is a "finite base" of $v$ then $v$ is in $\BoundKT{\arena}{0}{T_f}$.
		Otherwise, $T_f$ is not a "finite base" of $v$ in $\arena$ and thus $T^{v} \neq \emptyset$.
		Then, as	$ T_f \cup T^v$ is a "finite base" of $v$ in $\arena$, and all $T_i$ and $T^v$ are "secluded" in $v$, we have $v\in V$.
	\end{claimproof}
	
	Now, for all $v \in V$, we define an "arena" $Y^v$ that allows us to bring back some tokens of $T^v$ into the cohort while keeping all $T_i$ and $T^v$ "secluded".
	
	\begin{claim}
		For all $v \in V$,  there is a "sub-arena" $Y^{v}$ of $\arena$ such that:
		\begin{itemize}
			\item $v \in Y^v$
			
			\item $Y^v = \BoundKT{Y^v}{|S|}{ T_f \cup T^v}$
			
			\item a "safe random walk in $Y^v$" almost surely reaches $F^v$
			
			\item $T_1, \ldots, T_{d}, T^v$ are "secluded" in every configuration of $Y^v \setminus F^v$
		\end{itemize}
	\end{claim}
	
	\begin{claimproof}
		Observe that $F^v$ is a finite union of "ideals tracking $T_f \cup T^v$".
		As $F$ only contains configurations of which $T_f$ is a "finite base", $F \subseteq F^v$.
		Furthermore, $F^v$ is a subset of the set $F'$ of configurations with a finite base of size $< |T_f \cup T^v|$. Hence $\arena$ is also "winning@@arena" for $F'$.
		
		Since $F'$ is a finite union of "ideals", we can apply Lemma~\ref{lem:closure-pop} with $T_f = \emptyset$ to conclude that $\closure{\arena}{}$ is a "population arena" and is "winning@@arena" with respect to $F'$.
		
		By Lemma~\ref{lem:isolation}, there is a strategy to almost surely reach $F^v$ from $w$ while remaining in $\arena$ and ensuring that $T_f \cup T^v$ stays "secluded" while we have not reached  $F^v$.
		
		As we stay in $\arena$, we furthermore maintain the fact that the groups of tokens $T_1, \ldots, T_d$ are "secluded".
		As a consequence, $T_1, ,\ldots, T_d, T^v$ and the rest of the tokens occupy disjoint sets of states at all times.
		We can therefore define $R$ the set of commits reachable from $v$ while following this strategy. It is clearly an "arena".
		It is a "winning@@arena" "arena" with respect to reaching $F'$. 
		By Lemma~\ref{lem:closure-pop}, so is $\closure{R}{T_f \cup T^v}$. 
		Furthermore, since $R \subseteq \arena$ and $\arena$ is a "population arena tracking $T_f$", we have $\closure{R}{T_f \cup T^v} \subseteq \closure{\arena}{T_f \cup T^v} = \arena$.
		
		We apply the induction hypothesis to get a "winning@@arena" "sub-arena" $Y^v$ of $R$ (and of $\arena$) such that $\set{v} \subseteq Y^v$ and $Y^v = \BoundKT{Y^v}{|S|}{ T_f \cup T^v}$.	
	\end{claimproof}
	
	Define $X  = \BoundKT{\Gamma}{0}{T_f} \cup \bigcup_{v \in V} Y^v \cup \set{(\gamma, \dum) \mid \gamma \in F}$.
	Note that as $v \in Y^v$ for all $v$, $V \subseteq X$.
	The letter $\dum$ is given by Remark~\ref{rem:dummy}.
	It is used here to integrate $F$ in $X$ without conflicts in definitions (an "arena" being a set of commits and not configurations).
	
	\begin{claim}
		$X$ is a "winning@@arena" "arena" with respect to $F$.
	\end{claim}
	\begin{claimproof}
		We start by showing that $X$ is an "arena".
		Let $(x,a) \in X$. If $a = \dum$ then the only successor is $x$ itself.
		If $(x,a) \in Y^v$ for some $v \in V$, as $Y^v$ is an "arena", the successors are all in $Y^v$.
		If  $(x,a)\in \BoundKT{\Gamma}{1}{T_f}$, then by Claim~\ref{claim-path} every successor $v$ is either in $V$, or $\BoundKT{\Gamma}{0}{T_f}$ (thus also $\BoundKT{\Gamma}{1}{T_f}$), or $F$. In all three cases it is in $X$.
		
		We now show that for all $x \in X$, there is a path from $x$ to $F$ within $X$.
		Again, every configuration in $X$ is either in $F$, in $ \BoundKT{\arena}{1}{T_f}$, or in $V$.
		In the first case, we are done. In the second case, if $x \in \BoundKT{\arena}{0}{T_f}$ then
		we simply apply Lemma~\ref{lem:onlyone} to get a path from $x$ to $F$ in $\BoundKT{\arena}{1}{T_f}$, hence in $X$.
		If $x \in \BoundKT{\arena}{1}{T_f} \setminus  \BoundKT{\arena}{0}{T_f}$, then $x$ must be in $V$.
		
		So all is left to prove is that there is a path in $X$ from every configuration of $V$ to $F$.
		We show this by strong induction on $|T^v|$.
		Let $v \in V$, and $x \in Y^v$.  As $Y^v$ is "winning@@arena" with respect to $F^v$, there is a path in $Y^v$ from $y$ to some configuration $z \in F^v$.
		Recall that $F^v$ is the set of configurations of $\Gamma$ with a strict subset of $T^v \cup T_f$ as a "finite base".
		\begin{itemize}
			\item If $z \in F$ then we are done
			
			\item  If $z \notin F$ then $T_1, \ldots, T_d$ are all "secluded" in $z$. 
			\begin{itemize}
				\item  If $z \in \BoundKT{\arena}{0}{T_f}$ then by Lemma~\ref{lem:onlyone} there is a path from $z$ to $F$ in $\BoundKT{\Gamma}{1}{T_f}$, thus in $X$.
				
				\item If $z \notin \BoundKT{\arena}{0}{T_f}$ then $z$ is in $V$. As $z \in F^v$, we must have $|T^z| < |T^v|$. We can apply the induction hypothesis to get a path from $z$ to $F$ in $X$.
			\end{itemize}
			In both cases we have a path in $X$ from $x$ to $z$ and from $z$ to $F$, thus from $x$ to $F$.
		\end{itemize}
	\end{claimproof}
	
	Clearly $I \subseteq \BoundKT{\arena}{0}{T_f} \subseteq  X$.
	In order to obtain an "$|S|$-definable" "arena", we set $Y = \BoundKT{X}{|S|}{ T_f}$, i.e., we close $X$ under renaming and removing tokens outside $T_f$.
	Clearly $\BoundKT{\arena}{1}{T_f}$ and $F$ are closed under those operations. 
	For all $v$,

	as $Y^v = \BoundKT{Y^v}{|S|}{T_f \cup T^v}$, and $T^v$  has at most $|S|$ elements and is "secluded" in all of $Y^v \setminus F^v$, and all configurations of $F^v$ are in $F$ or have a "finite base" included in $T^v \cup T_f$, the set $\closure{Y^v}{T_f}$ is a union of "ideals tracking $T_f$" with largest constant $|S|$.
	Hence $Y = \BoundKT{Y}{|S|}{T_f}$. 
	
	Finally, by Lemma~\ref{lem:closure-pop}, $Y$ is a  "population arena tracking $T_f$" that is "winning@@arena" for $F$.
\end{proof}

\begin{proof}[Proof of Theorem~\ref{thm:cortosconj}]
	Since $\arena$ is the maximal "winning@@arena" "arena", by Lemma~\ref{lem:closure-pop}, we have $\arena = \closure{\arena}{}$, hence $\arena$ is a "population arena". 
	We can then straightforwardly apply Lemma~\ref{lem:recstartfine} with $d = 0$ (and $T_f =\emptyset$).
\end{proof}

%% file: reduction-to-SFP.tex
%

Let $\?M = (S, \Sigma, \Delta)$, $\vectV$, $\vectv_0$, $\vectF$ be an instance of the \problemm.
We split $\vectv_0$ in two parts $\vectv_0 = \vectv_f + \vectv_\omega$, so that $\vectv_f \in \NN^S$ and $\vectv_\omega \in \set{0, \omega}^S$. 
We start by defining an automaton $\?A$ controlling the movement of the tokens from $\vectv_f$, and translating the constraints given by $\vectV$ into "capacities" for the movement of the tokens from $\vectv_\omega$. 

Let $m  = \sum_{s \in S} \vectv_f(s)$. 
We define pairs $(\alpha, \beta)$ describing the set of allowed transfers of tokens. The movement of tokens from $\vectv_f$ is described by $\beta$, the one of other tokens by $\alpha$.
Formally, define $P$ the set of pairs $(\alpha, \beta)$ such that there exists $(\vectv, a) \in \vectV$ with:
\begin{itemize}
	\item $\alpha : S^2 \to \Nb $ and $\beta : S^2 \to \NN$ and $\sum_{s,s' \in S} \beta(s,s') = m$
	
	\item  for all $s \in S$, $\sum_{s' \in S}\alpha(s,s') + \beta(s,s') \leq \vectv(s)$ 
	
	\item for all $s,s' \in S$, if $\Delta(s,a)(s')=0$ then $\alpha(s,s') = \beta(s,s') =0$
	
	\item for all $s \in S$, if $\vectv(s) = \omega$ then for all $s' \in S$ such that $\Delta(s,a)(s')>0$, $\alpha(s,s') = \omega$
\end{itemize}
The first condition says that $\beta$ should describe the movement of $k$ tokens. The second and third one make sure that the transfers of tokens described by $\alpha$ and $\beta$ are allowed by $\vectV$.
The fourth one is there to ensure that $P$ is finite: if we have a commit letting us transfer arbitrarily many tokens from $s$ to $s'$, we want a single "capacity" with an $\omega$ label between $s$ and $s'$, and not one for each number $n \in \NN$.

\newcommand{\vectf}{\vect{f}}

Similarly, we define $P_\vectF$ the set of pairs $(\alpha, \beta)$ such that there exists $\vectf \in \vectF$ with:
\begin{itemize}	
	\item $\alpha : S^2 \to \Nb $ and $\beta : S^2 \to \NN$ and $\sum_{s \in S} \beta(s,s) = m$
	
	\item  for all $s \in S$, $\alpha(s,s) + \beta(s,s) \leq \vectf(s)$ 
	
	\item For all $s \neq s'$, $\alpha(s,s') = \beta(s,s') = 0$
	
	\item for all $s \in S$, if $\vectf(s) = \omega$ then $\alpha(s,s) = \omega$
\end{itemize}
 
 Those pairs do not let tokens change states: they are here to check, at the end, that the repartition of tokens is within $\vectF$.

We define the automaton $\mathcal{B}$ whose set of states is the set of vectors \[Q =  \set{\vectv \in\NN^S \mid  \sum_{s \in S}\vectv(s) =m} \cup \set{q_{\mathbf{end}}}.\]
The initial state of $\?B$ is $\vectv_f$ and the final one is $q_{\mathbf{end}}$. Transitions are of two types:
\begin{itemize}
	\item We have a transition $(\vectv, (\alpha,\beta), \vectv')$ whenever $(\alpha,\beta) \in P$ and $\vectv(s) = \sum_{s' \in S} \beta(s,s')$ and $\vectv'(s') = \sum_{s \in S} \beta(s,s')$ for all $s,s' \in S$.
	
	\item We have a transition $(\vectv, (\alpha,\beta), q_{\mathbf{end}})$ whenever $(\alpha,\beta) \in P_{\vectF}$ and $\vectv(s) = \beta(s,s)$ for all $s \in S$.
\end{itemize}

The last transition is there to check that we have reached a configuration of $\vectF$.

Then, we define $A = \set{\alpha \mid \exists \beta, (\alpha, \beta) \in P}$. This will be our alphabet of "capacities".
Let $\mathcal{A}$ be the automaton obtained by projecting $\mathcal{B}$ on $A$: every transition $(\vectv, (\alpha,\beta), \vectv')$ is replaced by 
$(\vectv, \alpha, \vectv')$ and $(\vectv, (\alpha,\beta), q_{\mathbf{end}})$ is replaced by 
$(\vectv, \alpha, q_{\mathbf{end}})$, the rest remains the same.

Let $K$ be the largest number appearing in the input.
Observe that the number of states is bounded by $(K|S|+1)^{|S|}$ (since $m \leq K|S|$) and the size of the alphabet by $(K+2)^{|S|}$ (since the defined "capacities" use values in $\set{0,\dots, K, \omega}$).
Hence the size of the resulting automaton is singly exponential in the number of states and the binary representation of  $K$

The following lemma formally relates "token flows" constrained by $L(\?A)$ and "paths" constrained by $\vectV$.
Roughly speaking, the former is the projection of the latter on the tokens which come from unbounded states in $\vectv_0$.
The automaton $\?A$ takes care of simulating the rest of the tokens, and recognizes the sequences of "capacities" that let us reach $\vectF$ from $\vectv_0$  while staying in $\vectV$ and carrying those remaining tokens.

\begin{lemma}
	\label{lem:flow-to-path}
	Let $N \in \NN$ and $\gamma^N$ a configuration such that on each state $s$, there are $N$ tokens if $\vectv_\omega(s)=\omega$ and $0$ token otherwise.  
	Let $\gamma^{N,f}$ a configuration obtained from $\gamma^{N}$ by adding $\vectv_f(s)$ tokens on each $s$ such that $\vectv_\omega(s)=0$.
	The following are equivalent:
	\begin{itemize}
		\item There is a "token flow" from $\gamma^N$ over a word of $L(\?A)$
		
		\item There is a "path" from $\gamma^{N,f}$ to a configuration of $\ideal{\vectF}$ in $\ideal{\vectV}$.
	\end{itemize}
\end{lemma}

\begin{proof}
		Let $T_\omega$ the set of tokens used in $\gamma^N$ and $T_f$ the set of additional tokens in $\gamma^{N,f}$.
	
		We start with the left-to-right direction.
		Suppose we have a "token flow" $\tau = \gamma_0 \to \cdots  \to \gamma_k$ over a word $w = \alpha_1 \dots \alpha_k$ in $L(\?A)$, with $\gamma_0 = \gamma^N$.
		By construction of $\?A$ there exist $\beta_1, \ldots, \beta_k$ such that $(\alpha_1, \beta_1) \dots (\alpha_k, \beta_k) \in P^* P_\vectF$ is accepted by the automaton $\?B$.
		Let $\vectv_0, \dots, \vectv_{k-1}, q_{\mathbf{end}}$ be the sequence of states visited by an accepting run of $\?B$ on this word. 
		
		We construct a path $\pi = \gamma^f_0 \xrightarrow{a_1} \dots \xrightarrow{a_{k-1}} \gamma^f_{k-1}$  in $\ideal{\vectV}$ from $\gamma_0^f = \gamma^{N,f}$ by making tokens of $\gamma^N$ take the same transitions at every step, and moving the other tokens according to $\beta_i$ at the $i$th step.
		
		Formally, we proceed by induction on $i\in \set{0, \dots, k-1}$ to define a "path" $\gamma^f_0 \xrightarrow{a_1} \dots \xrightarrow{a_{i}} \gamma^f_{i}$ so that:
		\begin{itemize}
			\item $\gamma^f_0 = \gamma^{N,f}$
			
			\item for all $s \in S$, there are $\vectv_i(s)$ tokens of $T_f$ in $s$ in $\gamma^f_i$
			
			\item The projection of $\gamma^f_i$ on $T_\omega$ is $\gamma_i$
			
			\item for all $j<i$, $(\gamma^f_j,a_{j+1}) \in \ideal{\vectV}$. 
		\end{itemize}
	
		The base case $i=0$ is immediate from the definition of $ \gamma^{N,f}$. 
		Let $i>0$, suppose we constructed $\gamma^f_0 \xrightarrow{a_1} \dots \xrightarrow{a_{i-1}} \gamma^f_{i-1}$. 
		\begin{itemize}
			\item The first item is immediate
			
			\item Since there is a transition reading $(\alpha_{i}, \beta_i)$ from $\vectv_{i-1}$ to $\vectv_i$ in $\?B$, for each $s \in S$, we have $\sum_{s' \in S} \beta(s,s') = \vectv_{i-1}(s)$ and for all $s' \in S$, $\sum_{s \in S} \beta(s,s') = \vectv_{i}(s')$.
			By the second item of the induction hypothesis, we can thus transfer $\beta(s,s')$ tokens of $T_f$ from $s$ to $s'$ for each pair $s,s'$, after which the number of tokens of $T_f$ in each $s'$ is $\vectv_i(s')$.
			
			\item We move tokens of $T_\omega$ as in the step $\gamma_{i-1} \to \gamma_{i}$ in $\tau$, hence the projection of the resulting configuration on $T_\omega$ is $\gamma_i$. Also, this means that the number of tokens of $T_\omega$ going from $s$ to $s'$ is at most $\alpha_i(s,s')$
			
			\item 	Since $(\alpha_i, \beta_i)\in P$, there exists $(\vectv, a) \in \vectV$ such that, for all $s, s' \in S$, we have $\sum_{s' \in S} \alpha_i(s,s') + \beta_i(s,s') \leq \vectv(s)$ and if $\Delta(s,a)(s')=0$ then $\alpha_i(s,s') = \beta_i(s,s')=0$. 
			We thus have, for each $s \in S$, \[|(\gamma_{i-1})^{-1}(s)| = |(\gamma_{i-1})^{-1}(s) \cap T_f| + |(\gamma_{i-1})^{-1}(s)\cap T_\omega| \leq \sum_{s' \in S} \alpha_i(s,s') + \sum_{s' \in S} \beta_i(s,s') \leq \vectv(s).\]
			Also, for all $s,s'$, if $\Delta(s,a)(s')=0$ then 
			\begin{align*}
				&|(\gamma_{i-1})^{-1}(s) \cap (\gamma_i)^{-1}(s)|\\
				= & |(\gamma_{i-1})^{-1}(s) \cap (\gamma_i)^{-1}(s') \cap T_f| + |(\gamma_{i-1})^{-1}(s) \cap (\gamma_i)^{-1}(s') \cap T_\omega| \\
				\leq & \alpha_i(s,s') + \beta_i(s,s') = 0.
			\end{align*}
			As a result, by setting $a_i =a$ we get  $(\gamma^f_{i-1},a_{i}) \in \ideal{(\vectv,a)} \subseteq \ideal{\vectV}$
		\end{itemize}

		In the end, we obtain a "path" $\pi = \gamma^f_0 \xrightarrow{a_1} \dots \xrightarrow{a_{k-1}} \gamma^f_{k-1}$  in $\ideal{\vectV}$ as wanted.
		It remains to show that $\gamma^f_{k-1}$ is in $\ideal{\vectF}$.
		There is a transition from $\vectv_{k-1}$ to $q_{\mathbf{end}}$ labeled $(\alpha_k, \beta_k)$, hence there exists $\vectf \in \vectF$ such that the conditions in the definition of $P_\vectF$ are satisfied, and $\vectv_{k-1}(s) = \beta_k(s,s)$ for all $s$.  
		We also have, since $\gamma_{k-1} \to \gamma_k$ is a "token flow" over $(\alpha_k, \beta_k) \in P_\vectF$, $|(\gamma_{k})^{-1}(s)| = |(\gamma_{k-1})^{-1}(s)| \leq \alpha_k(s,s)$ for all $s$.
		As a result, for all $s$ we have 
		\begin{align*}
			& |(\gamma^f_{k-1})^{-1}(s)|\\
			= & |(\gamma^f_{k-1})^{-1}(s) \cap T_f| + |(\gamma^f_{k-1})^{-1}(s) \cap T_\omega|\\
			= & \vectv_{k-1}(s) + |(\gamma_{k-1})^{-1}(s)|\\
			\leq & \beta_k(s,s)+ \alpha_k(s,s) \\
			\leq & \vectf(s).
		\end{align*}
	This concludes the proof of the left-to-right direction.
		
		We proceed with the right-to-left direction.
		
		Suppose we have a "path" $\pi = \gamma^f_0 \xrightarrow{a_1} \dots \xrightarrow{a_{k-1}} \gamma^f_{k-1}$ from $\gamma^{N,f}$ to $\ideal{\vectF}$ in $\ideal{\vectV}$.
		Let $\vectf \in \vectF$ be such that $\gamma^f_{k-1} \in \ideal{\vectf}$.
		
		For all $i \leq k-1$ we define the following objects. 
		Let $\gamma_i$ be the projection of $\gamma_i^f$ on $T_\omega$. Let $\vectv_i$ be the vector counting the number of tokens of $T_f$ in each state in $\gamma_i^f$.
		Since $\pi$ is a path in $\ideal{\vectV}$, there exists $(\vectv, a) \in \vectV$ such that $|(\gamma^f_i)^{-1}(s)|\leq \vectv(s)$ for all $s$ and for all $s,s'$, $\Delta(s,a)(s') = 0$ implies $|(\gamma^f_{i})^{-1}(s) \cap (\gamma^{f}_{i+1})^{-1}(s')| =0$. Let $\beta_i$ be the vector such that $\beta_i(s,s') = |(\gamma^f_{i})^{-1}(s) \cap (\gamma^{f}_{i+1})^{-1}(s') \cap T_f|$ for all $s,s'$.
		Let $\alpha_i$ be the vector such that:
		\begin{itemize}
			\item if $\Delta(s,a)(s')=0$ then $\alpha_{i}(s,s') = 0$,
			
			\item otherwise if  $\vectv(s) =\omega$ then $\alpha_{i}(s,s') =\omega$,
			
			\item otherwise $\alpha_i(s,s')= |(\gamma^f_{i})^{-1}(s) \cap (\gamma^{f}_{i+1})^{-1}(s) \cap T_\omega|$.
		\end{itemize}
		It is easy to check that $(\alpha_i, \beta_i) \in P$, that it labels a transition from $\vectv_{i-1}$ to $\vectv_i$ in $\?B$, and that $|(\gamma_{i-1})^{-1}(s) \cap (\gamma_{i})^{-1}(s')| \leq \alpha(s,s')$.
		
		Similarly, we construct $(\alpha_k, \beta_k)$ by setting $\alpha_k(s,s) = \omega$ if $\vectf(s)=\omega$ and  $|(\gamma^f_{k-1})^{-1}(s) \cap  T_\omega|$ otherwise, and $\beta_k(s,s) = |(\gamma^f_{k-1})^{-1}(s) \cap  T_f|$.
		It is easy to check that $(\alpha_k, \beta_k)$ is in $P_\vectF$ and labels a transition from $\vectv_{k-1}$ to $q_{\mathbf{end}}$ in $\?B$.
		Therefore, $(\alpha_1,\beta_1) \dots (\alpha_k,\beta_k) \in L(\?B)$, hence $\alpha_1 \dots \alpha_k \in L(\?A)$.

		Let $\tau = \gamma_0 \to \dots \to \gamma_{k-1} \to \gamma_k$, with $\gamma_k = \gamma_{k-1}$. By the properties shown above, this is a "token flow" over the word $\alpha_1 \dots \alpha_k$.
This concludes the proof.
\end{proof}

With this equivalence in hand, we can prove the reduction from the \problemm{} to the computation of $\multiflow$.

\ReductionToSFP*

\begin{proof}
	For all $N$ let $\gamma^N$ and $\gamma^{N,f}$ be as in Lemma~\ref{lem:flow-to-path}.
	Since every configuration of $\vectv_0$ can be obtained from $\gamma^{N,f}$ for some $N$ by removing and renaming tokens, $\?M$, $\vectV$, $\vectv_0$, $\vectF$ is a positive instance of the \problemm{} if and only if for all $N$ there is a path from $\gamma^{N,f}$ to $\ideal{\vectF}$ within $\ideal{\vectV}$.
	By Lemma~\ref{lem:flow-to-path}, this is the case if and only if for all $N$ there is a "token flow" over a word of $L(\?A)$ from $\gamma^{N}$.
	We now show that the latter property is true if and only if the two items of the lemma hold.
	
	Suppose for all $N$ there is a "token flow" $\tau_N$ over a word of $L(\?A)$ from $\gamma^{N}$.
	For all $N$, let $\gamma_{\mathbf{end}}^N$ be the last configuration in $\tau$, and $E_N$ the set $\set{(s,s') \mid \globalflow{\tau_N}(s,s') \geq \frac{N}{|S|}}$.
	Observe that for all $s$ such that $\vectv_\omega(s) = \omega$, since $s$ contains $N$ tokens initially in $\gamma^N$, there must exist $s'$ such that $\globalflow{\tau_N}(s,s') \geq \frac{N}{|S|}$.
	Thus, $\vectv_0^{-1}(\omega)\subseteq E_N$ for all $N$.
	Since there are finitely many possibilities for $E_N$, we can select an infinite subsequence of $(\tau_N)$ which all agree on $E_N$. 
	Formally, we select a set $E$ and an infinite sequence $n_1 < n_2 < \cdots$ of indices such that $E_{n_j} = E$ for all $j$.
	For all $j$ we have $\globalflow{\tau_{n_j}}(s,s') \geq \frac{n_j}{|S|}$. Since $(n_j)_{j \geq 1}$ tends to infinity as $j$ grows, we have $\multiflow(E, L(\?A)) = +\infty$.
	
	Now suppose we have $\multiflow(E, L(\?A)) = +\infty$ for some $E$ such that  $\vectv_0^{-1}(\omega)\subseteq E$.
	Then for all $N$ there exists a "token flow" $\tau_N$ over a word $w_N$ of $L(\?A)$ such that $\globalflow{\tau_N}(s,s') \geq N$ for all $(s,s') \in E$.
	In particular, since $\vectv_0^{-1}(\omega)\subseteq E$, for all $s$ such that $\vectv_0(s)=\omega$ we must have at least $N$ tokens in $s$ at the start of $\tau_N$.
	Hence, by removing and renaming tokens, we obtain a "token flow" from $\gamma^N$ over $w_N$, concluding the proof. 
\end{proof}

%% file: lower-bound.tex


\ThmExpHard*

A \reintro{Countdown Game}
is given by a directed graph $\?G=(V,E)$,
where edges carry positive integer weights, 
$E\subseteq (V\x\+N_{>0}\x V)$. 
Given an initial pair $(v,c_0)\in V\x\+N$ consisting of a vertex and a number, two opposing players (Players 1 and 2) alternately determine a sequence of such pairs as follows.
In each round, from $(v,c)$,
Player~1 picks a number $d\le c$ such that $E$ contains at least one edge $(v,d,v')$;
then Player~2 picks one such edge and the game continues from $(v',c-d)$.
Player~1 wins the game if and only if the play reaches a pair in $V\x\{0\}$.

\textsc{CountdownGame} is the decision problem{} which asks if Player~1 has a strategy to win a given game for a given initial pair $(v_0,c_0)$. 
All constants in the input are represented in binary.

\color{black}

\begin{proposition}[Thm.~4.5 in \cite{JLS2007}]
    \label{lemma}
    \textsc{CountdownGame} is \exptime-complete.
\end{proposition}

\label{sec:lower-bound-gadgets}
To reduce \textsc{CountdownGame} to the \problem{}, we first observe that the number of turns in a "Countdown Game" cannot exceed the initial value $c_0$ of the counter, as the counter decreases at each turn. Thus, if Player~2 has a winning strategy, choosing actions at random yields a positive probability of following that strategy, and therefore a positive probability of winning. Consequently, Player~1 wins the initial game if and only if she wins with probability one against a randomised adversary.

The main idea for our construction is to require \Laetitia to move tokens one by one 
away from a waiting state, first into 
the control graph of the "Countdown Game", and ultimately into the goal.
To avoid a loss in the intermediate phase, she must win an instance of that game against a randomizing opponent.
This is enforced using a combination of gadgets, including two
binary counters that 
can effectively test for zero, be set to specific numbers,
and that are set up so that they can decrement at the same rate.
%
%
As a result, Player~1 has a winning strategy for the two-player "Countdown Game" if, and only if, \Laetitia
can synchronize the $n$-fold product of the constructed MDP
for all $n$.


%
%





\medskip
For a given "Countdown Game" $\?G$ with an initial pair~$(v_0,c_0)$ we construct an MDP $\?M$ as follows.
We write that action $a$ ""takes state $s$ to successor $t$"" to mean that $\delta(s,a)(t)>0$.

A state $s$ is ""marked"" in a configuration $\gamma$ if at least one token occupies it: $\exists t. \gamma(t)>0$.
Whenever action $a$ takes state $s$ only back to itself we say that $s$ ""ignores"" $a$.
There are states $\HEAVEN$ (the target) and $\HELL$ which ignore all actions.
For a given state $s$, an action $a$ is ""angelic"" if it takes $s$ only to $\HEAVEN$,
and ""daemonic"" if it takes $s$ to $\HELL$.
An action $a$ is ""safe"" in a configuration if it is not "daemonic" for any "marked" state (in any gadget).

\begin{figure}[t]
\begin{subfigure}[T]{0.4\textwidth}
\input{Figures/gadget-waiting-room.tikz}
\end{subfigure}
\begin{subfigure}[T]{0.6\textwidth}
\hfill
\input{Figures/gadget-control.tikz}
\end{subfigure}
\caption{The waiting (on left) and
  the control gadgets (on right).
Edges labelled by $\Sigma_X$ are shorthand
    for several edges, one for each action in $\Sigma_X$.
    All but the depicted actions are daemonic.
}
\label{fig:control}
\label{fig:waiting-room}
\end{figure}
Besides the special states $\HEAVEN$ and $\HELL$, $\?M$ contains several gadgets described below.
\vspace{-1em}
\subparagraph*{Waiting.}
\newcommand{\WaitAct}{\texttt{wait}}
\newcommand{\WaitState}{\texttt{Wait}}
\newcommand{\ReadyState}{\texttt{Ready}}

The waiting gadget
has two states $\WaitState$ and $\ReadyState$
which react to the action $\texttt{wait}$
as depicted in Figure~\ref{fig:waiting-room} (left).
%
%
Whenever a configuration marks one of these states,
a strategy that continuously plays $\WaitAct$ will almost surely reach a configuration
in which exactly one component marks $\ReadyState$.


A special action $\texttt{go}$ (to indicate successful isolation of one component) takes
$\ReadyState$ to the initial state $v_0$ of the game $\?G$.
All other actions (in gadgets described below) are ignored.
This is similar to what happens in Example~\ref{ex:force-one}.


\subparagraph*{Game.}
The game $\?G=(G,E)$ is directly interpreted as MDP: 
For every edge $(s,d,s')\in E$ there is an action $(s,d)$ 
which takes $s$ to $s'$ 
and which is daemonic for all states $s'\neq s$.

The action $\texttt{win}$ is angelic for every state of $G$.  All other actions are ignored.
%
%


\subparagraph*{Binary Counters.}
A ($k$-bit) Counter consists of states $\BIT{i}{j}$ for all $0\le i< k$ and $j\in\{0,1\}$.
For every bit $i$ there is a decrement action $\dec{i}$ which
\begin{itemize}
    \item takes $\BIT{j}{0}$ only to $\BIT{j}{1}$ for all $0\le j < i$,
    \item takes $\BIT{i}{1}$ only to $\BIT{i}{0}$,
    \item is daemonic for $\BIT{i}{0}$, and
    \item is ignored by all $\BIT{j}{l}$, for all $i<j$ and $l\in\{0,1\}$.
\end{itemize}

We say that a configuration \emph{holds} the number 
$c<2^k$ in this counter if 
it marks those states that represent the binary expansion of $c$:
for all~$0\le i\le k-1$,
state $\BIT{i}{j}$ is marked iff
the $i$th bit in the binary expansion of $c$ is~$j$. 
An action~$a$ \emph{sets} the counter to number $d$
if for all $0\leq i<k$, 
it takes $\BIT{i}{0}$ to only $\BIT{i}{j}$ where $j\in \{0,1\}$ is the $i$th bit in the binary expansion of~$d$,
and is daemonic for all $\BIT{i}{1}$ (to ensure that the counter can only be set if it holds~$0$).
%
%

Additionally, for every bit $i$ the gadget has an error action $\error{i}$,
which is daemonic for $\BIT{i}{0}$ and $\BIT{i}{1}$, and 
angelic for every other state (of $\?M$).
These actions can be used to quickly synchronise any configuration in which the counter is not correctly initialised,
i.e., does not hold a number.
See Figure~\ref{fig:counter} for a depiction of a $4$-bit counter.

\medskip
The MDP $\?M$ will contain two distinct counter gadgets. 
A main counter $MC$ 
has 
$\log_2(n_0)$ bits to hold possible counter values of the "Countdown Game".
An auxiliary counter $AC$ has $\log_2(d_{\texttt{max}})$ many bits to hold the largest edge weight~$d_{\texttt{max}}$ in $\?G$.
These have distinct sets of states and actions,
so for clarity, 
we write $\DIV{C}{x}$ to refer to state (or action) $x$ in gadget $C$.
We connect some new actions to these two counters as follows.
%
%
%
%
\begin{itemize}
    \item The action $\texttt{go}$ sets $MC$ to~$n_0$;
this ensures that $MC$ holds $n_0$ when starting to simulate~$\?G$.
\item The action $\texttt{win}$ is daemonic for every state $\DIV{MC}{\BIT{i}{1}}$.
 This enforces that the $MC$ must hold $0$ when a strategy claims Player~1 wins $\?G$.
\item Any action~$(v,d)\in \Sigma_{\?G}$ sets $AC$ to $d$;
\item The action $\texttt{next}$ is daemonic for every state $\DIV{AC}{\BIT{i}{1}}$.
This enforces that a strategy must first count down from~$d$ to~$0$ before it can simulate the next move in~$\?G$.
\end{itemize}
\begin{figure}[t]
    \centering
    \input{Figures/gadget-counter-4h.tikz}
\caption{A (4-bit) Binary Counter.
    Not displayed are edges labelled by $\dec{i}$ that make the respective
    actions daemonic for state $\BIT{i}{0}$, and error actions $\error{i}$,
    which are daemonic for $\BIT{i}{0}$ and $\BIT{i}{1}$, for all bits
    $i\in\{0,1,2,3\}$.
}
\label{fig:counter}
\end{figure}
\subparagraph*{Control.}
The control gadget
will enforce that a synchronising strategy proposes actions in a proper order; see Figure~\ref{fig:control}.
It consists of states $W,G,A,B$, and contains actions of all gadgets above (including $\texttt{go}$,  $\texttt{win}$, $\texttt{next})$
and a new $\error{}$ action, which is angelic for all states except~$W$, for which it is daemonic.
All omitted edges in~Figure~\ref{fig:control} are daemonic.

\subparagraph*{Start/End.}
%
To complete the construction of~$\?M$, we introduce 
an initial state $\StartState$ and
actions~$\texttt{start}$ and $\texttt{end}$.
The action $\texttt{start}$ 
takes $\StartState$ to $\WaitState$ (Waiting gadget), $W$    
(Control gadget), and all $\BIT{i}{0}$ states of counters $AC$ and $MC$.
It is daemonic for every other state.

The  action $\texttt{end}$ is daemonic for $\WaitState$ and $\ReadyState$,
and angelic for every other state in~$\?M$.

\begin{theorem}
    \label{theorem}
    $\?M, \mathtt{Init}, \mathtt{Heaven}$ is a positive instance of the \problem{} iff Player~1 wins $\?G$.
\end{theorem}
\begin{proof}
Suppose Player~1 wins the game~$\?G$.
Fix $n$ and a set of $n$ tokens $T$. Recall that in $\?M^{(T)}$  all components of the initial configuration  mark~$\StartState$. A synchronising strategy proceeds as follows:
\begin{itemize}
        \item Play $\texttt{start}$ to initialize the Waiting and Control gadgets, and to set~$AC$ and $MC$ to~$0$.
        If any of the gadgets is not correctly initialised afterwards, play the respective error action to win directly.
           For instance, if $W$ is unmarked, play $\error{}$ to synchronise.
        \item 
            Reduce the number of components marking $\WaitState$ one by one 
            until a configuration is reached in which $\WaitState$ is not marked.
            Once this is true, play~$\texttt{end}$ to synchronise.
        \item To reduce the number of components marking $\WaitState$,
            isolate one of them, and move it to $\HEAVEN$ by simulating the "Countdown Game":
            \begin{enumerate}
            \item Play $\WaitAct$ until only a single component marks $\ReadyState$, then play $\texttt{go}$.
                This will mark~$v_0$ in the game gadget and sets~$MC$ to $n_0$. Recall that $(v_0,n_0)$ is the initial pair of~$\?G$.
            \item Simulate rounds of the game~$\?G$:
           assume state~$v$ in the game gadget is marked and
           the counter~$MC$ holds~$c$, then let $d$ be
           the number Player~$1$ plays to win
           from the pair~$(v,c)$ in~$\?G$. Play~$(v,d)$.
           This action will  set  $AC$ to $d$.
           Alternate between (safe) decrement actions in $AC$ and $AB$ until they hold $0$ and $c-d$, respectively.
           Play $\texttt{next}$.
       \item The above simulation of rounds in~$\?G$ is repeated until both $AC$ and $AB$ hold $0$, by assumption that Player~1 wins~$\?G$, this is possible. At this point it is safe to play $\texttt{win}$.
            \end{enumerate}
    \end{itemize}

    
   Conversely, assume that Player~1 cannot win~$\?G$.
    Suppose that after the (only possible) initial move $\texttt{start}$, all gadgets are correctly initialized.
    Clearly, for every $n$, this event has strictly positive probability. We argue that no strategy can synchronise
    such a configuration. Indeed, a successful strategy had to play a sequence in
    $\WaitAct^*\cdot \texttt{go}$ first, 
    followed by actions in $(\Sigma_{\?G}\cdot(\Sigma_{AC}\cdot\Sigma_{MC}\cdot\texttt{next})^*)^*$,
    by construction of the control gadget.
    If after playing $\texttt{go}$, more than one component mark~$v_0$, there is a non-zero chance that
    these will diverge, making subsequent actions in $\Sigma_{\?G}$ unsafe.
    If exactly one component marks~$v_0$ then the second sequence of actions (assuming all actions are safe) corresponds to a  play of~$\?G$.
    This inevitably leads to a configuration in which
    counter $MC$ holds $0$ and the control enforces that the next action is in $\Sigma_{MC}$.
    But any such action will be daemonic for some state in $MC$ and thus not be safe.
    We conclude that every strategy will lead to a configuration in which at least one component marks $\HELL$
    and thus cannot be synchronised.    
\end{proof}

Theorem~\ref{thm:exp-hard} follows immediately from \cref{lemma,theorem}.

%% file: Figures/gadget-waiting-room.tikz
\begin{tikzpicture}[
  node distance=0.75cm and 1.25cm,
  ]
\node[estate] (I)                    {\texttt{wait}};
\node[estate]         (R)  [right= of I]      {\texttt{ready}};

\path  (I.north east) edge node[anode,auto] {$\texttt{wait}$}(R.north west);
\path  (R.south west) edge node[anode,auto] {$\texttt{wait}$}(I.south east);

\path  (I) edge[loop left] node[anode]{$\texttt{wait}$} (I);
\end{tikzpicture}

%% file: Figures/gadget-control.tikz
\begin{tikzpicture}[
  node distance=0.75cm and 1.5cm,
  ]

\node[estate] (1)  {$W$};
\node[estate,right= of 1] (2)  {$G$};
\node[estate,right= of 2] (3)  {$A$};
\node[estate,right= of 3] (4)  {$B$};

\path (1) edge[loop left] node[anode,auto] {$\texttt{wait}$}(1);
\path (1.north east) edge node[anode,auto] {$\texttt{go}$} (2.north west);
\path (2.north east) edge node[anode,auto] {$\Sigma_{\?G}$} (3.north west);
\path (3.south west) edge node[anode,auto] {$\texttt{next}$} (2.south east);
\path (4.south west) edge node[anode,below] {$\Sigma_{AC}$}(3.south east);
\path (3.north east) edge node[anode,above] {$\Sigma_{MC}$}(4.north west);
\path (2.south west) edge node[anode,auto] {$\texttt{win}$} (1.south east);

\end{tikzpicture}

%% file: Figures/gadget-counter-4h.tikz
\begin{tikzpicture}[
  node distance=1cm and 2.75cm,
  ]

\node[estate] at (0,0) (Z0)  {$\Bit{0}{0}$};
\node[left= of Z0, estate] (Z1)  {$\Bit{1}{0}$};
\node[left= of Z1, estate] (Z2)  {$\Bit{2}{0}$};
\node[left= of Z2, estate] (Z3)  {$\Bit{3}{0}$};

\node[above=  of Z0, estate] (O0)  {$\Bit{0}{1}$};
\node[left= of O0, estate] (O1)  {$\Bit{1}{1}$};
\node[left= of O1, estate] (O2)  {$\Bit{2}{1}$};
\node[left= of O2, estate] (O3)  {$\Bit{3}{1}$};

\path  (Z0.north west) edge node[anode,align=left,auto] {$\Dec{1}$,\\$\Dec{2}$,\\$\Dec{3}$ }(O0.south west);
\path  (O0.south east) edge node[anode,right] {$\Dec{0}$}         (Z0.north east);
\path  (Z1.north west) edge node[anode,align=left,left] {$\Dec{2}$,\\$\Dec{3}$}(O1.south west);
\path  (O1.south east) edge node[anode,auto] {$\Dec{1}$}         (Z1.north east);
\path  (Z1) edge[loop left] node[anode,auto] {$\Dec{0}$}(Z1);
\path  (O1) edge[loop left] node[anode,auto] {$\Dec{0}$}(O1);

\path  (O2.south east) edge node[anode,auto] {$\Dec{2}$}         (Z2.north east);
\path  (Z2) edge[loop left] node[anode,auto,align=left] {$\Dec{0}$,\\$\Dec{1}$}(Z2);
\path  (O2) edge[loop left] node[anode,auto,align=left] {$\Dec{0}$,\\$\Dec{1}$}(O2);
\path  (Z2.north west) edge node[anode,align=left,left] {$\Dec{3}$ }(O2.south west);

\path  (O3.south east) edge node[anode,auto] {$\Dec{3}$}         (Z3.north east);
\path  (Z3) edge[loop left] node[anode,auto,align=left] {$\Dec{0}$,\\$\Dec{1}$,\\$\Dec{2}$}(Z3);
\path  (O3) edge[loop left] node[anode,auto,align=left] {$\Dec{0}$,\\$\Dec{1}$,\\$\Dec{2}$}(O3);
\end{tikzpicture}

%% file: sec.examples.tex
Throughout this section, we construct an automaton relying on the 
the language introduced in \cref{sec:lower-bound-gadgets}.
Specifically,
a state $s$ is \emph{marked} in a configuration $w$ if at least one token occupies it: $\exists t. w(t)>0$.
Whenever an action $a$ takes state $s$ only back to itself we say that $s$ \emph{ignores} $a$.
There are states $\HEAVEN$ (the target) and $\HELL$ which ignore all actions.
For a given state $s$, an action $a$ is \emph{angelic} if it takes $s$ only to $\HEAVEN$,
and \emph{daemonic} if it takes $s$ to $\HELL$.
An action $a$ is \emph{safe} in a configuration if it is not daemonic for any marked state (in any gadget).

%% file: sec.example-butterfly.tex
\begin{figure}[ht]
    \centering
\input{Figures/butterfly.tikz}
\caption{An automaton where \Laetitia can synchronise any finite number of tokens but no deterministic, $k$-definable winning strategy exists.
}
\label{fig:01w}
\end{figure}

Consider now the example in \cref{fig:01w}. 
We will set it up so that
initially, all tokens are randomly distributed onto states $L$ and $R$, and \Laetitia must eventually place all of them on only one side of the graph. To do so,
\newcommand{\winl}{\texttt{win}_l}
\newcommand{\winr}{\texttt{win}_r}
\begin{itemize}
    \item add an initial action which moves tokens from an initial state to $L$ and $R$
        and is daemonic everywhere else;
    \item add a fresh winning action that is angelic for all red states and daemonic for all green states; and
    \item add a fresh winning action that is angelic for all green states and daemonic for all red states.
\end{itemize}

\subparagraph*{Idea}
Each round starts with all tokens on $L$ and $R$.
\Laetitia stepwise proposes a sequence of actions,
either in $l^+ l_1[l_2l_3]$ or $r^+ r_1[r_2r_3]$.
Notice that until one side is empty, these are the only safe sequences to play.
At the end of each round, all tokens (except possibly one) will switch sides. \Laetitia can choose to isolate one token and keep it on its side, thereby getting closer to her goal of moving everyone to one side.
The relevant decisions to make are

\begin{enumerate}
    \item whether to go left or right at the start of a round
    \item when to stop playing $l$ (or $r$, resp.)
\end{enumerate}

We say that two configurations $\gamma_1, \gamma_2$ are $K$-equivalent if for all states $s$, they either have the same number of tokens in $s$ or both have at least $K$ tokens in $s$. 
A strategy is $K$-deterministic if it ``counts up to $K$''. That is, it is deterministic and for all $K$-equivalent $\gamma_1, \gamma_2$, it plays the same action from both.

\begin{lemma}
The example is a positive instance of the \problem.
However, for every $K\ge 0$, every  $K$-deterministic strategy is losing.
\end{lemma}
\begin{proof}
    Referring to the relevant decisions above,
    a winning strategy is to 
    \begin{itemize}
    	\item always pick the side with the least number of tokens (one can also pick the side at random), and
    	
    	\item play $l$ (or $r$) until exactly one token is isolated.
    \end{itemize}

    Consider a $K$-deterministic strategy.
    Say we start with $2K+2$ tokens, and $K+1$ tokens end up on both $L$ and $R$.
    Our strategy must play $l$ or $r$, say it plays $l$.
    We have $K+1$ tokens in the state below $L$ and in the one above $R$.
    Playing $l_2$ or $l_3$ brings us back to the previous configuration (up to renaming tokens), and playing $l_1$ has no effect, thus it must play $l$.
    Once this is done, it must keep playing $l$ until exactly one token is in the lower state: if all tokens are in the upper state we have established that it plays $l$, and if there are $2$ or more tokens below, every move is losing except $l$ (playing $l_1$ risks tokens splitting between the two left states, after which we are stuck).
    Therefore, the strategy must wait until one token is in the lower state, then play $l_1$ and then $l_2$ or $l_3$.
    
    $L$ now contains $K+2$ tokens and $R$ contains $K$, which is $K$-equivalent to $K+1$ and $K+1$. By repeating the same reasoning, in the next round the strategy does the same thing, which yields a configuration with $K+1$ tokens in both $L$ and $R$, meaning we are back to the initial configuration.
\end{proof}

%% file: Figures/butterfly.tikz
\tikzset{every path/.append style={
  shorten <=2pt,
  shorten >=2pt,
  -{Triangle}, 
}}
\tikzset{rstate/.style={estate,
          fill=col1!10,
          draw=col1,
}}
\tikzset{lstate/.style={estate,
          fill=col2!10,
          draw=col2,
}}
\tikzset{lcol/.style={col2,
}}
\tikzset{rcol/.style={col1,
}}

\newcommand{\ACT}[2]{#1_{#2}}

\begin{tikzpicture}[
  node distance=1cm and 0.75cm,
  ]

\node[lstate] (L1)               {L};
\node[lstate,below= of L1] (L2)  {};
\node[lstate,below= of L2] (L3)  {};
\node[lstate,above left=0.25cm and 1cm of L3] (L4a)  {};
\node[lstate,below left=0.25cm and 1cm of L3] (L4b)  {};

\path (L1) edge node[anode,auto] {$l$}(L2);
\path (L2) edge[loop left] node[anode,auto] {$l,\ACT{l}{1}$}(L1);

\path (L2) edge[ {Triangle}-{Triangle} ] node[anode,auto,swap] {$l$} (L3);

\path (L3) edge node[anode,below,pos=0.75] {$\ACT{l}{1}$}(L4a);
\path (L3) edge node[anode,above,pos=0.75] {$\ACT{l}{1}$}(L4b);

\draw[lcol] (L4a) -- ([shift={(-0.50,0)}]L4a.center)  |- node[anode,right, yshift=-1cm] {$\ACT{l}{2}$}  (L1.south west);
\draw[lcol] (L4b) -- ([shift={(-0.75,0)}]L4b.center)  |- node[anode,right, yshift=-3cm] {$\ACT{l}{3}$} (L1.north west);

\node[lstate,above= of L1] (L0)  {};
\path (L1) edge node[anode,auto] {$r$}(L0);
\path (L0) edge[loop left] node[anode,auto] {$\ACT{r}{},\ACT{r}{1}$}(L0);

\node[rstate, right= 2cm of L1] (R1)               {R};
\node[rstate,below= of R1] (R2)  {};
\node[rstate,below= of R2] (R3)  {};
\node[rstate,above right=0.25cm and 1cm of R3] (R4a)  {};
\node[rstate,below right=0.25cm and 1cm of R3] (R4b)  {};

\path (R1) edge node[anode,auto] {$r$}(R2);
\path (R2) edge[loop right] node[anode,auto] {$r,\ACT{r}{1}$}(R1);

\path (R2) edge[ {Triangle}-{Triangle} ] node[anode,auto,swap] {$r$} (R3);

\path (R3) edge node[anode,below,pos=0.75] {$\ACT{r}{1}$}(R4a);
\path (R3) edge node[anode,above,pos=0.75] {$\ACT{r}{1}$}(R4b);

\draw[rcol] (R4a) -- ([shift={(0.50,0)}]R4a.center)  |- node[anode,left, yshift=-1cm] {$\ACT{r}{2}$} (R1.south east);
\draw[rcol] (R4b) -- ([shift={(0.75,0)}]R4b.center)  |- node[anode,left, yshift=-3cm] {$\ACT{r}{3}$} (R1.north east);

\node[rstate,above= of R1] (R0)  {};
\path (R1) edge node[anode,auto] {$l$}(R0);
\path (R0) edge[loop right] node[anode,auto] {$\ACT{l}{},\ACT{l}{1}$}(R0);

\path[lcol] (R0) edge node[anode,above,pos=0.2,rotate=25] {$\ACT{l}{2},\ACT{l}{3}$}(L1);
\path[rcol] (L0) edge node[anode,above,pos=0.2,rotate=-25 ] {$\ACT{r}{2},\ACT{r}{3}$}(R1);
\path[rcol] (R2) edge node[anode,below,pos=0.2,rotate=-25] {$\ACT{r}{2},\ACT{r}{3}$}(L1);
\path[lcol] (L2) edge node[anode,below,pos=0.2,rotate=25 ] {$\ACT{l}{2},\ACT{l}{3}$}(R1);

\end{tikzpicture}

%% file: sec.example-chain.tex
In the previous section we saw that there are positive instances of the \problem{} on which $K$-deterministic strategies do not suffice for any $K$.
By contrast, a central result of this work, \cref{thm:cortosconj}, implies that a randomised variant of these strategies suffices. Namely, for a positive instance of the \problem{} with $K$ many states, there exists a \emph{$K$-definable strategy}: one that gives the same distribution over actions from all configurations that are $K$-equivalent.
Here, we show that
for every $K$, one can construct a positive instance of the population control problem{} in which all $<K$-definable strategies are not winning.
Our construction to show this uses the bottleneck gadgets of Figure~\ref{fig:bottlenecks}, in a way that prevents winning strategies that are $1$-definable.

We present gadgets that we call \emph{bottlenecks} with a fixed capacity $k$.
These are designed so that \Laetitia can ``pass through'' up to $k$ tokens, but not more, meaning that each such gadget is a negative instance of the random population control problem.

\begin{figure}[ht]
\begin{subfigure}[t]{0.5\textwidth}
\input{Figures/pipe-1.tikz}
\caption{A bottleneck with capacity $1$.}
\end{subfigure}
    \hfill
\begin{subfigure}[t]{0.5\textwidth}

\hfill
\input{Figures/pipe-2.tikz}
\caption{A bottleneck with capacity $2k$.}
\end{subfigure}
\caption{Depiction of bottleneck gadgets using actions in $\Sigma=\{a,b,c,d,e\}$. Each starts in a state $s$ and action $b$, and ends in a target state $t$ with action $e$.}\label{fig:bottlenecks}
\end{figure}

\subparagraph*{Bottlenecks}
Consider first the automaton in Figure~\ref{fig:bottlenecks}(a).
One readily sees that it is a negative instance of the \problem{}. \Laetitia can however safely move one token from start $s$ to the target $t$.

Figure~\ref{fig:bottlenecks}(b) shows how to construct bottlenecks of arbitrary finite capacity.
To do this, we replace the red (blue) edge by a bottleneck of capacity $k$
so that 1) all its actions are distinct to, and ignored outside of, that red (blue) gadget.
2) the global actions $a,b,e$ are daemonic for all but the last state of the red (blue) gadget.
For instance, for $k=2$, both red and blue edges are replaced by disjoint copies of the capacity-1 bottleneck in Figure~\ref{fig:bottlenecks}(a).

\begin{lemma}
    \label{lem:bottleneck-K}
    \Laetitia can move $2k$ tokens from $s$ to $t$
    in a bottleneck of capacity $2k$, but not with a $(k-1)$-definable strategy.
\end{lemma}
\begin{proof}
She can do so by playing $b$ and then repeating action $a$ until exactly $k$ tokens reside on states $l$ and $r$, respectively. She can then use the actions in the red (blue) gadgets to move these tokens to the central state $m$. Only then is it safe to play the action $e$ and move them all to the target.

However, with a $(k-1)$-definable strategy, \Laetitia cannot distinguish configurations with $k$ tokens on each side from the ones with $k+1$ on a side and $k-1$ on the other.
Therefore, she either 
risks entering a $k$-bottleneck with $k+1$ tokens
or continues to play the action $a$ indefinitely and keeps all tokens in states $l$ or $r$. Both cases are clearly losing.
\end{proof}


\subparagraph*{A length-$K$ chain of bottlenecks with capacity $1$}
We join $K$ copies of the gadget depicted in Figure~\ref{fig:bad-ex-chain} (without the red arrows for now)
so that for all $1\le n \le K$, the right-most state of the $n$-th copy is the left-most state of the $(n+1)$-th copy.
The initial state is $q_1$, the start of the first copy at level one, and the ultimate target state is $q_{K+1}$, the last state in the $K$-th copy.
Notice that each member of the chain is an instance of Figure~\ref{fig:force-one} followed by a bottleneck of capacity one (Figure~\ref{fig:bottlenecks}(a)), and operates on its own alphabet $\Sigma_n$ of actions.
We further impose the following constraints.
\begin{enumerate}
    \item 
       In all states with index $n$, all actions of index $<n$ are ignored
        and all actions of index $>n$ are daemonic.
    \item In states $q_n$, action $b_n$ 
        lead back to $q_1$ (indicated by the blue arrow).
\end{enumerate}

\begin{figure}[ht]
	\centering
	\input{Figures/force-one-pipe.tikz}
	\caption{A member of the chain}\label{fig:bad-ex-chain}
\end{figure}


Notice that, when moving a large number of tokens from $q_0$ to $q_{K}$,
\Laetitia will reach at least one configuration with exactly one token in $\set{l_i,r_i}$ for $1\le i \le K$.
Indeed, to move a token to the target $q_{K+1}$, she needs to play action $b_{K}$ at some point, which is safe only if there is at most one token in $s_{K}$.
To make tokens reach $q_{K+1}$, she must eventually play $b_{K}$ with exactly one token in $s_{K}$.
Since $b_{K}$ is daemonic for all states of index $<K$, before playing $b_{K}$ for the first time, all other tokens must be in $q_{K}$, and are thus sent to $q_1$ when $b_{K}$ is played.

Now, in order to move the isolated token to $q_{K+1}$, \Laetitia must bring all tokens to $q_{K}$ once again, while the isolated token is waiting in $l_{K}$ or $r_{K}$.
By repeating these steps, we conclude that there must be a point at which exactly one token is in $\set{l_i,r_i}$ for $1\le i \le K$

\subparagraph*{A Leaky Chain}
Let $k \in \NN$. 
We construct an instance where \Laetitia has a winning strategy for any number of tokens $N$, but cannot win with a $k-1$-definable strategy.
We describe the full system now, which has two disjoint parts. 

The first part of the construction is a bottleneck of capacity $K$ with $K \ge 2k$.
The second is the chain of $K$ many bottlenecks of capacity $1$ as described just above.

This bottleneck of capacity $K$ is disjoint from, 
and its actions are ignored throughout, the chain of bottlenecks.
The purpose of the capacity-$K$ bottleneck is to recover up to $K$ tokens and move them back to the initial state $q_1$ in the chain.
The final part in our construction is to add edges from all states $l_n$ and $r_n$ to the initial state of the capacity-$K$ bottleneck. This happens on action $a_1$ and is indicated by the red arrows in Figure~\ref{fig:bad-ex-chain}.

We also add a reset action that sends all states of the chain to the initial one, as well as the final state of the $K$-bottleneck, but is daemonic for all other states of the $K$-bottleneck.

We summarize in the following lemma.

\begin{lemma}
    The system constructed above is a positive instance of the population control problem: For every $N\in\NN$, \Laetitia has a strategy to almost surely move $N$ tokens from $q_1$ to $q_{K+1}$.
    For $N\ge k$, every $(k-1)$-definable strategy is losing.
\end{lemma}
\begin{proof}
    \Laetitia can win by applying the following strategy while all tokens are in the chain of bottlenecks: push one token into a chain $i$, have it wait in state $l_i$ or $r_i$ while all tokens on lower-indexed states move to $q_i$, then move it out. This proceeds recursively and ultimately moves all tokens to $q_{K+1}$.
    
    What may happen is that when playing an action some tokens from states $l_i, r_i$ fall in the $K$-bottleneck.
    There may be up to $K$ tokens falling simultaneously, as explained when describing the chain, but not more since our strategy allows at most one token at a time in each $\set{l_i, r_i}$.
    \Laetitia's only choice is then to get those tokens to the final state of the $K$-bottleneck, play the reset action and start over.
	This is a winning strategy.
	However, if \Laetitia uses a $(k-1)$-definable strategy then she cannot win almost surely when $K \ge 2k+2$ tokens enter the $K$-bottleneck, by \cref{lem:bottleneck-K}.
	Furthermore, she cannot win without a positive probability to arrive in this situation, as explained when describing the chain, hence there is no $(k-1)$-definable winning strategy.
\end{proof}

%% file: Figures/pipe-1.tikz
\begin{tikzpicture}[
		node distance=0.5cm and 1.25cm,
	]
	\path[use as bounding box] (-1,-2) rectangle (5.5,2);

	\node[state, initial] (A)   {$s$};
	\node[astate,above right= of A] (B)  {~~~};
	\node[astate,below right= of A] (C) {~~~};
	\node[state,above right= of C] (D)  {~~~};
	\node[state,accepting, right= of D] (E)  {$t$};

	\draw (A) edge node[] {$b$} (B);
	\draw (A) edge node[swap] {$b$} (C);
	
	\draw (B) edge node[] {$d$} (D);
	\draw (C) edge node[swap] {$c$} (D);
	\draw (D) edge node[] {$e$} (E);
	
	\draw[draw=none,loop above] (B) edge[draw=none] (B);
	\draw[draw=none,loop below] (C) edge[draw=none] (C);
\end{tikzpicture}

%% file: Figures/pipe-2.tikz
\begin{tikzpicture}[
		node distance=0.5cm and 1.25cm,
	]

	\path[use as bounding box] (-1,-2) rectangle (5.5,2);

	\node[state, initial] (A)   {$s$};
	\node[astate,above right= of A] (B)  {$l$};
	\node[astate,below right= of A] (C) {$r$};
	\node[state,above right= of C] (D)  {$m$};
	\node[state,accepting, right= of D] (E)  {$t$};

	\draw[->](A) edge node {$b$} (B);
	\draw[->](A) edge node[swap] {$b$} (C);

	\draw[->](B) edge[bend right=10] node[swap] {$a$} (C);
	\draw[->](C) edge[bend right=10] node[swap] {$a$} (B);
	\draw (B) edge[loop above] node {$a$} (B);
	\draw (C) edge[loop below] node {$a$} (C);

	\draw (B) edge[col1] node[] {$\le k$} (D);
	\draw (C) edge[col2] node[swap] {$\le k$} (D);
	\draw (D) edge node[swap] {$e$} (E);

\end{tikzpicture}

%% file: Figures/force-one-pipe.tikz
\begin{tikzpicture}[
		node distance=0.5cm and 1.5cm,
	]

	\node[estate, initial] (Q) {$q_n$};
	\node[state, right= of Q] (A)  {$s_n$};
	\node[astate,above right= of A] (B)  {$l_n$};
	\node[astate,below right= of A] (C) {$r_n$};
	\node[estate,above right= of C] (D)  {$q_{n+1}$};

	\draw (Q) edge[loop above] node {$\Sigma_n$} (Q);
	\draw[] (Q) edge[bend right=20] node[swap] {$a_n$} (A);
	\draw[] (A) edge[bend right=20] node[swap] {$a_n$} (Q);

	\draw (A) edge node[] {$b_n$} (B);
	\draw (A) edge node[swap] {$b_n$} (C);
	\draw (B) edge node[] {$d_n$} (D);
	\draw (C) edge node[swap] {$c_n$} (D);
	
	\draw (Q) edge[col2] node[left] {$b_{n}$} ++(0,-1.5);
	\draw (B) edge[col1, bend left] node[yshift=-5mm] {$a_{1}$} ++(0,-3.3);
	\draw (C) edge[col1, bend left=20] node[left] {} ++(-0.3,-1.3);
\end{tikzpicture}

%% file: sec.example-2exp-cutoff.tex
The \intro{cut-off} of an MDP is the minimal number of tokens against which Controller does not have an almost-sure winning strategy, 
In this section we show that there are MDPs whose "cut-off" is finite but doubly-exponential in the number of states. 
We reuse a construction from~\cite{BertrandDGG17}. We simply replace the adversary with randomisation.
Let us summarise it here.

Let $N \in \NN$. We build an MDP $\?M_N$ as follows: 
We have an initial state $i$ and a final state $f$, plus a sink state $\mathtt{sink}$.

\newcommand{\start}{\mathtt{start}}
\newcommand{\sink}{\mathtt{sink}}
\newcommand{\leftpart}{\mathtt{left}}
\newcommand{\rightpart}{\mathtt{right}}
\renewcommand{\dec}{\mathtt{dec}}
\newcommand{\stopcount}{\mathtt{stop}}

\begin{itemize}
	\item First of all we have a \emph{splitting gadget}, as shown in Figure~\ref{fig:splitting}. 
	This gadget forces Controller to play a word of $((\Sigma_N\setminus\set{l,r})\set{l,r})^*$ until all three states are empty. We call a sequence of two actions of $(\Sigma_N\setminus\set{l,r})\set{l,r}$ a \emph{round}.
	Say there are $2^K$ tokens in $\mathtt{s}$ at the start, then there is a positive probability that they distribute evenly between $\leftpart$ and $\rightpart$ whenever they split from $\mathtt{s}$.
	Therefore, for any strategy of Controller, there is a positive probability that after $K$ rounds there is still at least one token in $\mathtt{s}$.
	
	\item Second, we have a \emph{counter gadget}, already used in Section~\ref{sec:app-lower-bound}, as shown in Figure~\ref{fig:counter}. 
	This one consists of $2N$ states, representing $N$ bits, $\set{(i:0), (i:1) \mid 0\leq i \leq N-1}$.
	For all $0 \leq j \leq N-1$ there is an action $\dec_j$ that fixes all states $(k:0), (k:1)$ with $k>j$, sends $(j:1)$ to $(j:0)$, and sends each $(k:0)$ with $k<j$ to $(k:1)$. The tokens in states $(k:1)$  with $k<j$ and $(j:0)$ are all sent to $\sink$.
	Action $\stopcount$ sends all tokens in $(i:0)_{1\leq i \leq N}$ to $f$ and all tokens in $(i:1)_{1\leq i \leq N}$ to $\sink$.
	Actions $l,r$ leave the tokens in those states idle.
	This gadget implements a binary counter over $N$ bits. Let $\mathtt{Dec} = \set{\dec_j \mid 1 \leq j \leq N}$.
	If we start with tokens in states $(i:0)_{1\leq i \leq N}$ and not in states $(i:1)_{1\leq i \leq N}$, then at all times Controller has only one action from $\mathtt{Dec}$ available, which increments the counter. 
	Once every bit is $1$, the only possible action apart from $\set{l,r}$ is $\stopcount$.
	She is forced to play a sequence of actions of $\mathtt{Dec}^{2^N-1}\stopcount$, interleaved with some $l$ and $r$.
	
	\item We add an action $\start$ which sends tokens from $i$ to $\set{\mathtt{s}} \cup \set{0_i \mid 1 \leq i \leq N}$, and leaves all other tokens idle. All other actions map $i$ to $\sink$.
\end{itemize}

\begin{figure}
	\begin{center}
	\begin{tikzpicture}[
		node distance=0.5cm and 1.5cm,
		]
		
		\node[astate, initial above] (A) {$s$};
		\node[astate,below left= of A] (B)  {\footnotesize$~\mathtt{left}~$};
		\node[astate,below right= of A] (C) {\footnotesize$\mathtt{right}$};
		\node[state,below right= of B] (D) {$f$};

		\draw (A) edge[bend right=20] node[above left=1mm] {$\Sigma\setminus\set{l,r,\mathtt{stop}}$} (B);
		\draw (A) edge[bend left= 20] node[above right=1mm] {$\Sigma\setminus\set{l,r,\mathtt{stop}}$} (C);
		\draw (B) edge[bend right= 20] node[below] {$l$} (D);
		\draw (C) edge[bend left= 20] node[below] {$r$} (D);
		\draw (B) edge[bend right= 20] node[swap] {$l$} (A);
		\draw (C) edge[bend left= 20] node[] {$r$} (A);
		\draw (D) edge[loop above] node[] {$\Sigma$} (D);
	\end{tikzpicture}
	\end{center}
	\caption{The splitting gadget. The action $\mathtt{stop}$ sends all tokens in $s$, $\mathtt{left}$ and $\mathtt{right}$ to a sink state.}
	\label{fig:splitting}
\end{figure}

If there are more than $2^{2^N} + N$ tokens, then, with positive probability, when playing $\start$ one goes to each of the $(i:0)_{1\leq i \leq N}$ and the rest to $\mathtt{s}$. 
Then, as observed above, with positive probability the first gadget forces Controller to play at least $2^N$ actions from $\Sigma_N\setminus \set{l,r}$ before being able to play $\stopcount$. However, the second gadget forces Controller to play $\stopcount$ after playing $2^N$ actions from $\Sigma\setminus \set{l,r, \stopcount}$. Hence Controller cannot win almost surely.

If there are fewer than $2^{2^N}$ tokens, then Controller can play as follows: Play $\start$, then alternate between playing the available $\dec_j$ (it may be that several actions are available because some $(i:0)$ did not receive any tokens in the first step; in that case, choose the minimal $j$) and playing either $l$ or $r$ (whichever sends the most tokens to $f$).
After $2^N$ steps, all tokens in the counter gadget are in the states $(i:1)_{1 \leq i \leq N}$.
Since at every round we send at least half of the tokens in the splitting gadget to $f$, after $2^N$ rounds the gadget is empty. Controller can then play $\stopcount$ safely, and send all tokens in the counter gadget to $f$.

We obtain an MDP with $2N + 6$ states whose "cut-off" is between $2^{2^N}$ and $2^{2^N}+N$.

%% file: biblio.cleaned.bib
@STRING{procams	= "Proceedings of the American Mathematical Society" }

@Article{	  ford1956maximal,
  title		= {Maximal flow through a network},
  author	= {Lester Randolph Ford and Delbert Ray Fulkerson},
  journal	= {Canadian journal of Mathematics},
  volume	= {8},
  pages		= {399--404},
  year		= {1956},
  publisher	= {Cambridge University Press}
}

@Article{	  SoloveichikCWB08,
  author	= {David Soloveichik and Matthew Cook and Erik Winfree and
		  Jehoshua Bruck},
  title		= {Computation with finite stochastic chemical reaction
		  networks},
  journal	= {Nat. Comput.},
  volume	= {7},
  number	= {4},
  pages		= {615--633},
  year		= {2008},
  doi		= {10.1007/S11047-008-9067-Y},
  bibsource	= {dblp computer science bibliography, https://dblp.org}
}

@Article{GimbertMT25,
	title		= {Optimal Sequential Flows},
	author	= {Gimbert, Hugo and Mascle, Corto and Totzke, Patrick},
	journal	= {arXiv preprint arXiv:2511.13806},
	year		= {2025},
	url          = {https://doi.org/10.48550/arXiv.2511.13806}
}

@Article{	  dickson1913finiteness,
  title		= {Finiteness of the odd perfect and primitive abundant
		  numbers with n distinct prime factors},
  author	= {Dickson, Leonard Eugene},
  journal	= {American Journal of Mathematics},
  volume	= {35},
  number	= {4},
  pages		= {413--422},
  year		= {1913},
  publisher	= {JSTOR}
}

@InProceedings{	  JLS2007,
  author	= {Marcin Jurdzinski and Fran{\c{c}}ois Laroussinie and
		  Jeremy Sproston},
  title		= {Model Checking Probabilistic Timed Automata with One or
		  Two Clocks},
  booktitle	= {Tools and Algorithms for the Construction and Analysis of
		  Systems, 13th International Conference, {TACAS} 2007, Held
		  as Part of the Joint European Conferences on Theory and
		  Practice of Software, {ETAPS} 2007 Braga, Portugal, March
		  24 - April 1, 2007, Proceedings},
  series	= {Lecture Notes in Computer Science},
  volume	= {4424},
  pages		= {170--184},
  publisher	= {Springer},
  year		= {2007},
  doi		= {10.1007/978-3-540-71209-1_15},
  bibsource	= {dblp computer science bibliography, https://dblp.org}
}

@Article{	  DBLP:journals/corr/abs-1909-06420,
  author	= {Corto Mascle and Mahsa Shirmohammadi and Patrick Totzke},
  title		= {Controlling a Random Population is \exptime-hard},
  journal	= {CoRR},
  volume	= {abs/1909.06420},
  year		= {2019},
  eprinttype	= {arXiv},
  eprint	= {1909.06420},
  bibsource	= {dblp computer science bibliography, https://dblp.org}
}

@Article{	  ColcombetFO21,
  author	= {Thomas Colcombet and Nathana{\"{e}}l Fijalkow and Pierre
		  Ohlmann},
  title		= {Controlling a random population},
  journal	= {Log. Methods Comput. Sci.},
  volume	= {17},
  number	= {4},
  year		= {2021},
  doi		= {10.46298/LMCS-17(4:12)2021},
  bibsource	= {dblp computer science bibliography, https://dblp.org}
}

@InProceedings{	  SistlaG87,
  author	= {A. Prasad Sistla and Steven M. German},
  title		= {Reasoning with Many Processes},
  booktitle	= {Proceedings of the Symposium on Logic in Computer Science
		  {(LICS} '87), Ithaca, New York, USA, June 22-25, 1987},
  pages		= {138--152},
  publisher	= {{IEEE} Computer Society},
  year		= {1987},
  bibsource	= {dblp computer science bibliography, https://dblp.org}
}

@InProceedings{	  EmersonN95,
  author	= {E. Allen Emerson and Kedar S. Namjoshi},
  title		= {Reasoning about Rings},
  booktitle	= {Conference Record of POPL'95: 22nd {ACM} {SIGPLAN-SIGACT}
		  Symposium on Principles of Programming Languages, San
		  Francisco, California, USA, January 23-25, 1995},
  pages		= {85--94},
  publisher	= {{ACM} Press},
  year		= {1995},
  doi		= {10.1145/199448.199468},
  bibsource	= {dblp computer science bibliography, https://dblp.org}
}

@Article{	  DBLP:journals/ita/Kirsten05,
  author	= {Daniel Kirsten},
  title		= {Distance desert automata and the star height problem},
  journal	= {{RAIRO} Theor. Informatics Appl.},
  volume	= {39},
  number	= {3},
  pages		= {455--509},
  year		= {2005},
  doi		= {10.1051/ITA:2005027},
  bibsource	= {dblp computer science bibliography, https://dblp.org}
}

@Article{	  lmcs:5647,
  title		= {{Controlling a population}},
  author	= {Nathalie Bertrand and Miheer Dewaskar and Blaise Genest
		  and Hugo Gimbert and Adwait Amit Godbole},
  doi		= {10.23638/LMCS-15(3:6)2019},
  journal	= {{Logical Methods in Computer Science}},
  volume	= {{Volume 15, Issue 3}},
  year		= {2019},
  month		= jul
}

@Article{	  HazardIK24,
  author	= {Emile Hazard and Olivier Idir and Denis Kuperberg},
  title		= {Explorable Parity Automata},
  journal	= {CoRR},
  volume	= {abs/2410.23187},
  year		= {2024},
  doi		= {10.48550/ARXIV.2410.23187},
  eprinttype	= {arXiv},
  eprint	= {2410.23187},
  bibsource	= {dblp computer science bibliography, https://dblp.org}
}

@InProceedings{	  AkshayGV18,
  author	= {S. Akshay and Blaise Genest and Nikhil Vyas},
  title		= {Distribution-based objectives for Markov Decision
		  Processes},
  booktitle	= {Proceedings of the 33rd Annual {ACM/IEEE} Symposium on
		  Logic in Computer Science, {LICS} 2018, Oxford, UK, July
		  09-12, 2018},
  pages		= {36--45},
  publisher	= {{ACM}},
  year		= {2018},
  doi		= {10.1145/3209108.3209185},
  bibsource	= {dblp computer science bibliography, https://dblp.org}
}

@Article{	  UhlendorfMDCFBHB15,
  title		= {In silico control of biomolecular processes},
  author	= {Uhlendorf, Jannis and Miermont, Agn{\`e}s and Delaveau,
		  Thierry and Charvin, Gilles and Fages, Fran{\c{c}}ois and
		  Bottani, Samuel and Hersen, Pascal and Batt, Gregory},
  journal	= {Computational Methods in Synthetic Biology},
  pages		= {277--285},
  year		= {2015},
  publisher	= {Springer}
}

@Article{	  AngluinADFP06,
  author	= {Dana Angluin and James Aspnes and Zo{\"{e}} Diamadi and
		  Michael J. Fischer and Ren{\'{e}} Peralta},
  title		= {Computation in networks of passively mobile finite-state
		  sensors},
  journal	= {Distributed Comput.},
  volume	= {18},
  number	= {4},
  pages		= {235--253},
  year		= {2006},
  doi		= {10.1007/S00446-005-0138-3},
  bibsource	= {dblp computer science bibliography, https://dblp.org}
}

@Article{	  Doyen23,
  author	= {Laurent Doyen},
  title		= {Stochastic Games with Synchronization Objectives},
  journal	= {J. {ACM}},
  volume	= {70},
  number	= {3},
  pages		= {23:1--23:35},
  year		= {2023},
  doi		= {10.1145/3588866},
  bibsource	= {dblp computer science bibliography, https://dblp.org}
}

@InProceedings{	  BertrandDGG17,
  author	= {Nathalie Bertrand and Miheer Dewaskar and Blaise Genest
		  and Hugo Gimbert},
  title		= {Controlling a Population},
  booktitle	= {28th International Conference on Concurrency Theory,
		  {CONCUR} 2017, September 5-8, 2017, Berlin, Germany},
  series	= {LIPIcs},
  volume	= {85},
  pages		= {12:1--12:16},
  publisher	= {Schloss Dagstuhl - Leibniz-Zentrum f{\"{u}}r Informatik},
  year		= {2017},
  doi		= {10.4230/LIPICS.CONCUR.2017.12},
  bibsource	= {dblp computer science bibliography, https://dblp.org}
}

@Book{		  rozenberg1997handbook,
  title		= {Handbook of Formal Languages},
  author	= {Rozenberg, G. and Salomaa, A.},
  number	= {v. 1-3},
  isbn		= {9783540614869},
  series	= {Handbook of formal languages / G. Rozenberg; A. Salomaa},
  year		= {1997},
  publisher	= {Springer}
}

@Book{		  Puterman:book,
  author	= {Puterman, Martin L.},
  title		= {Markov Decision Processes: Discrete Stochastic Dynamic
		  Programming},
  year		= {1994},
  isbn		= {0471619779},
  edition	= {1st},
  publisher	= {John Wiley \& Sons, Inc.},
  doi		= {10.1002/9780470316887}
}

@Article{	  Ornstein:AMS1969,
  issn		= {00029939, 10886826},
  doi		= {10.2307/2035700},
  author	= {Donald Ornstein},
  journal	= procams,
  number	= {2},
  pages		= {563--569},
  publisher	= {American Mathematical Society},
  title		= {On the Existence of Stationary Optimal Strategies},
  volume	= {20},
  year		= {1969}
}

@PhDThesis{	  schmitz:tel-01663266,
  title		= {{Algorithmic Complexity of Well-Quasi-Orders}},
  author	= {Schmitz, Sylvain},
  school	= {{{\'E}cole normale sup{\'e}rieure Paris-Saclay}},
  year		= {2017},
  month		= nov,
  type		= {Accreditation to supervise research},
  hal_id	= {tel-01663266},
  hal_version	= {v1}
}

@TechReport{	  Petri62,
  author	= {Petri, Carl Adam},
  institution	= {Rheinisch-Westf{\"a}lisches Institut f{\"u}r
		  Instrumentelle Mathematik an der Universit{\"a}t Bonn},
  publisher	= {Rheinisch-Westf{\"a}lisches Institut f{\"u}r
		  Instrumentelle Mathematik an der Universit{\"a}t Bonn},
  number	= {2},
  title		= {{Kommunikation mit Automaten}},
  type		= {Dissertation, Schriften des IIM},
  year		= 1962
}
